\documentclass[11pt]{article}
\usepackage{format}
\usepackage[english]{babel}
\usepackage[utf8]{inputenc}
\usepackage[normalem]{ulem}
\usepackage[makeroom]{cancel}
\usepackage{natbib}
 
\usepackage{preamble}

\newcommand{\wve}[1]{\textcolor{dpurple}{Wouter: #1}}

\newcommand{\R}{\mathbb{R}}
\newcommand{\pr}{\mathbb{P}}

\newcommand{\Z}{\mathbb{Z}}

\newtheoremstyle{quoted}     
  {}{}                       
  {\itshape}                 
  {}                         
  {\bfseries}                
  {}                         
  { }                        
  {``#1 #2''\thmnote{ (#3).}} 

\theoremstyle{quoted}

\title{\vspace{-1in}\huge
A Computational Method for Solving the Stochastic Joint Replenishment Problem in High Dimensions}

\author{
\normalsize Barış Ata \vspace{-1mm}\\ 
\scriptsize Booth School of Business, University of Chicago, \href{Baris.Ata@chicagobooth.edu}{Baris.Ata@chicagobooth.edu}   \vspace{.5mm} \\ 
\normalsize Wouter J.E.C. van Eekelen \vspace{-1mm}\\ 
\scriptsize Booth School of Business, University of Chicago, \href{Wouter.vanEekelen@chicagobooth.edu}{Wouter.vanEekelen@chicagobooth.edu}   \vspace{.5mm} \\ 
\normalsize Yuan Zhong \vspace{-1mm}\\ 
\scriptsize Booth School of Business, University of Chicago, \href{Yuan.Zhong@chicagobooth.edu}{Yuan.Zhong@chicagobooth.edu}   \vspace{.5mm} \\ 
} 

\date{\small\today}

\linedabstractkw{
We consider a discrete-time formulation for a class of high-dimensional stochastic joint replenishment problems. 
First, we approximate the problem by a continuous-time impulse control problem. 
Exploiting connections among the impulse control problem, backward stochastic differential equations (BSDEs) with jumps, 
and the stochastic target problem, we develop a novel, simulation-based computational method that relies on deep neural networks
to solve the impulse control problem. 
Based on that solution, we propose an implementable inventory control policy for the original (discrete-time) stochastic joint replenishment problem, 
and test it against the best available benchmarks in a series of test problems.  
For the problems studied thus far, our method matches or beats the best benchmark we could find, 
and it is computationally feasible up to at least 50 dimensions---that is, 50 stock-keeping units (SKUs).
}{Inventory management, joint replenishment problem, economies of scale, impulse control, high-dimensional PDEs} 

\begin{document}


\maketitle






   

\section{Introduction}\label{sec:intro}

The joint replenishment problem (JRP) concerns the management of multi-item inventory systems,
in which a fixed replenishment cost is incurred jointly by items procured together. 
A JRP is called stochastic (SJRP) when demands are random. 
A classical problem in operations management, the JRP has been extensively studied over several decades with wide-ranging applications; 
see, for example, the literature surveys in \cite{goyal1989joint}, \cite{aksoy1988multi}, \cite{khouja2008review} and \cite{creemers2022joint}.
In particular, the SJRP is  relevant to sectors involving large-scale operations 
with highly variable demand and vast product assortments, such as modern retail, 
where joint replenishment can amortize fixed costs such as transportation, and yield 
substantial economies of scale.

While coordinated replenishment offers attractive cost-saving opportunities, 
effective deployment in practice is challenging, especially when the number of distinct items, i.e., the problem dimension, is large. 
A fundamental reason is that item-level inventory decisions become interdependent under nonzero fixed replenishment costs. 
This interdependence creates two difficulties: (i) the computational effort required to solve the associated Markov decision process (MDP) 
scales with the sizes of the inventory state and decision spaces, 
both of which grow exponentially with dimension, leading to the ``curse of dimensionality''; 
and (ii) fixed costs induce large multi-item orders, producing jumps in the inventory state and amplifying dimensionality effects.
As a result, there has been limited progress in the literature on obtaining optimal replenishment policies in the multidimensional setting; 
see, for example, Section~6 of \cite{perera2023survey2} for a review of related literature.
Instead, a substantial body of literature has focused on designing effective heuristics for the SJRP; see, for example, \cite{khouja2008review}.
These heuristics typically propose intuitively appealing classes of {parameterized inventory control policies, 
reducing the control problem to the optimization of the policy parameters.}

In this paper, we revisit the classical SJRP \citep{goyal1989joint}, which is formulated in discrete time,
and develop an effective computational method for the high-dimensional setting.
Toward this end, we consider a closely related continuous-time impulse control problem, 
in which instantaneous jumps in the system state can be {effected} at a fixed cost, 
and whose solution naturally leads to implementable policies for the original inventory control problem. 
Traditional numerical approaches to impulse control problems use grid-based methods, 
such as finite difference and finite element schemes, 
to solve the associated Hamilton-Jacobi-Bellman (HJB) partial differential equations (PDEs), 
but they also suffer from the curse of dimensionality. 
{Compared to the discrete-time SJRP, the continuous-time impulse control formulation has a key advantage: 
It admits an equivalent probabilistic representation of the associated value function \citep{kharroubi2010backward}. 
This representation allows us to employ the deep backward stochastic differential equation (BSDE) method 
developed in the seminal work of \cite{han2018solving} for tackling high-dimensional problems.
Since \cite{han2018solving}, a large literature has used deep neural networks to solve high-dimensional PDEs following the deep BSDE framework.}
We build on and contribute to this literature by developing a computational
method for solving the high-dimensional impulse control problem. More importantly, we leverage that
solution to propose effective control policies for the classical SJRP.

To put our method in perspective, we briefly review the deep BSDE framework, 
highlight our main challenges, and discuss the innovations we introduce to address them.
Deep BSDE is a simulation-based computational approach for solving high-dimensional PDEs using deep neural networks.
It starts from a probabilistic (BSDE) representation of the target function, 
typically a stochastic equality satisfied almost surely {(a.s.)} along sample paths of a {\em reference process}, 
which is a simulatable random process chosen by the algorithm designer.
Based on this representation, one constructs a loss function and trains neural network approximations 
for the target function (and its gradient) by minimizing the loss, using samples generated by simulating the reference process.
When the representation is a stochastic equality, the loss can be taken as the standard mean squared error between the left- and right-hand sides, 
with the function and gradient replaced by their neural network approximations, respectively.
Under appropriate choices of the loss and the reference process, the success of the deep BSDE method then hinges primarily on extensive engineering efforts in hyperparameter tuning and neural network training.

For our impulse control problem, two immediate challenges arise in designing the loss and the reference process within the deep BSDE framework. 
First, while an equivalent probabilistic representation exists, 
it takes the form of a family of almost-sure stochastic inequalities rather than an equality \citep{kharroubi2010backward}. 
To address this more complex case, 
we design a novel, nonstandard loss that 
{dualizes inequality constraints by incorporating 
a loss term for violations of the stochastic inequality and a tunable penalty parameter.}
The penalty parameter introduces considerable complications for neural network training, 
as its tuning schedule must be carefully coupled with the learning rate schedule to ensure training stability. 
Second, since the system state can be instantaneously displaced, 
deriving an effective control policy from the trained neural networks 
demands more than accurate local approximations of the value function and its gradient, 
as is common in prior literature (see, e.g., \cite{han2018solving,ata2023drift,ata2023dynamic}); 
rather, the value function and its gradient must be well approximated across an expansive region of the state space. 
This, in turn, requires the reference process to explore the state space broadly. 
Indiscriminate exploration, however, yields irrelevant samples that hinder training. 
Thus, the reference process must strike an effective balance between exploration and exploitation; 
although this trade-off has been considered in prior work, 
it is substantially more prominent and nuanced in the context of impulse control. Toward this end, 
we design several sophisticated reference processes with carefully designed state-dependent jumps
to ensure that relevant samples are generated for efficient training.

As mentioned earlier, our neural network solution to the impulse control problem naturally yields 
implementable inventory control policies for the original (discrete-time) formulation of the SJRP. 
We demonstrate the effectiveness of these policies through a series of test problems in the original SJRP setting.
For low-dimensional test problems, we compare our policies primarily to the optimal ones derived by solving the MDPs associated with the original inventory control problem. 
{We also compare our policy with the best available heuristics. Across settings, the performance of the proposed policy remains consistently within 1\% of the optimal policy. It also matches or outperforms the best-performing heuristic. Moreover, the policy itself closely tracks the optimal policy (see Figures \ref{fig:mmm_case} and \ref{fig:comparison_pane}). For the high-dimensional test problems, however, solving the underlying MDP is computationally intractable. Thus,} we benchmark our approach against the best available heuristic policies. For the high-dimensional test problems considered, our policy matches or beats the best {performing heuristic}. 

The rest of the paper is structured as follows. Section~\ref{sec:litreview} provides an overview of related literature. Section~\ref{sec:mathmodel} presents our model of the SJRP and its impulse control approximation. Section~\ref{sec:SDE} introduces {the HJB equation}, and {the aforementioned stochastic identity, Eq.~\eqref{eq:keyidentityprop}}. Section~\ref{sec:algorithm} describes our algorithm.
{Section~\ref{sec:testproblems_benchmarks} 
presents our test problems and the benchmark policies.  
Computational results are described in Section~\ref{sec:computational_results}.}
We conclude in Section~\ref{sec:conclusion} with some discussion 
on possible future research directions.
Appendices \ref{app:A}--\ref{app:D} contain further technical details, which include a {formal derivation} of the key stochastic identity, how the hyperparameters of our neural networks are tuned, 
a complete description of the benchmark policies, 
and the validation of our computational approach in a one-dimensional 
inventory model with fixed cost.

\section{Literature Review}\label{sec:litreview}
{We review the following three relevant streams of literature: (i) Joint replenishment problems, (ii) impulse control problems, 
and (iii) deep learning methods for solving high-dimensional PDEs.}

\textbf{Joint replenishment problems.} The JRP is a classical problem 
{in inventory management} that has {received considerable research attention} over the past few decades. Broadly speaking,  the JRP literature can be divided into problems with deterministic demand and those with stochastic demand \citep{goyal1989joint}. 
The deterministic JRP can be viewed as the most basic multi-item extension of the classical economic order quantity (EOQ) model \citep{harris1990many}. It assumes constant demand rates over an infinite planning horizon, with replenishment occurring at regular intervals. The objective is to determine the optimal frequency for replenishing the different items (see, e.g., \citealp{viswanathan1996new,kaspi1991economic,goyal1979note}).
{One class of effective policies is characterized by the power-of-two rule under which orders are placed at power-of-two multiples of a base period.}
\citealp{roundy198598,roundy198698,jackson1985joint} have shown that the best power-of-two policies can perform within {two to six} percent of optimality. {Much of the (deterministic) JRP literature focuses on the following cost structure: There is a group replenishment {fixed} cost, 
{which is incurred every time an order is placed,} 
and there are individual fixed costs for each item.} By contrast, the generalized JRP (GJRP) allows the shared fixed costs to depend on the given collection of items included in the replenishment \citep{federgruen1992joint}. \cite{adelman2005duality} prove the existence of optimal policies for the GJRP with capacity constraints using duality theory. \cite{adelman2012computing} use approximate dynamic programming to determine near-optimal policies in this setting.

Our paper, however, concerns the SJRP; {we now review this literature}. 
For the single-item problem with procurement fixed costs, 
\cite{scarf1960} {introduces the notion of $K$-convexity to show} that the optimal policy has the following simple structure: 
If the inventory level drops {to} or below a fixed level $s$, 
an order is placed to raise it back to another fixed level $S$. 
Efficient algorithms exist to find optimal parameters for such $(s,S)$ policies; 
see, for example, \cite{veinott1965computing,zheng1991finding}. 
{In a single-item variant involving advance demand information, \cite{gallego2001integrating} give a geometric characterization of a convexity notion that generalizes $K$-convexity and simplifies the classical optimality proofs of $(s,S)$ policies and their state-dependent generalizations.}
{Since Scarf's seminal work, there have been several attempts to identify effective structural properties 
of optimal policies for multi-item inventory models. 
We refer the readers to \cite{perera2023survey} for a detailed review of these works, 
and discuss only the most relevant literature. \cite{johnson1967optimality} considers an infinite-horizon, 
$d$-item SJRP in discrete time, with integer-valued demand and the inventory 
process living in {the state space $\Z^d$}. 
The demand in each time period is allowed to depend on the initial inventory levels in that period. 
The cost structure is identical to ours, i.e., there is a single fixed cost and linear, item-specific variable costs, 
whenever an order is placed. He considers both the total discounted cost and long-run average cost objectives and establishes the optimality of a so-called $(\sigma,S)$ policy under both objectives, 
where $\sigma$ is a subset of $\Z^d$, and $S$ is a vector in $\Z^d$. 
Under this policy, if the inventory state $x$ is in $\sigma $ and $x\leq S$ component-wise, 
then one orders up to the levels $S$, and does not order if $x\notin \sigma$. 
Note that the policy is not specified when $x \in \sigma$ and $x \not \leq S$, 
but if the initial inventory state $x\leq S$, then under a $(\sigma,S)$ policy, 
the inventory state always stays below $S$ (component-wise), so a $(\sigma,S)$ policy is optimal in this case. 
Furthermore, for the long-run average cost objective, even if the initial state $x\not \leq S$, 
an optimal policy can simply first wait for the inventory levels to deplete and drop below $S$, 
and then follow a $(\sigma, S)$ policy. \cite{kalin1980optimality} proves the optimality of $(\sigma,S)$ policies 
under conditions that are different from those in \cite{johnson1967optimality}, 
and characterizes the complement of $\sigma$, a set where no order is placed, 
as being increasing under a certain partial ordering. 
His proof makes use of the notion of $(K,\eta)$-quasiconvexity in $\R$, 
a generalization of the concept of quasiconvexity used in \cite{porteus1971optimality} and in \cite{schal1976optimality}, 
for various single-item inventory problems. 
\cite{gallego2005k} recover results in \cite{kalin1980optimality} using generalized notions of $K$-convexity in $\R^d$.
Even though the optimal policy structure is known when the initial inventory state $x\leq S$, 
determining the ordering region $\sigma$ and the order-up-to level vector $S$ is extremely challenging computationally, especially in high dimensions.}

{Due to the computational intractability, research on SJRP often focuses on classes of {parameterized} control policies that are intuitively appealing. 
Once attention is restricted to a particular class of control policies, the problem reduces to {optimizing the policy parameters}. This is often accomplished by numerical evaluations of the policy performances through simulation. See, for example,} \cite{atkins1988periodic,federgruen1984coordinated,ignall1969optimal,silver1965some,silver1981establishing,viswanathan1997note,balintfy1964basic}. For a {comprehensive review} of the various heuristic control policies, we refer to \cite{khouja2008review}.

{\bf Impulse control problems.} 
One year after Scarf established the optimality of the $(s,S)$ policy in discrete time, \cite{beckmann1961inventory} demonstrates that this policy is also optimal in continuous time with Poisson demand. To prove this, \cite{beckmann1961inventory} uses what is known as uniformization. This approach restricts ordering decisions to occur only at the jumps of the Poissonian demand, effectively reducing the problem to its discrete-time counterpart. 
To address more general demand models such as diffusion or mixed-diffusion processes, impulse control problems offer a flexible framework that can effectively capture the lump-sum ordering nature of continuous-time inventory control problems with fixed costs.  
{These stochastic control problems} allow the system controller to intervene through jumps in the underlying state space \citep{bensoussan1984impulse}. 
Several authors have used the associated HJB equations to prove the optimality of $(s,S)$--type policies for various one-dimensional inventory models; see, for example, \linebreak\cite{constantinides1976stochastic,harrison1983impulse, sulem1986solvable, bensoussan2005optimality, helmes2017continuous}. For applications of the one-dimensional impulse control problems in economics, we refer the reader to \cite{stokey2008economics}.

While numerous studies have investigated the continuous-time inventory model with fixed costs for single-item problems, only a handful of papers have addressed multi-item problems 
(\cite{perera2023survey}). A pioneering attempt to solve the multi-item setting is made by {\cite{sulem1986explicit}}, who considers a two-item inventory system with constant, deterministic demand rates. There are individual fixed costs for ordering a single item, and a fixed cost of ordering both items, which is lower than the sum of individual fixed costs. The variable costs are assumed to be zero. Using the associated HJB equation, she establishes the optimality of a $(\gamma, \Gamma)$ policy, where $\gamma$ represents the ordering boundary, and $\Gamma$ denotes the order-up-to level boundary, similar to the $(\sigma, S)$ policy in the discrete-time setting. \cite{li2022optimal} {also consider two-item inventory systems, and} extend {\cite{sulem1986explicit}} in several ways. First, they allow variable costs, and establish the optimality of a $(\gamma, \Gamma)$ policy in the deterministic setting with constant demand rates. 
They provide characterizations of the boundaries $\gamma$ and $\Gamma$ in this case, facilitating their numerical constructions. 
Second, they establish various theoretical properties of a stochastic model with diffusion demand, 
such as the $(K, \eta)$-quasiconvexity of the value function, and used the finite-difference scheme to obtain numerical solutions. Their numerical study suggests that the optimal policy in the stochastic case is also of the $(\gamma, \Gamma)$ type. 
{Note that their finite-difference scheme is 
not computationally feasible in high-dimensional settings, i.e., 
when the number of distinct items is large.}
Other common numerical methods for solving impulse control problems {include} iterated optimal stopping and policy iteration; see, e.g., \cite{azimzadeh2018convergence}, \cite{azimzadeh2016weakly}, and Chapter 12 in \cite{oksendal2019stochastic}. Unfortunately, these methods face the same {computational} tractability issues.

{\bf Deep learning methods for solving high-dimensional PDEs.} Deep learning has gained significant traction as a method for solving high-dimensional PDEs since the seminal papers \cite{han2017deep, han2018solving}.
Their method uses the probabilistic characterization of {semilinear} PDEs as backward stochastic differential equations (BSDEs); see, for example, \cite[Chapter~7]{yong1999stochastic}. The method outlined in \cite{han2018solving} proposes to discretize the BSDE forward in time using a standard Euler scheme, {and approximate} the value function and {its gradient} with neural networks. It then learns the value function and its gradient by optimizing the neural network parameters {through minimization of} the expected squared distance between the known terminal value and the discretized BSDE's terminal value. We refer to \cite{han2020convergence} for a detailed convergence study.
With the advent of high-performance GPU computing, these deep-learning methods have been successfully applied to solve high-dimensional PDEs; see \cite{beck2020overview,weinan2021algorithms,chessari2023numerical} for extensive literature reviews on the subject. 

\cite{ata2023dynamic} and \cite{ata2024singular,ata2023drift} introduced these techniques to the operations management domain.
The former studies a semilinear parabolic PDE linked to a drift control problem arising in dynamic call center scheduling {(also see \citealp{ata2026dynamic})}, while the latter two develop computational methods for solving control problems in stochastic processing networks in the heavy-traffic regime.
{\cite{atazhou2024} solves a sequential vehicle routing problem motivated by an eviction enforcement application using a similar approach. \cite{atalisi2024} studies dynamic control of a make-to-order manufacturing system with throughput time constraints in high dimensions using deep learning.}
\cite{ataxu2025} solves a high-dimensional dynamic stochastic matching problem building on the deep BSDE method. 
The preceding literature predominantly addresses drift-rate control problems
within the canonical deep BSDE setting,
where the value function satisfies an almost-sure equality
and the neural network training relies on a standard mean-squared loss.
In contrast, while our impulse control problem is also tackled within the broad deep BSDE framework,
it differs fundamentally: 
{Our approach is based on the framework of BSDEs with constrained jumps \citep{kharroubi2010backward}, 
which provides a probabilistic representation of the value function of impulse control problems 
as the minimal solution to a constrained BSDE --- a substantive methodological departure from the unconstrained BSDEs 
that underlie the canonical deep BSDE setting. Consequently,} the stochastic identity satisfied by the value function involves
an almost-sure inequality rather than an equality. 
This renders the standard mean-squared loss unsuitable, 
and necessitates a novel loss function penalizing violations of the inequality,   
deliberately crafted reference processes, 
and careful tuning of joint penalty-rate and learning-rate schedules;
see also the discussion in {Section~\ref{sec:algorithm}}. Relatedly,
our deep-learning framework is {inspired in part by the mathematical} finance literature \citep{buehler2019deep, biagini2023neural}, 
where deep learning is employed to identify optimal hedging strategies for pricing financial derivatives, 
problems often modeled by nonlinear PDEs. 
{In particular, the stochastic identity (Eq.~\eqref{eq:keyidentityprop}) on which our approach is based exploits a connection between impulse control problems and the so-called stochastic target problems, the latter of which have been studied extensively in the mathematical finance literature (see, e.g., \citealp[Chapter~7]{touzi2012optimal}).} 
{Finally, \cite{bayraktar2023neural} studies an optimal switching problem --- a problem class related to impulse control \citep{bouchard2009stochastic} --- in energy markets, 
building on the framework of \cite{hure2020deep}. Unlike our setting, theirs involves a finite set of modes and exogenous, 
rather than controller-induced, jumps in the state. 
They further establish error bounds for their approach, whereas we do not pursue such bounds here; 
see also \cite{germain2022approximation} and \cite{han2020convergence} for error analyses of other deep learning methods for solving PDEs.}

\section{Model}\label{sec:mathmodel}
Consider a warehouse that stocks and sells $d$ different items. 
The warehouse manager may place orders for any subset of these items. 
Define the order vector {$y=(y_1,\ldots,y_d)^\top$, where $y_i\geq0$ denotes the number of units of item $i$ ordered 
and the superscript $^\top$ denotes transpose,} 
and let $c(y)$ denote the total ordering cost, where
$$
c(y) = \begin{cases}
c_0 + \sum_{i=1}^d c_i y_i,\  &\text{if }  \sum_{i=1}^d y_i > 0,\\
0, &\text{ otherwise.}  \\
\end{cases}
$$
That is, the warehouse manager incurs a fixed cost $c_0$ for placing the order---for example, $c_0$ may represent the fixed cost of using a shared resource associated with an order such as the transportation cost incurred from using a truck---and $c_iy_i$ denotes the associated variable cost of ordering $y_i$ units of item $i$ for $i=1,\ldots,d$.

We let {$x=(x_1,\ldots,x_d)^\top$} denote a generic system state, where $x_i$ represents the inventory level of item $i$. 
The warehouse manager reviews the system state periodically at times $n=1,2,\ldots$. 
At the beginning of each period, upon observing the system state, but before observing demand, he makes an ordering decision; 
{it is possible that the manager decides not to place an order. 
If an order is placed, it is received immediately, i.e., the delivery lead time is zero.}
{Then, at the end of the time period,} demand is realized, and any unmet demand is backlogged. 
In addition to ordering costs, there are linear holding and backlog costs. 
To be specific, a penalty cost $p_i$ is incurred for each unit of backlogged demand {of item $i$}, 
and a holding cost $h_i$ is charged for each unit of {inventory of item $i$} carried over from
one period to the next. Thus, if we let
\begin{equation}
f_i(x_i)=\begin{cases}
h_ix_i, \ & x_i\geq0, \\
-p_ix_i, \ & x_i<0,
\end{cases}
\label{eq:inv_cost}
\end{equation}
then we have that
$$
f(x) = \sum_{i=1}^d f_i(x_i)
$$
is the per-period state cost {borne} by the system {manager} when the system state is $x$. All costs are incurred at the end of the period.

We now describe the system dynamics mathematically. First, let $X(n)$ denote the system state (i.e., the vector of inventory levels for items $1,\ldots,d$) at the end of period $n$. Second, let $\xi(n)$ denote the $d$--dimensional demand vector in period $n$. 
{We assume that $\{\xi(n)\}_{n \geq 1}$ are independent and identically distributed
nonnegative random vectors with mean vector $\mu \in \mathbb{R}^d$ and 
positive-definite covariance matrix $\Sigma \in \mathbb{R}^{d \times d}$.}
To facilitate future analysis, we define the cumulative demand {during the first $n$ periods as} 
\begin{equation}\label{eq:cumdemand}
D(n) = \sum_{k=1}^n \xi(k), \ n\geq1.
\end{equation}
Third, let $y(n)\geq0$ denote the order vector in period $n$. Then, the evolution of the system state can be expressed as
\begin{equation}\label{eq:inventorydynamics}
X(n) = X(n-1) + y(n)- \xi(n),\ n\geq1,
\end{equation}
where $X(0)$ denotes the initial system state.

Focusing on the infinite-horizon discounted cost criterion, the warehouse manager's problem is to choose a non-anticipating control (policy), denoted by $u=\{y(n), n\geq1\}$, consisting of order vectors for each period. 
The associated expected total discounted cost can be written as follows:
\begin{equation}\label{eq:MDPobjective1}
{\tilde J(x, u)}=\E_x\Big[\sum_{n=1}^{\infty} \gamma^{n-1} \left(f(X(n)))+c(y(n))\right)\Big],    
\end{equation}
where $\gamma\in(0,1)$ is the discount factor, and $\E_x[\cdot]$ denotes the expectation conditional on the initial state $X(0)=x$.

For every initial state $x$, we seek to find the optimal ordering policy that minimizes the total expected discounted cost \eqref{eq:MDPobjective1}. That is, we want to determine
\begin{equation}\label{eq:MDPvaluefunction}
{\tilde V(x)\coloneqq \inf_u \tilde J(x,u)},
\end{equation}
where the infimum is taken over the class of non-anticipating controls.
As mentioned earlier, it is computationally prohibitive to solve \eqref{eq:MDPvaluefunction} because the standard MDP-based techniques suffer from the curse of dimensionality. With this in mind, we next describe a continuous-time approximation of the foregoing problem. 

{\bf Impulse Control Approximation.}
We now formally derive an impulse control approximation to the preceding inventory control problem.
One can rewrite the stochastic dynamics of the inventory state process \eqref{eq:inventorydynamics} as
\begin{equation}\label{eq:discrete_process}
X(N) = x + \sum_{n=1}^N y(n)- D(N),\ N\geq1.
\end{equation}
The key simplification is to view the warehouse manager's problem in continuous time and allow him to place orders at any point in time. With slight abuse of notation, we define a new \emph{continuous-time} process
\begin{equation}\label{eq:continuous_process}
X(t) = x + \sum_{\{j : \tau_j \leq t\}} y^{(j)} - D(t),\ t\geq 0,
\end{equation}
where the time $t$ replaces $N$ in Eq.~\eqref{eq:discrete_process} 
{and $y^{(j)}$ represents the $j$th order placed at time $\tau_j$.}
In addition, we approximate the cumulative demand process $D(\cdot)$ by a Gaussian process. 
To be specific, we assume
\begin{equation}\label{eq:browniandemand}
D(t)= \mu t + \sigma B(t),\ t\geq0,
\end{equation}
where $B(\cdot)$ is a $d$--dimensional standard Brownian motion, $\mu \in \R^d$ is the drift vector, {and $\sigma \in \R^{d\times d}$ is chosen so that $\sigma \sigma^{\top} = \Sigma$, where $\Sigma$ is the covariance matrix of the per-period demand vector $\xi(n)$.} 

{This Gaussian approximation of the original demand process is obtained by matching its first two moments. 
Such Gaussian approximations are ubiquitous in the inventory management literature, e.g., the newsvendor model and its extensions, 
because they enhance tractability. Despite its benefits, this approach introduces a potential approximation error. 
Therefore, once we derive the proposed policy under this approximation, we test its performance in the context of the {discrete-time} formulation 
using the original demand distribution. As the computational results in Section \ref{sec:computational_results} demonstrate, 
our approach yields policies for the original discrete-time problem formulation 
that match or outperform the best available benchmark,
confirming the practical effectiveness of this approximation.}

The warehouse manager's goal is to minimize the total expected discounted cost. 
More precisely, let $0 < \tau_1 < \tau_2 < \ldots$ denote the sequence of strictly increasing, non-anticipating stopping times
at which he places orders. Then, under the control policy $u=\{(y^{(j)},\tau_j),j\geq1\}$, the system state evolves according to the $d$-dimensional stochastic process 
\begin{equation}
X^{u}(t) = x - \mu t - \sigma B(t) + \sum_{\{j : \tau_j \leq t\}} y^{(j)}, {\ t\geq 0.}
\end{equation}

Letting the annual interest rate be $r > 0$, the expected total discounted cost under policy $u$ is given by
\begin{equation}
{J(x, u)} = \E_x\Big[\int_0^\infty e^{-r t} f(X^u(t)) dt + \sum_j e^{-r \tau_j} c(y^{(j)})\Big].
\end{equation}
We then define the optimal value function as follows:
\begin{equation}\label{eq:impulsevaluefunction}
V(x)=\inf_{u\in\mathcal{U}} {J(x,u)}\ \text{for each } x\in\R^d,
\end{equation}
where the infimum is taken over the set $\mathcal{U}$ of {non-anticipating} policies.



\section{The HJB Equation and an Equivalent Stochastic Identity}\label{sec:SDE}


In this section, we {describe} an identity, Eq.~\eqref{eq:keyidentityprop} below, that is similar in spirit to the results used by \cite{han2018solving} to justify their deep BSDE method for solving {semilinear} PDEs. That earlier work provided the motivation for our approach, though the specifics of our approach to solving the impulse control problem differ significantly.

The value function $V$ we seek in \eqref{eq:impulsevaluefunction} can be characterized as the solution to a particular kind of highly nonlinear PDE, namely, the HJB equation associated with the impulse control problem. 
Using the infinitesimal dynamic programming principle, one can show formally that the value function $V$ should satisfy the following HJB equation:
\begin{equation}\label{eq:qvi}
\left(\left(\mathcal{L}+r\right) V - f \right) (x) \vee \left( V(x) - \inf_{y \in \R_+^d} \left(V(x+y) + c(y)\right)\right) = 0, {\ x\in\R^d,}
\end{equation}
{where $\mathcal{L}:=-\frac{1}{2}\cdot \text{Tr}\left(\sigma \sigma^\top \text{Hess}(V) \right) + \mu^\top \nabla V$ is the diffusion operator; $\nabla V$ and $\text{Hess}(V)$ denote the gradient and Hessian of $V$, respectively; and 
$a \vee b\coloneqq\max\{a,b\}$ for scalars $a$ and $b$.}
To elaborate, 
the first term on the left-hand side of Eq.~\eqref{eq:qvi}, {which we will refer to as the ``no action'' term,} represents taking no action, {which intuitively means that} the warehouse manager opts to do nothing when the system state is $x$.  
The second term, {which we will refer to as the ``intervention'' term,} corresponds to placing an order when the system state is $x$, with the warehouse manager determining the order size vector $y^*$ that minimizes future expected discounted costs. 

The HJB equation \eqref{eq:qvi} will not be used directly to learn the value function $V$ in our approach. 
Instead, we rely on a related probabilistic characterization of $V$, to be described shortly in Eq.~\eqref{eq:keyidentityprop}. 
To this end, we begin by specifying what we call a \emph{reference policy} 
that is used to generate sample paths of the inventory process. 
That is, it facilitates our sampling of the paths living in the state space. 
We refer to the state process under the reference policy as the \emph{reference process}, 
denoted by $\tilde{X}(\cdot)$. Roughly speaking, we choose a reference policy so that 
the reference process tends to occupy the parts of the state space 
that we expect the optimal policy to visit frequently, {so as to enable efficient and stable training of the neural networks}.
{Note that the reference process plays the same role as the forward diffusion process in 
standard deep BSDE methods; see, e.g., Eq.~(16) in \cite{han2018solving}.}

{
As mentioned earlier, \cite{johnson1967optimality} characterized the optimal policy for the SJRP as a $(\sigma, S)$ policy. 
Also, for our approximating impulse control problem, a careful examination of the HJB equation \eqref{eq:qvi} 
suggests that whenever an order is placed, the inventory vector is brought up to an order-up-to vector 
that is independent of the current inventory vector, provided the latter does not exceed the order-up-to vector
(more will be said about this {in} Section \ref{sec:algorithm}). 
Thus, we build on a (randomized) order-up-to policy for choosing a reference policy. 
Additionally, we set the order times under the reference policy as the event times of a Poisson process with rate $\lambda$. 
We let $T_j$ denote the $j$th event time of the Poisson process for $j=1,2,\dots$. 
Given these event times, we let $Z_j \in \mathbb{R}^d,$ for $j=1,2,\dots,$ 
be random vectors drawn independently from a common distribution $\Phi$. 
In particular, $Z_j$ will serve as the (random) order-up-to vector for the $j$th order under the reference policy. 
{That is, $(Z_j - \tilde{X}(T_j-))^+$ denotes the order quantity associated with the order-up-to vector $Z_j$; 
here, $(x)^+$ denotes the componentwise positive parts of $x\in \mathbb{R}^d$, 
and $\tilde{X}(T_j-)$ denotes the inventory vector just before time $T_j$.}
}

{
Randomizing the order-up-to vectors {$Z_j$} facilitates greater exploration of the state space and improves the training process. 
{To further enhance state space exploration,}
we let $\left(\zeta_j\right)_{j=1}^\infty$ denote a sequence of i.i.d random vectors in $\mathbb{R}^d$, 
whose components are independent exponential random variables with mean $\alpha_i = \alpha \mu_i/\lambda$ for $i=1, \dots, d$. 
Then, we let 
$$
Y_j = \max\left\{ Z_j - \tilde{X}(T_j-), \zeta_j \right\}
$$
denote the vector of order quantities of various items for the $j$th order, 
where the $\max\{\cdot, \cdot\}$ is applied componentwise. 
This adjustment guarantees positive order quantities even if 
$Z_j \ngeq \tilde{X}(T_j-)$, i.e., if $Z_j$ does not dominate $\tilde{X}(T_j-)$ componentwise, 
{in which case some components of $(Z_j - \tilde{X}(T_j-))^+$ are zero.} 
To ensure the stability of the reference process, we set $\lambda \mathbb{E}[\zeta_j] < \mu_i$ for $i=1,\ldots,d$, i.e., $\alpha < 1$. Further details of the reference process parameters will be described in Section 5. Given $\lambda, \Phi$ and $\alpha$, we let
\begin{equation}\label{eq:uncontrolled_control}
U(t; \alpha, \lambda, \Phi) = \sum_{j: T_j \le t} Y_j, \quad t \ge 0,
\end{equation}
and refer to the process $U(\cdot; \alpha, \lambda, \Phi)$ as the reference
policy. {When the {meaning is clear from the context}, we suppress dependence on $\alpha$, $\lambda$, and $\Phi$, 
and simply write $U(\cdot)$.}
The corresponding reference process is given by
\begin{equation}\label{eq:uncontrolled_reference}
\tilde{X}(t) = x - \mu t - \sigma B(t) + U(t), \quad t \ge 0.
\end{equation}
}





{{\bf Stochastic identity.} We now present the equivalent probabilistic representation of the value function $V$. 
Fix a reference policy $U(\cdot)$. 
{Then, for any $T>0$ and $\tilde{X}(0)=x\in\R^d$, 
$V(x)$ can be formally characterized as follows: 
}
\begin{equation}\label{eq:keyidentityprop}
\begin{aligned}
V(x) = \sup \Big\{ v \in \R \ :\  & v - \int_0^T e^{-r t}   \nabla V(\tilde{X}(t))^\top \sigma dB(t)  \\
&  \leq e^{-r T} V(\tilde{X}(T)) + \int_0^T e^{-r t} f(\tilde{X}(t))dt + \sum_{j : T_j \leq T} c(Y_j)e^{-rT_j},\ \textnormal{a.s.} \Big\}.
\end{aligned}
\end{equation}
{\noindent The right-hand side of Eq.~\eqref{eq:keyidentityprop} defines {a} stochastic target problem, 
which, roughly speaking, is a control problem that aims to steer the system state so as to hit a target set at the terminal time $T$. 
The connection between impulse control and stochastic target problems was first established in \cite{bouchard2009stochastic}, for the case of discrete jump sizes. 
\cite{kharroubi2010backward} later extended this connection for more general impulse control problems. 
{For some intuition on the stochastic identity \eqref{eq:keyidentityprop}, 
note that {under suitable integrability conditions,} the almost-sure inequality implies that 
$V(x)$ satisfies
\[
V(x) \leq \E\left[\int_0^T e^{-r t} f(\tilde{X}(t))dt + \sum_{j : T_j \leq T} c(Y_j)e^{-rT_j} + e^{-r T} V(\tilde{X}(T)) \right].
\]
The right-hand side is the expected discounted cost under a policy that follows the reference policy on $[0,T]$ 
and acts optimally from time $T$ onwards.
Since $V(x)$ is the optimal cost over all policies, any such policy must cost at least $V(x)$.
The surprising part is the exact characterization of $V(x)$ as the largest value for which the almost-sure inequality holds; 
we provide a more substantive formal derivation in Appendix~\ref{app:A}.}

\section{Computational Method}\label{sec:algorithm}

The identity \eqref{eq:keyidentityprop} serves as the basis of our computational method. Seeking an approximate solution to it, we define the discretized loss function displayed in Eq.~\eqref{eq:lossfunction}. Recall the reference process defined in Eq.~\eqref{eq:uncontrolled_reference}. As a preliminary to defining the loss function, we fix a partition $0=t_0<t_1<\ldots<t_N=T$ of the time horizon and simulate $K$ discretized sample paths of the reference process at times $t_0,t_1,\ldots,t_N$; see Subroutine~1. 

{
{\bf Reference policy.} As reviewed in Section 2, decades of research activity has culminated in effective heuristic policies for the SJRP. These include the $(R,S)$, $(Q,S)$, {$(R,Q,S)$}, and can-order policies, which are described in Section 6 and Appendix C.2. We consider their performances as benchmarks for our proposed policy. Crucially, they each have
an order-up-to vector $S$ as a policy parameter. In what follows, we set
\begin{equation}\label{eq:reference-S}
\mathbb{E}[Z_j] = {\psi}S \quad \text{for } j \ge 1,
\end{equation}
where $S$ is derived for one of these three benchmark policies, {and $\psi\geq1$ is a scaling factor}. More specifically, assuming the optimal order-up-to levels
are positive,\footnote{This is a reasonable assumption for our test problems because the backlog costs are considerably larger than the holding cost.} we set $Z_j$ to be a vector of lognormal random variables, where its $i$th component has mean {$\psi S_i$} and variance {$(\nu\psi S_i)^2$}. Here, $\nu > 0$ is a tuning parameter and is used to control the degree of exploration of the state space. 
We introduce the tuning parameter $\psi$ to control the degree of exploration during the training process. By default, we set $\psi=1$, unless the order-up-to vector $S$ is too close to the origin, in which case we consider $\psi>1$ to increase exploration. For virtually all test problems, we use $\psi=1$ and report its value only when it differs from one; see Tables~\ref{tab:hyperparams}--\ref{tab:hyperparams-50d} in Appendix~\ref{app:B_configs}.
Similarly, the parameter $\alpha$ that determines the means of the exponential random vectors $\zeta_j$, for $j \ge 1$, is also viewed as a tuning parameter.
}


\begin{algorithm}[h!]
\floatname{algorithm}{Subroutine}
\caption{Euler-Maruyama discretization scheme.}
\begin{algorithmic}[1]
\Require A reference process $U(\cdot;\alpha,\lambda,\Phi)$, the drift vector $\mu$, the covariance matrix $\sigma\sigma^\top $, the time horizon $T$, the number of intervals $N$, a discretization step size 
{$\Delta t = T/N$, and}
an initial state $x_0$. 
\Ensure Discretized reference process $\tilde{X}(t_n)$ for $n = 1, \ldots, N$, the Brownian increments $\Delta B(t_n)$ and compound Poisson increments $\Delta U(t_n)$ for $n = 0, \ldots, N - 1$.

\Function{Euler-Maruyama}{$T, \Delta t, x$}
    \State {Construct the partition $0=t_0<\ldots<t_N=T$, where $t_{n+1}=t_n+\Delta t$ for $n=0,\ldots, N-1$. }
    \State {Generate $N$ i.i.d.~$d$--dimensional Gaussian random vectors $\Delta B(t_n)$ with mean vector zero and covariance matrix $\Delta t I$ for $n=0,\ldots, N-1$.}
    \State {Generate $N$ i.i.d. $d$--dimensional vectors $Z_n$, using $\Phi$ {whose components have independent lognormal distributions}, for $n=0,\ldots, N-1$.}
    \State {Generate $N$ i.i.d. $d$--dimensional vectors $\zeta_n$ for $n=0,\ldots, N-1$, where each component is drawn independently from an exponential distribution with mean $\alpha_i$.}
    \For{$n = 0$ to $N-1$}
    \State Generate $I_n \sim \operatorname{Bernoulli}(\lambda \Delta t)$
    \State $\Delta U(t_n) \leftarrow I_n \cdot \max\{Z_n - \tilde{X}(t_n),\zeta_j\}$ (with $\max\{\cdot,\cdot\}$ applied componentwise)
            \State $\tilde{X}(t_{n+1}) \gets \tilde{X}(t_n) -\mu \Delta t - \sigma \Delta B(t_n) + \Delta U(t_n)$
    \EndFor
    \State \Return $\tilde{X}(t_n)$ for $n = 0, \ldots, N$, and $\Delta B(t_n)$ and $\Delta U(t_n)$ for $n = 0, \ldots, N-1$.
\EndFunction
\end{algorithmic}
\label{algo:euler}
\end{algorithm}



{\bf Loss function.} Our method, described in Algorithm \ref{algo}, computes the loss function defined in Eq. \eqref{eq:lossfunction} by summing over the sample paths of the reference process and over discrete time steps to approximate the various integrals over $[0,T]$ appearing in the identity \eqref{eq:keyidentityprop}.
We approximate the value function $V(\cdot)$ and its gradient $\nabla V(\cdot)$ by deep neural networks, $H_\theta$ and $G_\vartheta$, respectively, with associated parameter vectors $\theta$ and $\vartheta$. Using the sample paths of the time-discretized reference process, the loss function $L_\beta$ is defined as
\begin{equation}\label{eq:lossfunction}
\begin{aligned}
L_\beta & (\theta,\vartheta):= \frac{1}{K}\sum_{k=1}^{K}\Big[-H_\theta(x^{(k)}_0)+ \beta\cdot 
 l\Big(H_\theta(x^{(k)}_0)  - e^{-r T} H_\theta(\tilde{X}^{(k)}(T))\\
& - \sum_{n=0}^{N-1} e^{-r t_n} G_\vartheta(\tilde{X}^{(k)}(t_n))^\top \sigma\Delta B^{(k)}(t_n)-\sum_{n=0}^{N-1} e^{-r t_n} f(\tilde{X}^{(k)}(t_n)) \Delta t_n - \sum_{n=0}^{N-1} e^{-r t_n} c(\Delta U^{(k)}(t_n))\Big)\Big],
\end{aligned}
\end{equation}
where the tuning parameter $\beta\gg1$ can be viewed as a Lagrange multiplier. The penalty function $l(\cdot)$ is given by
\begin{equation}\label{eq:squaredrectloss}
{l(x)=\left[(x)^+\right]^2.}
\end{equation}
The function $L_\beta$ consists of (the negative of) the value function estimate, $-H_\theta(x_0^{(k)})$, and a penalty term, $\beta l(\cdot)$, averaged over $K$ initial states $x_0^{(k)},\ k=1,\ldots,K$. {As mentioned in Step~6 of Algorithm~2, the process $\tilde{X}(\cdot)$ is simulated continuously, meaning that the terminal state of an iteration becomes the initial state of the next iteration. Loosely speaking, this approach can be seen as an attempt to capture the long-run behavior of the reference process.}
The penalty term penalizes violations of the almost sure inequality that appears in Eq.~\eqref{eq:keyidentityprop}. The higher $\beta$ is, the lower the violation probability is. {Note from Eq.~\eqref{eq:keyidentityprop} that we seek to choose the largest value $v$ while satisfying the inequality there almost surely. Thus, the loss function $L_\beta$ is designed so that as we minimize it,} we increase the value function estimate $H_\theta$, while simultaneously decreasing the violation probability. Also, we average over $K$ initial states to avoid overfitting at any particular state. The tuning parameter $\beta$ seeks to balance these two different objectives.

{\bf The penalty rate and learning rate schedules.} Our computational method uses a stochastic gradient descent (SGD) type method to minimize the loss function. 
This involves choosing a learning rate schedule as done for deep learning algorithms. 
Additionally, we also need to choose the parameter $\beta$ that is used to penalize the violations of the stochastic inequality on the right-hand side of {Eq.~\eqref{eq:keyidentityprop}}. 
As one would expect, the higher the value of $\beta$, the lower the violation probability is. 
However, we observed empirically that the training is less stable with a higher value of $\beta$. 
{To deal with this instability, we adopt the approach of gradually increasing}
the penalty parameter $\beta$ during training, while decreasing the learning rate {as needed}. 
{More specifically, we begin with a relatively low value of the penalty parameter $\beta$, 
at which we observe that a significant fraction of simulated sample paths violate the stochastic inequality. 
At the same time, we set the learning rate parameter relatively large and iterate until the loss appears to converge, 
lowering the learning rate if the training becomes unstable.
Then, we increase the penalty parameter (e.g., by a factor of ten) and repeat the same procedure: 
We iterate until the loss appears to converge, and decrease the learning rate as needed if we observe instability.}
We keep increasing the penalty parameter in this manner until the violation probability is sufficiently low, e.g., less than 1\%. 
For this final value of the penalty parameter, we choose a learning rate schedule manually by observing the evolution of the loss as we iterate, 
and lowering the learning rate as needed. {See Appendix~\ref{app:B_material} for more details.}

Using the approach just described, and through extensive experimentation, we identified a set 
of {\em a priori} fixed penalty rate and learning rate schedules {--- $\{\beta_m\}$ and $\{\eta_m\}$, where $m$ 
indexes the iterations ---} to be considered in our experiments, 
as reflected in Algorithm \ref{algo}. In particular, 
{both schedules are treated as inputs to the algorithm and held fixed across all iterations of a given training run.
As the reader will see, while Algorithm~\ref{algo} formally allows $\{\beta_m\}$ and $\{\eta_m\}$ to vary with $m$, 
in practice we restrict to piecewise-constant schedules, with values changing only at pre-selected small set of iterations.}
See Appendix~\ref{app:B_configs} for {additional details on} the specific penalty parameter and learning rate schedules used for our test problems.
Further implementation details of our computational method are provided in Appendix~\ref{app:B}.

\begin{algorithm}[h!]
\caption{Algorithm for approximating the value function of the impulse control problem.}
\begin{algorithmic}[1]
\Require The number of iteration steps $M$, a batch size $K$, a time horizon $T$, the number of intervals $N$, a discretization step-size $\Delta t$, an initial state $x_0$, a fixed reference policy $U(\cdot;\alpha,\lambda,\Phi)$, a penalty scheme $\{\beta_m\}$, {a learning rate schedule $\{\eta_m\}$}, and an optimization algorithm (e.g., SGD, ADAM, or RMSProp).
\Ensure The neural network approximation of the value function $V(\cdot)$, $H_{\theta}(\cdot)$, and of 
{its gradient $\nabla V(\cdot)$,}
$G_{\vartheta}(\cdot)$.
\State Initialize the neural networks for $H_\theta(\cdot)$ and $G_\vartheta(\cdot)$. Set {$x^{(k)}_0=x_0$} for all $k=1,\ldots,K$.
\For{$m \gets 1$ to $M$}
\State For each sampled initial point $x^{(k)}_0$, simulate a discretized sample path of the inventory state process, Brownian motion increments, and increments of the compound Poisson process $\{\tilde{X}^{(k)}, \Delta B^{(k)},\Delta U^{(k)}\}$ with a time horizon $T$ and a step-size $\Delta t$, starting from $x^{(k)}_0$ by invoking \textsc{Euler-Maruyama}$(T, \Delta t, x^{(k)}_0)$ for $k = 1, \ldots, K$.
\State Compute the Lagrangian function
\begin{equation*}
\begin{aligned}
&\!\!\!\!L_{\beta_m}(\theta,\vartheta) = \frac{1}{K}\sum_{k=1}^{K}\Bigg(-H_\theta(x^{(k)}_0)+ \beta_m \Bigg[\Big(H_\theta(x^{(k)}_0)  - e^{-r T} H_\theta(\tilde{X}^{(k)}(T))\\
&\!\!\!\! - \sum_{n=0}^{N-1} e^{-r t_n} G_\vartheta(\tilde{X}^{(k)}(t_n))^\top \sigma\Delta B^{(k)}(t_n)-\sum_{n=0}^{N-1} e^{-r t_n} f(\tilde{X}^{(k)}(t_n)) \Delta t - \sum_{n=0}^{N-1} e^{-r t_n} c(\Delta U^{(k)}(t_n))\Big)^{+}\Bigg]^2\Bigg)
\end{aligned}
\end{equation*}
\State Compute the gradient of the Lagrangian $\nabla L$ with respect to $(\theta, \vartheta)$ and update using the chosen optimization algorithm, {with the learning rate $\eta_m$}.
\State Update the initial states by setting, for each $k=1,\ldots,K$, $x^{(k)}_0 \gets \tilde{X}^{(k)}(T)$.
\EndFor
\State \Return Functions $H_\theta(\cdot)$ and $G_\vartheta(\cdot)$.
\end{algorithmic}
\label{algo}
\end{algorithm}

{\bf The proposed policy.}
After obtaining 
{the neural network approximations $H_\theta$ 
and $G_\vartheta$}
of the value function and its gradient, {respectively}, 
we {propose} a policy exploiting the structure of the HJB equation \eqref{eq:qvi}.
To be specific, Eq. \eqref{eq:qvi} implies that either 
the no action term is zero, i.e.,
\begin{equation}\label{eq:qvi_diffusion}
\mathcal{L}V(x)+r V(x) - f(x) = 0, 
\end{equation}
{or the intervention term is zero, i.e., }
\begin{equation}\label{eq:qvi_obstacle}
V(x) = \inf_{y \in \R_+^d} \left\{V(x+y) + c(y)\right\}.
\end{equation}
{In the former case, it is optimal to take no action. 
In the latter case, it is optimal to place an order, 
and the order vector is given by the minimizer of 
the right-hand side of Eq. \eqref{eq:qvi_obstacle}.}

One way to operationalize this is to check if condition \eqref{eq:qvi_diffusion} holds. If \eqref{eq:qvi_diffusion} holds for some state $x$, no action is required; otherwise, intervention through impulse control is warranted. That is, if for our neural network estimates of $V$ and $\nabla V$, we have 
\begin{equation}\label{eq:qvi_diffusion_approx}
\mathcal{L}V(x)+r V(x) - f(x) \leq -\epsilon < 0,
\end{equation}
then we intervene through impulse control. Here $\epsilon$ is a suitably chosen numerical tolerance to account for optimization and sampling errors. 
To obtain an estimate of the value of $\mathcal{L}V$, note that estimates of the value function $V$ and its gradient $\nabla V$ are already available from the deep neural network approximations, $H_\theta$ and $G_\vartheta$, respectively. It then remains to estimate the Hessian of $V$ through{, for example,} automatic differentiation {of $G_\vartheta$} \citep{hure2020deep,pham2021neural}. 


If one decides to intervene, the impulse control $y^*$ is determined by selecting 
\begin{equation}\label{eq:argminimpulse}
y^*\in \operatorname{argmin}_{y \in \R_+^d} \left\{H_\theta(x+y) + c(y)\right\}.
\end{equation}
To solve the minimization problem in \eqref{eq:argminimpulse}, one can use a variety of numerical optimization techniques. Further details are provided in Appendix~\ref{app:B}.


{We interpret this policy in the context of the original (discrete-time) formulation of the SJRP as follows: Given a system state $x$, the warehouse manager checks if an intervention is warranted. If so, he places an order, and his order vector is $y^*$ given in Eq.~\eqref{eq:argminimpulse};} 
{see Algorithm~\ref{alg:policy} for a detailed description.}

{\bf Simulation and optimization.} 
{{\em A priori}, under the proposed policy, every time that the warehouse manager decides to place an order, 
he needs to solve the minimization problem in Eq. \eqref{eq:argminimpulse} to determine the order size vector $y^*$. 
This is computationally burdensome for performance simulations. 
However, the following observation allows us to compute a single order-up-to vector only once at the beginning of a simulation, 
resulting in a substantial speedup. 
Consider the minimization problem in Eq. \eqref{eq:argminimpulse}, 
{and assume that the optimal solution $y^* \neq 0$.}
Because the inventory state $x$ is not a decision variable, 
the problem is in fact equivalent to
\begin{equation}\label{eq:argminimpulse_alt}
\min_{y \in \R_+^d} \left\{H_\theta(x+y) + c(y) + \sum_{i=1}^d c_i x_i \right\},
\end{equation}
which can be further simplified and bounded as follows: 
\begin{align}
\min_{y \in \R_+^d} \left\{H_\theta(x+y) + c(y) + \sum_{i=1}^d c_i x_i \right\}
 &~= \min_{y \in \R_+^d} \left\{H_\theta(x+y) + c(x+y) \right\} \nonumber\\
 &~\geq \min_{z \in \R^d} \left\{H_\theta(z) + c(z) \right\}. \label{eq:single_order_up_to_level}
\end{align}
{Here, the equality uses the fact that $y^* \neq 0$, 
which ensures that $c(y^*) + \sum_{i=1}^d c_i x_i = c(x+y^*)$.}
Let $z^*$ be the optimal solution to the minimization problem in Eq. \eqref{eq:single_order_up_to_level}. 
Then, as long as $x \le z^*$ componentwise, 
the minimization problems in Eqs. \eqref{eq:argminimpulse} and \eqref{eq:single_order_up_to_level} are equivalent. 
This in turn implies that under our proposed policy, the order-up-to vector is simply $z^*$, irrespective of the initial state $x_0$, 
as long as $x_0\leq z^*$. 
Thus, if we assume that $z^* \geq 0$, which holds in all our test problems, 
then, since we start our simulations from the initial state $x_0=0$, 
the order-up-to vector $z^*$ does not need to be recomputed at every ordering decision. 
}


{
{\bf An alternative approach for deriving the optimal order quantity.} One can alternatively use the neural network approximation of the gradient, $G_{\vartheta}$, for solving \eqref{eq:argminimpulse}. In particular, the first-order
optimality condition yields $\nabla_z V(z) = -c$, where $c$ is the vector of variable ordering costs. Thus, substituting $G_\vartheta$ for $\nabla V$ then yields
\begin{equation}\label{eq:stationary_conditionZ}
G_\vartheta(z^*) = -c.
\end{equation}
Solving this system gives the optimal order-up-to vector. 
{One can compare these approaches in two ways: (i) by fixing the same hyperparameter configuration for both methods \emph{a priori}, or (ii) by considering the best-performing hyperparameter configuration for each method separately. Under the first comparison, the results of the two methods can differ substantially. Under the second comparison, however, the results are close, except for some test problems in the high-variability case.}
As mentioned earlier, since the order-up-to vectors only need to be computed once at the beginning of a simulation, 
running both optimization routines yields negligible overhead. 
Thus, we compute both vectors and select the one with the lower simulated cost.}
{In practice, both optimization problems are solved over a compact box $[\underline{z},\overline{z}] := \{z\in\mathbb{R}_+^d:\underline{z}\leq z\leq\overline{z}\}$. The box is chosen large enough to contain a near-optimal order-up-to vector, while remaining small enough to restrict the search to regions that were thoroughly explored during training, where the neural network approximations are considered reliable; see Appendix~\ref{app:B} for further details.}

\begin{algorithm}[h!]
\caption{Proposed inventory control policy.}
\label{alg:policy}
\begin{algorithmic}[1]
\Require Neural network approximations $H_\theta(\cdot)$ and $G_\vartheta(\cdot)$ of the value function $V(\cdot)$ and its gradient $\nabla V(\cdot)$, respectively; numerical tolerance $\epsilon > 0$; 
{a compact box $[\underline{z},\overline{z}]
:= \{z\in\mathbb{R}_+^d:\underline{z}\le z\le\overline{z}\}$
over which the order-up-to vector is computed;}
initial point $z_0 \in [\underline{z}, \overline{z}]$.
\State \textbf{Preprocessing:} Compute the order-up-to vector $z^*$ using one of two approaches:
\begin{itemize}
    \item[(a)] Solve the minimization problem
    $z^* \in \operatorname{argmin}_{z \in [\underline{z}, \overline{z}]} \left\{ H_\theta(z) + c(z) \right\}$;
    \item[(b)] Solve the stationarity condition $G_\vartheta(z^*) = -c$ by minimizing $\frac{1}{2}\|G_\vartheta(z) + c\|_2^2$ over $z \in [\underline{z}, \overline{z}]$.
\end{itemize}
\Comment Both approaches use L-BFGS with box constraints, initialized at $z_0$; see Appendix~\ref{app:B} for details.
\State \textbf{At each decision epoch:} Observe the current inventory state $x$.
\State Compute the Hessian approximation $\mathrm{Hess}_\vartheta(x) := \mathrm{Jacobian}(G_\vartheta)(x)$ via automatic differentiation.
\State Compute the no-action term
$$
\mathcal{N}(x) := -\tfrac{1}{2}\mathrm{Tr}(\sigma\sigma^\top \mathrm{Hess}_\vartheta(x)) + \mu^\top G_\vartheta (x) + r H_\theta(x) - f(x).
$$
\If{$\mathcal{N}(x) \leq -\epsilon$} \Comment{No-action condition \eqref{eq:qvi_diffusion} violated}
    \State Place an order with order vector $y^* = (z^* - x)^+$.
\Else
    \State Do not place an order.
\EndIf
\end{algorithmic}
\end{algorithm}

{To repeat, under natural conditions one needs to solve \eqref{eq:argminimpulse} or \eqref{eq:stationary_conditionZ} 
only once as argued above. This is done in Step 1 of Algorithm~\ref{alg:policy} as the preprocessing step.}


\section{Test Problems and Benchmark Policies}\label{sec:testproblems_benchmarks}

{
We begin by describing the test problems.
We consider a total of 43 test problems\footnote{Additionally, Appendix D considers a one-dimensional impulse control problem studied in \cite{sulem1986solvable} and verifies the validity of our solution method by comparing it to the one derived in \cite{sulem1986solvable}; see Figure~\ref{fig:value-gradient} in Appendix~\ref{app:D}.} that can be divided into three groups. 
{First, we consider 27 twelve-dimensional problems based on an example studied in \cite{atkins1988periodic}.
These problems vary across three attributes: Demand variability, fixed ordering cost, and backlog penalty costs.
For each attribute, we consider low, medium, and high values, yielding 27 test problems.}}
    

{
Computing the optimal policy for these problems is intractable due to the curse of dimensionality. 
Thus, to assess the performance of our proposed policy, we use the {state-of-the-art} heuristic policies drawn from the literature as benchmarks; 
see below for more on those benchmark policies. However, when the number of items is small, e.g., $d=2$, 
one can compute the optimal policy using standard dynamic programming methods. 
This, of course, helps assess how close our proposed policy is to the optimal one for low-dimensional problems. 
Therefore, we next consider 7 test problems of dimension two. 
To do so, we consider only two products of the 12-dimensional problem. 
For our base case of the two-dimensional test problems, we choose the medium values of demand variability, ordering cost, and {backlog penalty costs}. 
Additionally, while fixing all else we consider setting each of {the three attributes,}
one at a time, to low and high values, resulting in a total of seven test problems of dimension two.
Lastly, to assess the scalability of our method, we consider nine test problems of dimension 50 that are designed by building on the earlier test problems. More will be said about them below.
}

{
For all test problems, we set the time unit as a year, the initial inventory to zero, and the annual interest rate to
$r=5\%$, and allow the warehouse manager to make ordering decisions weekly.
}

{
{\bf 12-dimensional test problems.} We build on the 12-item problem studied in \cite{atkins1988periodic}. 
We set the holding cost parameters $h_i=2$ for all items ($i=1,\dots,12$). 
To capture the three levels of demand variability, we proceed as follows: 
{For low variability, we model the item-level demand as independent Poisson processes. 
In this case, the coefficient of variation (CV) of item-level annual demand ranges from $0.158$ (for items with annual demand rate $40$) 
to $0.224$ (for items with annual demand rate $20$); cf. Table~\ref{tab:atkins12item_description}.
For medium and high variability cases, we model each item-level annual demand as an independent negative binomial random variable,  
and set their parameters so that the CV of the annual demand is $0.5$ per item for the medium variability case, 
and it is $1$ per item for the high-variability case.}
We choose the negative binomial distribution because it can effectively model both high and low demand rates, and its dispersion parameter can be used to fit a wide range of empirical retail demand distributions \citep{nahmias1994optimizing,ehrhardt1979power}. Table~\ref{tab:atkins12item_description} shows the annual demand rates and variable costs for each item. As mentioned earlier, we consider three possible values for the
fixed cost of ordering and the backlog penalty costs. More specifically, we set $p_i=p$ for all $i$ and let
$p \in \{10, 50, 100\}$ and $c_0 \in \{20, 100, 200\}$. \\
}

\begin{table}[h!]
\centering
\caption{Demand rates and variable costs for the 12-item problem}
\label{tab:atkins12item_description}
\begin{tabular}{ccccccccccccc}
\toprule
Item & 1 & 2 & 3 & 4 & 5 & 6 & 7 & 8 & 9 & 10 & 11 & 12\\
\midrule
Demand rate & 40 & 35 & 40 & 40 & 40 & 20 & 20 & 20 & 28 & 20 & 20 & 20 \\
$c_i$ & 0.1 & 0.1 & 0.2 & 0.2 & 0.4 & 0.2 & 0.4 & 0.4 & 0.6 & 0.6 & 0.8 & 0.8 \\
\bottomrule
\end{tabular}
\end{table}

{\bf Low-dimensional test problems.} Letting $d=2$, the parameters of our test problem are taken from the preceding one. 
Namely, we focus on items 1 and 7 there. For the base case, we set $c_0=50$ and $p=50$, and model item-level annual demand as a negative
binomial random variable with annual CV $=0.5$, i.e., the medium variability case. The other six 2-dimensional test problems follow from varying one of the following at a time: The demand variability, the fixed cost of ordering, and the backlog penalty parameter. 
More specifically, these alternative parameters correspond to demand variability $\in \{\text{low, high}\}$, $c_0 \in \{20, 100\}$, and $p \in \{10, 100\}$.
For the low variability case, {we model the item-level demands as independent Poisson processes as done earlier.} 
{For the high-variability case, item-level demands follow negative binomial distributions with annual CV $=1.0$.} 




{
{\bf High-dimensional test problems.} As mentioned earlier, we also consider nine 50-dimensional test problems. 
To set the problem parameters, we partition the items into three groups. 
To be specific, items 1-15 form the first group, items 16-30 form the second group, and items 31-50 form the third group. 
Within each group, items share the same annual demand rate, per-unit holding cost, backlog penalty, and variable cost; 
these parameters differ across groups. We vary the fixed cost $c_{0}\in\{50,150,250\}$, 
and again model demand as Poisson (low variability) or negative binomial {with an item-level CV of 0.5 (medium variability) or 1.0 (high variability).}
These lead to nine additional test problems. Table~\ref{tab:100dtest} summarizes the key parameters for each group.
}

\begin{table}[h!]
\centering
\caption{Problem parameters for the 50-item test case (annual demand rates and holding and backlog costs).}
\label{tab:100dtest}
\begin{tabular}{cccc}
\toprule
{Items} & 1--15 & 16--30 & 31--50 \\
\midrule
Demand rate & 50.0 & 25.0 & 12.5 \\
$h_i$& 1.0 & 2.0 & 4.0 \\
$p_i$ & 25.0 & 50.0 & 100.0 \\
$c_i$ & 0.1 & 0.2 & 0.4 \\
\bottomrule
\end{tabular}
\end{table}

{
Next, we describe the benchmark policies we use to assess the performance of our proposed policy. Crucially, we make all performance comparisons in the context of the original discrete-time formulation, cf.~Eqs.~\eqref{eq:inv_cost}--\eqref{eq:MDPvaluefunction}, as opposed to the impulse control approximation \eqref{eq:continuous_process}--\eqref{eq:impulsevaluefunction}, even though our proposed policy is derived using the latter.
}

{
{\bf Benchmark policies.} 
As mentioned earlier, it is computationally feasible to solve the low-dimensional test problems. Thus, we compute the optimal policy using policy iteration and use it as the benchmark for those cases. For the other test problems, we rely on effective heuristic policies drawn from the extensive literature on the SJRP. More specifically, we consider the following benchmark policies:
\begin{enumerate}
    \item[(i)] The periodic review $(R,S)$ policy proposed by \cite{atkins1988periodic};
    \item[(ii)] the $(Q,S)$ policy described in \cite{pantumsinchai1992comparison}; 
    {\item[(iii)] the hybrid $(R,Q,S)$ policy proposed by \cite{ozkaya2006stochastic};}
    \item[(iv)] the can-order policy described in \cite{federgruen1984coordinated};
    \item[(v)] the independent $(s,S)$ control policy \citep{sulem1986solvable,zheng1991finding};
\end{enumerate}
}

Under the $(R,S)$ policy, the inventory levels of all items are replenished every $R$ time periods, 
with each item $i = 1, \dots, d$ having its inventory level brought back to its target $S_i$. 
The $(Q, S)$ policy, studied in \cite{pantumsinchai1992comparison}, 
places an order whenever the total aggregate demand since the previous replenishment exceeds a specified value $Q$, 
raising the inventory, for each item $i = 1, \dots, d$, 
back to its order-up-to level $S_i$. 
{The $(R, Q, S)$ policy, proposed by \cite{ozkaya2006stochastic}, 
builds on the $(R, S)$ and $(Q, S)$ policies, and can be described as follows: 
It places an order whenever (a) the total aggregate demand since the previous replenishment reaches $Q$, 
or (b) $R$ time units have elapsed since the last decision epoch, whichever occurs first, 
raising each item $i = 1, \dots, d$ to its order-up-to level $S_i$.}
{Under} the can-order policy, each item $i$ is managed with three control parameters: 
A reorder point $s_i$, a can-order level $c_i$, and an order-up-to level $S_i$. 
For each item $i = 1, \dots, d$, a replenishment is triggered when its inventory level drops to or below $s_i$, 
bringing it back up to $S_i$. If at this point the inventory level of any other item $j \neq i$ is 
at or below its can-order level $c_j$, this item $j$ is also included in the replenishment, 
and its inventory is raised to $S_j$. 
The independent $(s,S)$ policy assumes the setting without any benefits from joint ordering 
and calculates the individual optimal $(s,S)$ levels as if there were no opportunity 
for group savings on the fixed cost. 
Further implementation details of these benchmark policies are provided in Appendix~\ref{app:C2}.

\section{Computational Results}\label{sec:computational_results}
{For the test problems introduced in Section~\ref{sec:testproblems_benchmarks}, 
we now compare the performance of our proposed policy, henceforth referred to as the neural network (NN) policy, to the benchmarks. 
For the benchmark policies, i.e., $(R,S)$, $(Q,S)$, {$(R,Q,S)$}, can-order, and independent $(s,S)$, 
we compute their performance estimates by averaging over {100,000} sample paths of length 10{,}000 time periods each. 
The resulting standard errors are negligible ($\leq0.01\%$ of the respective estimated performance), 
so we do not report them. 
For the NN policy, depending on the problem dimension, 
we estimate its performance by averaging over {100 or 1,000} sample paths of length 10{,}000 periods each; 
further details are provided below.
The numbers of simulated sample paths under the NN policy are chosen to strike a balance 
between computational tractability and statistical accuracy: 
On the one hand, checking the no-action condition Ineq.~\eqref{eq:qvi_diffusion_approx} is computationally intensive, 
as it involves evaluating the approximate Hessian of the value function through automatic differentiation; 
on the other hand, we simulate enough paths to obtain standard errors that are less than $1\%$ of the respective estimated performance.
For a test problem, we say that the NN policy {\em matches} a benchmark 
if the estimated benchmark cost is within $1\%$ of the estimated cost under the NN policy, 
i.e.,
\[
\frac{\left|~\text{Estimated benchmark cost} - \text{Estimated NN cost}~\right|}{\text{Estimated NN cost}} \leq 1\%,
\]
and that the NN policy {\em beats} a benchmark if the estimated benchmark cost is at least $1\%$ more than the estimated NN cost, i.e.,
\[
\frac{\text{Estimated benchmark cost} - \text{Estimated NN cost}}{\text{Estimated NN cost}} > 1\%.
\]

{On our hardware, each neural network training takes roughly 2.5 hours in 12 dimensions and 4 hours in 50 dimensions. 
{On a single CPU, the performance simulations take roughly 6 hours in 12 dimensions and 10 hours in 50 dimensions.}
Computation time thus grows only modestly with the problem dimension and does not exhibit exponential growth in $d$. 
This is also because we use neural networks whose capacity scales modestly with the problem dimension (e.g., three hidden layers for the 12-dimensional instances and four for the 50-dimensional instances). 
Further implementation details, such as machine configurations, programming language, and software packages, are described in Appendix~\ref{app:B_material}. 
Precise training and simulation times for each instance are reported in Appendix~\ref{app:B_configs}.}
Our code is available at \url{https://github.com/wvaneeke/DeepSJRP}.

Before presenting the computational results in detail, we first summarize the main findings.
Overall, the NN policy performs strongly and often surpasses the benchmarks by a significant margin.
{Across the 36 high-dimensional test problems (27 in dimension twelve and 9 in dimension fifty),
it matches or beats the best benchmark, and does so by a clear margin in most medium- to high-variability cases.}
In all seven two-dimensional test problems, the performance gap between the NN policy and the MDP solution is within $1\%$,
and the NN policy structure closely aligns with that of the MDP, suggesting that the NN policy learns the optimal policy in low dimensions. 
We now proceed to the detailed results. 
}

\paragraph{Results in dimension 12.} Tables~\ref{table:12dlow}, \ref{table:12dmed}, and \ref{table:12dhigh} 
report computational results in the low, medium, and high variability settings, respectively, 
for the main 12-dimensional test problems. 
Each table reports the simulated performances of our NN policy, 
together with relative gaps between the various benchmark policies and our policy.
{The performance of the NN policy is estimated by averaging over {100 sample paths.}
The standard errors are shown as percentages of the estimated averages. 
Reported standard errors are small ($<1\%$), indicating statistically reliable estimates.
}

{For the nine test problems with low variability in Table \ref{table:12dlow}, 
our NN policy performs on par with the best benchmark policy, 
{which turn out to be $(Q,S)$ and $(R,Q,S)$ in all cases (these two policies in fact attain identical costs to within simulation error across all instances in our study). In each case, the optimal $R$ is large enough that orders are nearly always triggered by the $Q$ threshold rather than by the periodic review, so $(R,Q,S)$ effectively reduces to $(Q,S)$. 
The same is true for all subsequent test problems as well. Thus, in the discussions that follow, we do not discuss $(R,Q,S)$ explicitly.}
The NN policy outperforms the can-order policy by a significant margin in most cases, 
but the performance gap shrinks as the backlog penalty increases. 
Intuitively, the shrinking performance gaps can be explained by the can-order policy's tendency
to order more proactively than other benchmark policies, 
suggesting that it is better at avoiding states with large backlogs 
and performs better as the backlog penalty increases. 
The NN policy also significantly outperforms the independent $(s,S)$ policy. 
This is expected, as the independent $(s,S)$ policy does not exploit cost-saving opportunities 
of joint replenishment and instead triggers an order each time an item reaches its reorder point individually. 
}

{
We now focus on the performance comparisons among NN, $(R,S)$, and $(Q,S)$ in Table \ref{table:12dlow}. {For all values of the fixed cost,} the NN policy tends to perform better when compared to $(R,S)$ and $(Q,S)$, as the backlog penalty increases. A possible explanation is that the neural networks are flexible enough to learn the shape of the ordering boundary well, so under high backlog penalties, the NN policy can better avoid costly states with large backlogs. When the fixed cost is high, it is optimal to place orders less frequently, allowing demand to pool between replenishments and thereby reducing effective variability. 
In this regime, the system behaves nearly deterministically, and we expect $(Q,S)$ and $(R,S)$ to be close to optimal. The intuition that $(R,S)$ is near-optimal is threefold: (i) Infrequent orders imply that aggregate inter-order demand is less variable, bringing the system closer to being deterministic; 
(ii) in a deterministic setting, it is optimal to place orders at regularly spaced times; and (iii) if optimal order times are approximately regular, a policy that optimizes order-up-to levels and (deterministic) review times, namely $(R,S)$, is near-optimal. The slight advantage of $(Q,S)$ over $(R,S)$ likely stems from its additional adaptivity in timing orders. Consequently, in a system with low effective variability (low demand variability and large fixed cost), there is little {room} for the NN policy to improve upon these benchmarks.
}





In the medium-variability setting (Table \ref{table:12dmed}), 
the NN policy 
often delivers sizable gains over $(Q,S)$ and $(R,S)$, with the advantage growing as the backlog penalty increases. 
This aligns with the intuitive explanation discussed in the low-variability setting: The NN policy appears to learn the ordering boundary more accurately. 
The NN policy also dominates can-order in all cases, though the performance gap narrows as backlog penalties rise; 
again this observation is consistent with our earlier explanation on can-order’s proactive replenishments. 
Finally, as expected, the NN policy substantially outperforms the independent $(s,S)$ policy. 
In the high-variability setting (Table \ref{table:12dhigh}), the qualitative pattern is similar to the medium-variability cases, 
with the NN policy matching or beating the best benchmark {in all cases,} 
and often with high margins. 

\begin{table}[H]
    \centering
    \caption{\footnotesize Computational results for $d=12$ with Poisson demand and $h=2.0$ (standard errors shown as percentage of simulation average; other entries are percentage gaps vs.\ NN policy, with standard errors in percentage points).}\label{table:12dlow}
    \resizebox{0.95\textwidth}{!}
    {%
    \footnotesize
    \begin{tabular}{@{}llrrrrrr@{}}
    \toprule
    \multicolumn{2}{c}{} &  & \multicolumn{5}{c}{Cost of Other Policies (Relative to NN Policy)} \\ 
    \cmidrule(lr){4-8}
    $c_0$ & $p$ & Cost of NN Policy & $(R,S)$ & $(Q,S)$ & $(R,Q,S)$ & Can-Order & Ind.\ $(s,S)$ \\
    \midrule
    \multirow{3}{*}{20}  & 10  & $6236.59 \pm 0.06\%$  & $0.79\% \pm 0.06\%$  & $-0.10\% \pm 0.06\%$ & $-0.09\% \pm 0.06\%$ & $10.48\% \pm 0.07\%$ & $96.06\% \pm 0.12\%$  \\
                         & 50  & $7209.34 \pm 0.07\%$  & $2.58\% \pm 0.07\%$  & $1.12\% \pm 0.07\%$  & $1.12\% \pm 0.07\%$  & $7.02\% \pm 0.08\%$  & $79.67\% \pm 0.13\%$  \\
                         & 100 & $7524.61 \pm 0.07\%$  & $3.83\% \pm 0.07\%$  & $2.13\% \pm 0.07\%$  & $2.13\% \pm 0.07\%$  & $3.63\% \pm 0.07\%$  & $74.65\% \pm 0.12\%$  \\
    \midrule
    \multirow{3}{*}{100} & 10  & $10164.88 \pm 0.06\%$ & $0.60\% \pm 0.06\%$  & $-0.08\% \pm 0.06\%$ & $-0.08\% \pm 0.06\%$ & $15.53\% \pm 0.07\%$ & $149.11\% \pm 0.15\%$ \\
                         & 50  & $11752.23 \pm 0.08\%$ & $2.13\% \pm 0.08\%$  & $0.93\% \pm 0.08\%$  & $0.91\% \pm 0.08\%$  & $7.60\% \pm 0.09\%$  & $129.14\% \pm 0.18\%$ \\
                         & 100 & $12229.83 \pm 0.09\%$ & $3.36\% \pm 0.10\%$  & $1.93\% \pm 0.09\%$  & $1.93\% \pm 0.09\%$  & $6.67\% \pm 0.10\%$  & $120.94\% \pm 0.21\%$ \\
    \midrule
    \multirow{3}{*}{200} & 10  & $13082.89 \pm 0.06\%$ & $0.52\% \pm 0.06\%$  & $-0.05\% \pm 0.06\%$ & $-0.05\% \pm 0.06\%$ & $14.37\% \pm 0.07\%$ & $169.87\% \pm 0.17\%$ \\
                         & 50  & $15254.12 \pm 0.08\%$ & $0.59\% \pm 0.08\%$  & $-0.43\% \pm 0.08\%$ & $-0.42\% \pm 0.08\%$ & $7.83\% \pm 0.08\%$  & $146.19\% \pm 0.19\%$ \\
                         & 100 & $15723.30 \pm 0.08\%$ & $2.54\% \pm 0.08\%$  & $1.29\% \pm 0.08\%$  & $1.29\% \pm 0.08\%$  & $7.84\% \pm 0.09\%$  & $138.69\% \pm 0.19\%$ \\
    \bottomrule
    \end{tabular}
    }
\end{table}

\begin{table}[H]
    \centering
    \caption{\footnotesize Computational results for $d=12$ with negative binomial demand ($\operatorname{CV}=0.5$) and $h=2.0$ (standard errors shown as percentage of simulation average; other entries are percentage gaps vs.\ NN policy, with standard errors in percentage points).}\label{table:12dmed}
    \resizebox{0.95\textwidth}{!}
    {%
    \footnotesize
    \begin{tabular}{@{}llrrrrrr@{}}
    \toprule
    \multicolumn{2}{c}{} & \textbf{} & \multicolumn{5}{c}{Cost of Other Policies (Relative to NN Policy)} \\ 
    \cmidrule(lr){4-8}
    $c_0$ & $p$ & Cost of NN Policy & $(R,S)$ & $(Q,S)$ & $(R,Q,S)$ & Can-Order & Ind.\ $(s,S)$ \\
    \midrule
    \multirow{3}{*}{20}  & 10  & $8110.75 \pm 0.20\%$  & $8.04\% \pm 0.21\%$  & $0.89\% \pm 0.20\%$  & $0.89\% \pm 0.20\%$  & $16.16\% \pm 0.23\%$ & $51.02\% \pm 0.30\%$  \\
                         & 50  & $10522.38 \pm 0.20\%$ & $21.56\% \pm 0.25\%$ & $7.52\% \pm 0.22\%$  & $7.52\% \pm 0.22\%$  & $1.86\% \pm 0.21\%$  & $30.82\% \pm 0.27\%$  \\
                         & 100 & $11459.16 \pm 0.23\%$ & $28.03\% \pm 0.29\%$ & $11.06\% \pm 0.25\%$ & $11.06\% \pm 0.25\%$ & $3.56\% \pm 0.23\%$  & $29.24\% \pm 0.29\%$  \\
    \midrule
    \multirow{3}{*}{100} & 10  & $13087.53 \pm 0.22\%$ & $4.77\% \pm 0.24\%$  & $0.61\% \pm 0.23\%$  & $0.61\% \pm 0.23\%$  & $19.65\% \pm 0.27\%$ & $94.14\% \pm 0.44\%$  \\
                         & 50  & $16637.16 \pm 0.21\%$ & $14.43\% \pm 0.24\%$ & $6.53\% \pm 0.23\%$  & $6.54\% \pm 0.23\%$  & $5.24\% \pm 0.22\%$  & $65.13\% \pm 0.35\%$  \\
                         & 100 & $17658.49 \pm 0.21\%$ & $20.92\% \pm 0.26\%$ & $11.06\% \pm 0.24\%$ & $11.17\% \pm 0.24\%$ & $2.16\% \pm 0.22\%$  & $59.07\% \pm 0.34\%$  \\
    \midrule
    \multirow{3}{*}{200} & 10  & $16396.74 \pm 0.21\%$ & $4.09\% \pm 0.22\%$  & $0.80\% \pm 0.21\%$  & $0.80\% \pm 0.21\%$  & $21.67\% \pm 0.26\%$ & $116.31\% \pm 0.46\%$ \\
                         & 50  & $20728.23 \pm 0.24\%$ & $11.83\% \pm 0.27\%$ & $5.77\% \pm 0.26\%$  & $5.77\% \pm 0.26\%$  & $7.10\% \pm 0.26\%$  & $84.37\% \pm 0.44\%$  \\
                         & 100 & $21939.27 \pm 0.22\%$ & $17.43\% \pm 0.26\%$ & $9.92\% \pm 0.24\%$  & $9.92\% \pm 0.24\%$  & $2.55\% \pm 0.23\%$  & $76.41\% \pm 0.39\%$  \\
    \bottomrule
    \end{tabular}
    }
\end{table}

\vspace{-.5cm}

\begin{table}[H]
    \centering
    \caption{\footnotesize Computational results for $d=12$ with negative binomial demand ($\operatorname{CV}=1.0$) and $h=2.0$ (standard errors shown as percentage of simulation average; other entries are percentage gaps vs.\ NN policy, with standard errors in percentage points).}
    \label{table:12dhigh}
    \resizebox{0.95\textwidth}{!}
    {%
    \footnotesize
    \begin{tabular}{@{}llrrrrrr@{}}
    \toprule
    \multicolumn{2}{c}{} & \textbf{} & \multicolumn{5}{c}{Cost of Other Policies (Relative to NN Policy)} \\ 
    \cmidrule(lr){4-8}
    $c_0$ & $p$ & Cost of NN Policy & $(R,S)$ & $(Q,S)$ & $(R,Q,S)$ & Can-Order & Ind.\ $(s,S)$ \\
    \midrule
    \multirow{3}{*}{20}  & 10  & $8465.58 \pm 0.31\%$ & $22.44\% \pm 0.37\%$ & $-0.50\% \pm 0.30\%$ & $-0.50\% \pm 0.30\%$ & $19.06\% \pm 0.36\%$ & $31.21\% \pm 0.40\%$ \\
                         & 50  & $13251.55 \pm 0.29\%$ & $43.62\% \pm 0.42\%$ & $3.68\% \pm 0.30\%$  & $3.68\% \pm 0.30\%$  & $2.22\% \pm 0.30\%$  & $9.48\% \pm 0.32\%$  \\
                         & 100 & $16286.62 \pm 0.33\%$ & $47.07\% \pm 0.48\%$ & $5.74\% \pm 0.35\%$  & $5.74\% \pm 0.35\%$  & $0.62\% \pm 0.33\%$  & $8.73\% \pm 0.36\%$  \\
    \midrule
    \multirow{3}{*}{100} & 10  & $15485.99 \pm 0.28\%$ & $17.48\% \pm 0.32\%$ & $0.90\% \pm 0.28\%$  & $0.90\% \pm 0.28\%$  & $24.67\% \pm 0.34\%$ & $57.55\% \pm 0.43\%$ \\
                         & 50  & $22180.68 \pm 0.27\%$ & $40.54\% \pm 0.38\%$ & $5.75\% \pm 0.28\%$  & $5.75\% \pm 0.28\%$  & $2.51\% \pm 0.28\%$  & $24.66\% \pm 0.33\%$ \\
                         & 100 & $23886.56 \pm 0.30\%$ & $58.31\% \pm 0.47\%$ & $13.53\% \pm 0.34\%$ & $13.53\% \pm 0.34\%$ & $1.92\% \pm 0.31\%$  & $24.14\% \pm 0.37\%$ \\
    \midrule
    \multirow{3}{*}{200} & 10  & $20320.22 \pm 0.28\%$ & $14.02\% \pm 0.32\%$ & $1.15\% \pm 0.29\%$  & $1.15\% \pm 0.29\%$  & $24.77\% \pm 0.35\%$ & $70.28\% \pm 0.48\%$ \\
                         & 50  & $27493.10 \pm 0.30\%$ & $38.11\% \pm 0.41\%$ & $10.49\% \pm 0.33\%$ & $10.49\% \pm 0.33\%$ & $5.92\% \pm 0.32\%$  & $39.64\% \pm 0.42\%$ \\
                         & 100 & $29768.37 \pm 0.33\%$ & $51.82\% \pm 0.50\%$ & $16.35\% \pm 0.39\%$ & $16.35\% \pm 0.39\%$ & $2.48\% \pm 0.34\%$  & $34.56\% \pm 0.45\%$ \\
    \bottomrule
    \end{tabular}
    }
\end{table}

\paragraph{Low-dimensional results.} Table \ref{tab:2d_results} and Figures \ref{fig:mmm_case} and \ref{fig:comparison_pane} 
summarize our computational results in the two-dimensional test problems. 
To estimate performance of the NN policy, we average over 1,000 sample paths, each of length 10,000 periods. 
As mentioned earlier, the NN performance estimates are within $1\%$ of the MDP solutions. 
Across the base case (Figure~\ref{fig:mmm_case}) and the six variants (Figure~\ref{fig:comparison_pane}), 
the NN policy closely tracks the MDP solution: The decision boundaries largely overlap (dense purple regions), 
with minor local deviations near extremely-valued states. 
These visual observations are consistent with the small optimality gaps in Table~\ref{tab:2d_results} {(\(\approx\) $-0.01\%$ to $-0.69\%$ relative gap between the MDP solution and the NN policy across different cases)}, 
indicating that not only does the NN policy perform near-optimally, 
but it also captures the geometry of the optimal policy. 
Among the benchmark policies, the can-order policy performs the best and closely matches the MDP solution in all seven test problems. Recall the intuition that a good reference policy should closely track the optimal policy. Indeed, the can-order policy is a strong candidate, and we used it as the reference policy for training in most instances.
One exception is the high-variance case ($\mathrm{CV} = 1.0$, see the 3rd row of Table~\ref{tab:2d_results}). In this case, we use the optimal policy derived from the MDP. That is, we use its order-up-to vector as the mean of $Z$ (see Eq.~\eqref{eq:reference-S}), and compute the average number of orders per year and set it as the arrival rate $\lambda$ of the Poisson process that underlies our reference process. Among the reference policies we considered, this led to the best performance, supporting the intuition that under a good reference policy, sample paths of the state process should occupy parts of the state space most frequently visited by an optimal policy.
Still, the can-order policy is slightly cheaper than the NN policy in this case ($-0.09\%$). Intuitively, because high demand variability makes large backlogs more likely, the can-order policy's proactive replenishment is particularly effective in this regime.
}

\begin{table}[H]
    \centering
    \caption{\footnotesize Computational results for the 2-dimensional test problems (standard errors shown as percentage of simulation average; other entries are percentage gaps vs.\ NN policy, with standard errors in percentage points).}
    \label{tab:2d_results}
    \resizebox{0.95\textwidth}{!}{%
    \footnotesize
    \begin{tabular}{@{}lrrrrrr@{}}
    \toprule
    & & \multicolumn{5}{c}{Cost of Other Policies (Relative to NN Policy)} \\
    \cmidrule(lr){3-7}
    Test Instance & Cost of NN Policy & MDP & $(R,S)$ & $(Q,S)$ & Can-Order & Ind.\ $(s,S)$ \\
    \midrule
    Base Case  & $2940.87 \pm 0.12\%$ & $-0.21\% \pm 0.12\%$ & $47.27\% \pm 0.18\%$ & $9.66\% \pm 0.13\%$ & $0.77\% \pm 0.12\%$ & $14.95\% \pm 0.14\%$ \\
    Low CV     & $2606.97 \pm 0.05\%$ & $-0.01\% \pm 0.05\%$ & $10.09\% \pm 0.05\%$ & $2.88\% \pm 0.05\%$ & $3.79\% \pm 0.05\%$ & $26.42\% \pm 0.06\%$ \\
    High CV    & $3236.72 \pm 0.23\%$ & $-0.69\% \pm 0.22\%$ & $118.18\% \pm 0.50\%$ & $7.44\% \pm 0.24\%$ & $-0.09\% \pm 0.23\%$ & $5.82\% \pm 0.24\%$ \\
    $c_0=20$   & $2066.99 \pm 0.12\%$ & $-0.08\% \pm 0.12\%$ & $58.07\% \pm 0.20\%$ & $9.50\% \pm 0.14\%$ & $0.07\% \pm 0.12\%$ & $10.60\% \pm 0.14\%$ \\
    $c_0=100$  & $3934.66 \pm 0.12\%$ & $-0.15\% \pm 0.12\%$ & $38.90\% \pm 0.16\%$ & $8.90\% \pm 0.13\%$ & $1.02\% \pm 0.12\%$ & $18.26\% \pm 0.14\%$ \\
    $p=10$    & $2589.48 \pm 0.11\%$ & $-0.68\% \pm 0.11\%$ & $22.85\% \pm 0.14\%$ & $2.60\% \pm 0.12\%$ & $1.85\% \pm 0.12\%$ & $18.32\% \pm 0.14\%$ \\
    $p=100$    & $3071.26 \pm 0.13\%$ & $-0.04\% \pm 0.13\%$ & $56.59\% \pm 0.20\%$ & $11.85\% \pm 0.14\%$ & $0.19\% \pm 0.13\%$ & $14.55\% \pm 0.15\%$ \\
    \bottomrule
    \end{tabular}
    }
\end{table}

\begin{figure}[H]
    \centering
    \includegraphics[width=0.75\textwidth]{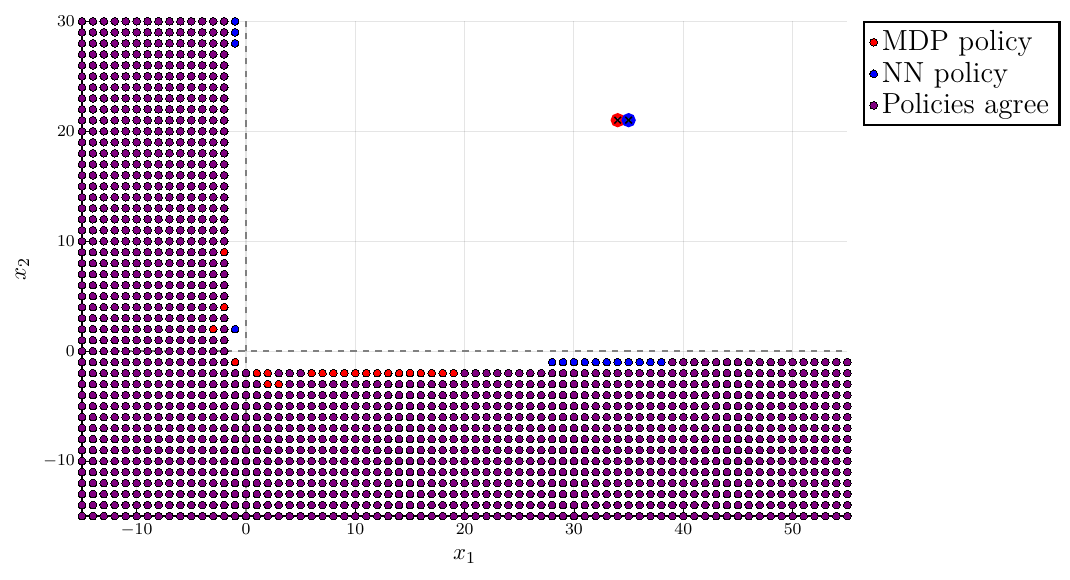}
    \caption{\footnotesize MDP and neural network policies (base case $d=2$).}
    \label{fig:mmm_case}
\end{figure}


\begin{figure}[H]
    \centering
    \begin{subfigure}[b]{0.31\textwidth}
        \includegraphics[width=\linewidth]{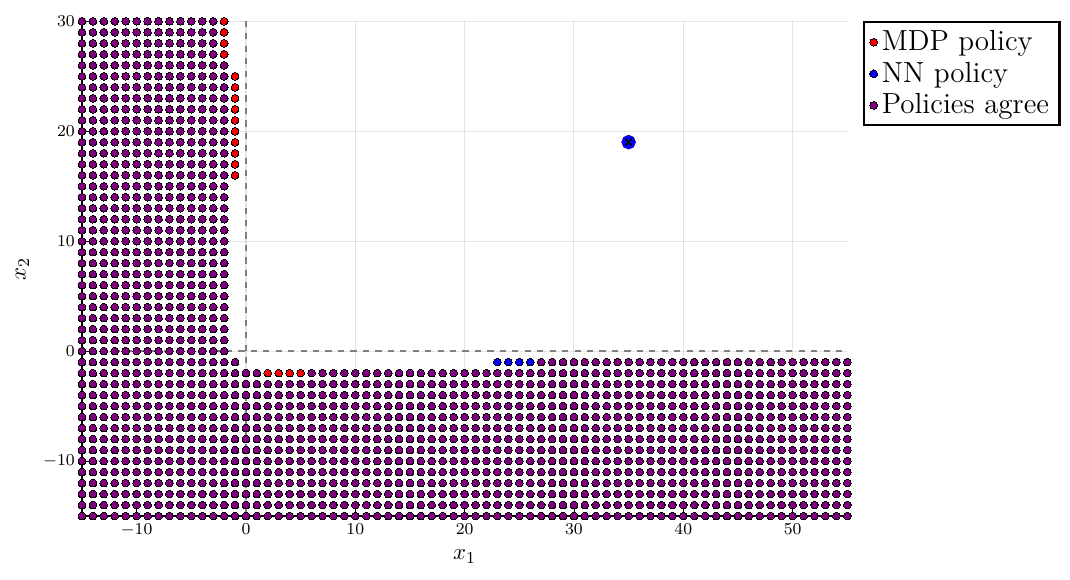}
        \caption{Low CV}
        \label{fig:lmm}
    \end{subfigure}
    \begin{subfigure}[b]{0.31\textwidth}
        \includegraphics[width=\linewidth]{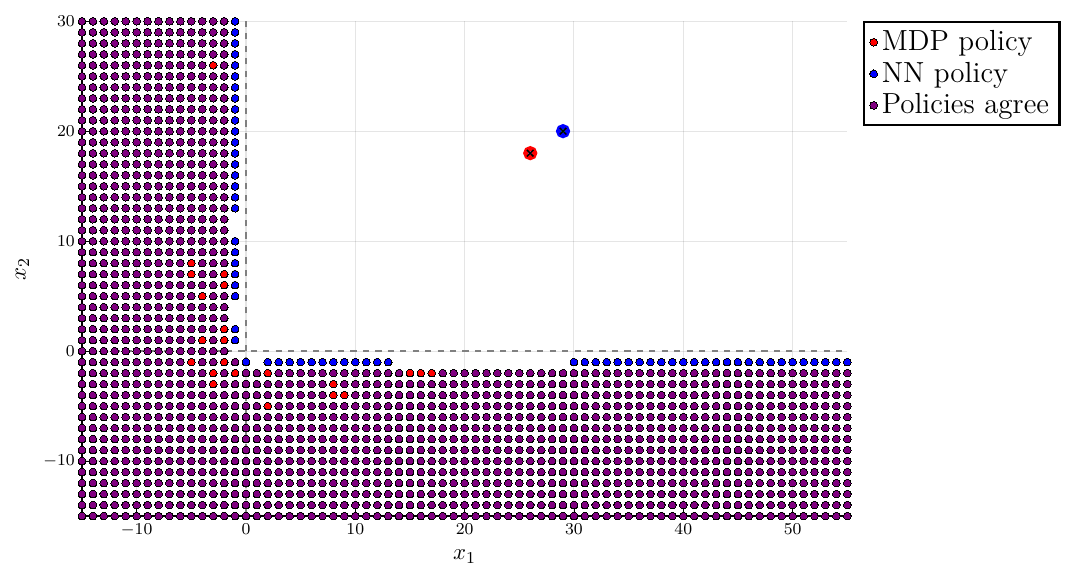}
        \caption{High CV}
        \label{fig:hmm}
    \end{subfigure}





        \begin{subfigure}[b]{0.31\textwidth}
            \includegraphics[width=\linewidth]{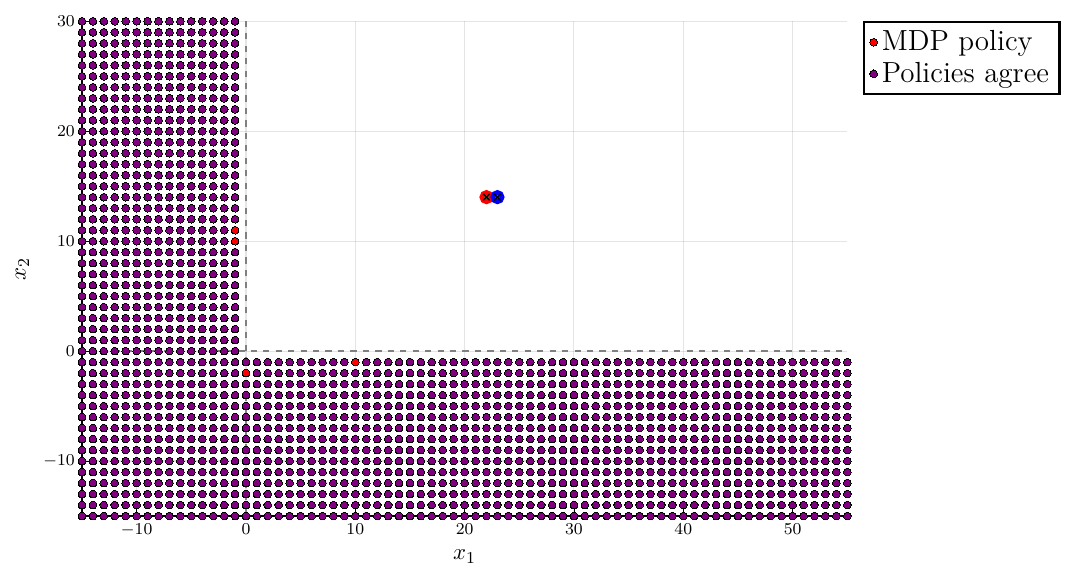}
            \caption{$c_0=20$}
            \label{fig:mlm}
        \end{subfigure}
        \vspace{5mm}
        \begin{subfigure}[b]{0.31\textwidth}
            \includegraphics[width=\linewidth]{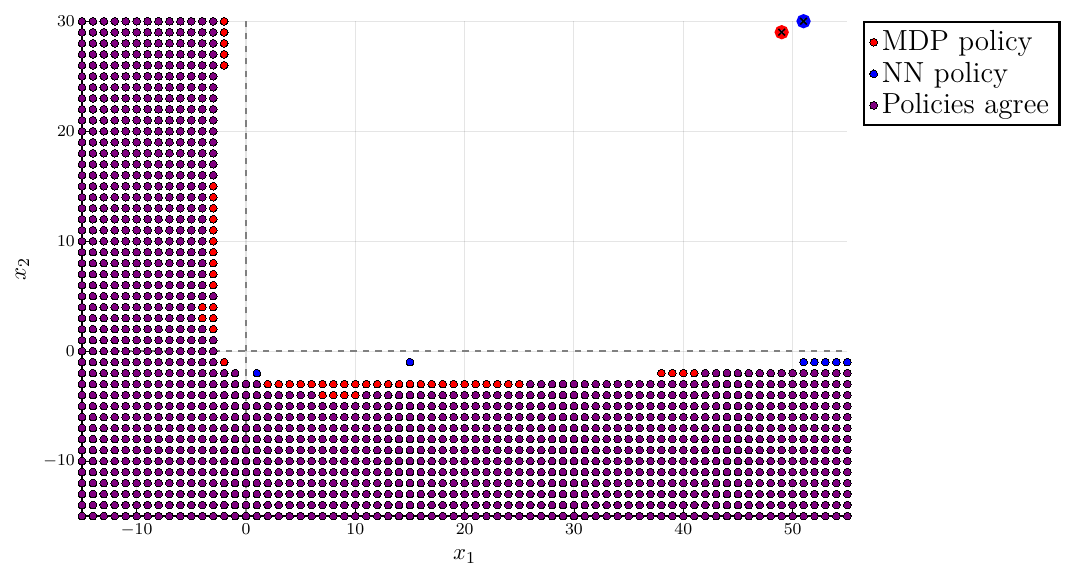}
            \caption{$c_0=100$}
            \label{fig:mhm}
        \end{subfigure}
    
        \begin{subfigure}[b]{0.31\textwidth}
            \includegraphics[width=\linewidth]{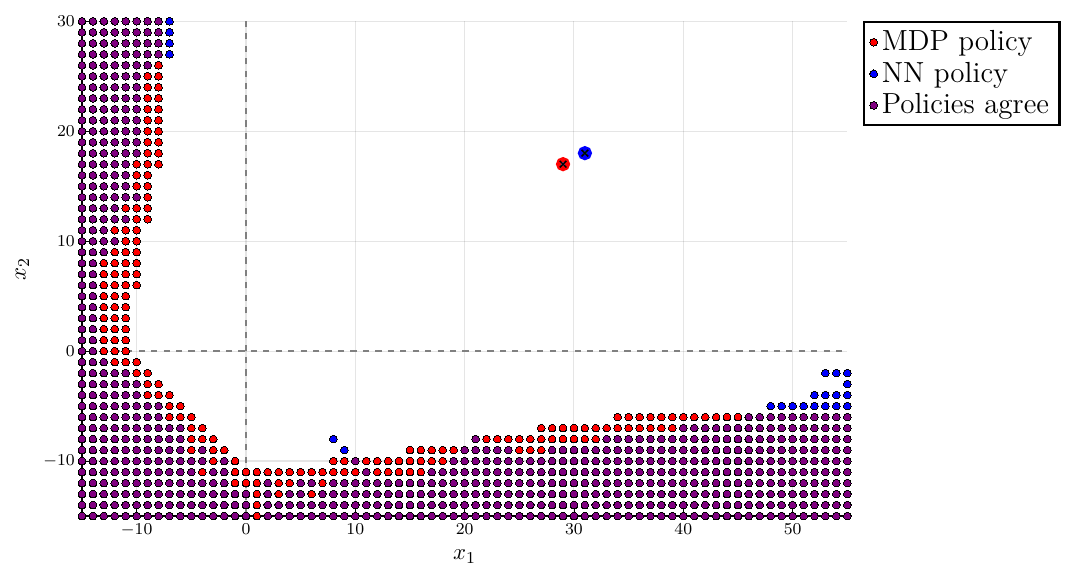}
            \caption{$p=10$}
        \label{fig:mml}
        \end{subfigure}
        \begin{subfigure}[b]{0.31\textwidth}
            \includegraphics[width=\linewidth]{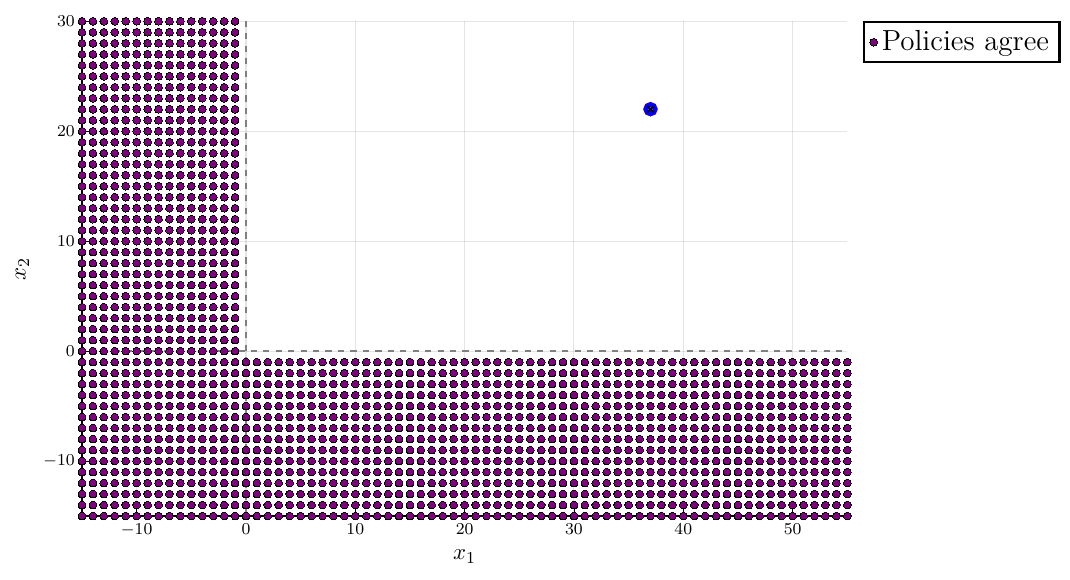}
            \caption{$p=100$}
            \label{fig:mmh}
        \end{subfigure}
    
    \caption{\footnotesize Comparison of MDP and neural network policies across the remaining six two-dimensional test problems.}
    \label{fig:comparison_pane}
\end{figure}

{\paragraph{Results in dimension 50.} Finally, we turn to the 50-dimensional test problems.
We estimate NN performance by averaging over {100} sample paths per test problem. 
Standard errors are small {($\approx 0.05\% - 0.20\%$)}, indicating reliable estimates.
The NN policy beats the best benchmark {in all cases}, with performance gaps up to {5.59\%}.
Under low variability, NN performs comparably to {$(R,S)$ and $(Q,S)$}, while clearly outperforming can-order and independent $(s,S)$ policies.
As variability increases, NN’s advantage over {$(R,S)$ and $(Q,S)$} generally widens, whereas gaps versus can-order and independent $(s,S)$ remain relatively stable.}


\begin{table}[H]
\centering
\caption{\footnotesize Computational results for $d=50$ (standard errors shown as percentage of simulation average; other entries are percentage gaps vs.\ NN policy, with standard errors in percentage points).}\label{table:50d}
\resizebox{0.95\textwidth}{!}
{\footnotesize
\begin{tabular}{@{}llrrrrrr@{}}
\toprule
\multicolumn{2}{c}{} & \textbf{} & \multicolumn{5}{c}{Cost of Other Policies (Relative to NN Policy)} \\ 
\cmidrule(lr){4-8}
CV & $c_0$ & Cost of NN Policy & $(R,S)$ & $(Q,S)$ & $(R,Q,S)$ & Can-Order & Ind.\ $(s,S)$ \\
\midrule
\multirow{3}{*}{Poisson} & 50  & $21874.05 \pm 0.05\%$ & $0.87\% \pm 0.05\%$  & $0.47\% \pm 0.05\%$ & $0.46\% \pm 0.05\%$ & $25.04\% \pm 0.06\%$ & $178.71\% \pm 0.14\%$ \\
                         & 150 & $30544.01 \pm 0.05\%$ & $0.75\% \pm 0.05\%$  & $0.45\% \pm 0.05\%$ & $0.46\% \pm 0.05\%$ & $16.06\% \pm 0.06\%$ & $263.51\% \pm 0.18\%$ \\
                         & 250 & $36305.63 \pm 0.06\%$ & $0.72\% \pm 0.06\%$  & $0.40\% \pm 0.06\%$ & $0.40\% \pm 0.06\%$ & $20.32\% \pm 0.07\%$ & $299.10\% \pm 0.22\%$ \\
\midrule
\multirow{3}{*}{0.5}     & 50  & $36287.52 \pm 0.09\%$ & $7.18\% \pm 0.10\%$  & $3.86\% \pm 0.09\%$ & $3.89\% \pm 0.09\%$ & $14.71\% \pm 0.10\%$ & $75.33\% \pm 0.16\%$  \\
                         & 150 & $48528.69 \pm 0.11\%$ & $5.21\% \pm 0.11\%$  & $3.11\% \pm 0.11\%$ & $3.10\% \pm 0.11\%$ & $20.16\% \pm 0.13\%$ & $136.02\% \pm 0.26\%$ \\
                         & 250 & $56102.44 \pm 0.10\%$ & $4.51\% \pm 0.11\%$  & $2.80\% \pm 0.10\%$ & $2.80\% \pm 0.10\%$ & $20.89\% \pm 0.12\%$ & $167.71\% \pm 0.27\%$ \\
\midrule
\multirow{3}{*}{1.0}     & 50  & $51997.42 \pm 0.18\%$ & $15.52\% \pm 0.21\%$ & $2.34\% \pm 0.18\%$ & $2.34\% \pm 0.18\%$ & $13.96\% \pm 0.21\%$ & $26.63\% \pm 0.23\%$  \\
                         & 150 & $72719.95 \pm 0.18\%$ & $15.54\% \pm 0.20\%$ & $4.90\% \pm 0.18\%$ & $4.90\% \pm 0.18\%$ & $22.44\% \pm 0.22\%$ & $59.40\% \pm 0.28\%$  \\
                         & 250 & $85697.70 \pm 0.20\%$ & $14.39\% \pm 0.23\%$ & $5.59\% \pm 0.21\%$ & $5.59\% \pm 0.21\%$ & $24.39\% \pm 0.25\%$ & $77.96\% \pm 0.35\%$  \\
\bottomrule
\end{tabular}
}
\end{table}

\section{Concluding Remarks}\label{sec:conclusion}
In this paper, we tackle a classical inventory control problem in high dimensions by formulating it as an impulse control problem and approximating its solution via deep-learning techniques. 
{Through extensive hyperparameter tuning, a tailored training algorithm, and careful engineering of policy extraction, we were able to match or beat the best available benchmarks.} 
{We note that the vast array of SJRP policies surveyed in \cite{khouja2008review} is largely tailored to settings 
that also feature item-specific minor fixed costs, with the \emph{modified} periodic review policy of \citet{atkins1988periodic} serving as a representative example. 
Nonetheless, under the single joint fixed cost structure we consider, 
these policies either reduce to or are dominated by one of our main benchmarks: 
For example, the modified periodic review policy of \citet{atkins1988periodic} and 
the $(R, s, S)$ policy of \citet{viswanathan1997note} 
are both dominated by the optimized $(R, S)$ policy, 
and the $(Q, s, S)$ policy of \citet{nielsen2005analytical} is dominated by the optimized $(Q, S)$ policy.}

{While we focus on the SJRP with linear holding and backlogging costs and unconstrained order vectors (subject only to non-negativity constraints), the methodology extends naturally in several directions. Within the present framework, it already accommodates general Lipschitz continuous inventory cost functions $f$ (cf., e.g., growth conditions (2.4) and (2.5) in \cite{kharroubi2010backward}) as well as capacity constraints on order vectors, for example, requiring the total order size not to exceed a given capacity.} 
{Another natural extension is to the lost sales case, where unmet demand is forfeited rather than backlogged.} {In this setting, the inventory state is constrained to the nonnegative orthant by minimally pushing it back into the state space whenever it reaches the boundary. The associated pushing process on each face of the orthant represents the cumulative lost demand for the corresponding product. Under this formulation, one accounts for lost-demand costs rather than the backlogging costs considered here. The resulting HJB equation is therefore posed on the nonnegative orthant and is accompanied by Neumann-type boundary conditions arising from the pushing, or reflection, at the boundary. Despite these differences, we expect that our approach can be extended to address this problem, although doing so would require resolving a number of technical subtleties. We leave this extension to future work.}

{Beyond these extensions, incorporating lead times represents another important direction for future research. In this regard, \cite{bar1995explicit} study an impulse control problem with fixed delivery lags. They show that the optimal policy depends only on the aggregate state, defined as the sum of on-hand inventory and outstanding orders. They also show that} {the HJB equation takes a modified form with an intervention operator that averages holding and backlog costs over demand during the lead time. These structural results suggest that our framework may extend to this setting, though the requisite modifications to the stochastic identity \eqref{eq:keyidentityprop} remain to be developed.}
{Looking further ahead, it seems natural to extend our methodology to more
complex supply-chain models, such as those with demand spillovers between
facilities.}

{Complementary to these problem extensions, a natural direction concerns our use of the Brownian approximation 
for cumulative demand:  
Reflected Brownian motions or subordinator-like processes are potential alternatives.  
Beyond the inventory setting itself, our approach not only provides effective solutions to supply chain and inventory management problems,} 
but also opens up avenues for applying these numerical methods to broader classes of impulse control problems. 
In particular, the techniques developed here {may be} used to tackle impulse control problems in other fields, 
most notably in financial engineering — for example, cash flow management and investment decisions \citep{altarovici2017optimal,liu2004optimal}. 
{We leave these explorations for future work.}


\begin{singlespace}
\bibliography{references}

@book{applebaum2009levy,
  title={L{\'e}vy Processes and Stochastic Calculus},
  author={Applebaum, David},
  year={2009},
  publisher={Cambridge University Press}
}

@inproceedings{he2015delving,
  title={Delving deep into rectifiers: Surpassing human-level performance on imagenet classification},
  author={He, Kaiming and Zhang, Xiangyu and Ren, Shaoqing and Sun, Jian},
  booktitle={Proceedings of the IEEE international conference on computer vision},
  pages={1026--1034},
  year={2015}
}

@book{nocedal2006numerical,
  title={Numerical Optimization},
  author={Nocedal, Jorge and Wright, Stephen J},
  year={2006},
  publisher={Springer}
}

@article{hure2020deep,
  title={Deep backward schemes for high-dimensional nonlinear PDEs},
  author={Hur{\'e}, C{\^o}me and Pham, Huy{\^e}n and Warin, Xavier},
  journal={Mathematics of Computation},
  volume={89},
  number={324},
  pages={1547--1579},
  year={2020}
}

@article{bayraktar2023neural,
  title={A neural network approach to high-dimensional optimal switching problems with jumps in energy markets},
  author={Bayraktar, Erhan and Cohen, Asaf and Nellis, April},
  journal={SIAM Journal on Financial Mathematics},
  volume={14},
  number={4},
  pages={1028--1061},
  year={2023},
  publisher={SIAM}
}

@article{ata2026dynamic,
  title={Dynamic Scheduling of a Parallel-Server Queueing System: A Computational Method for High-Dimensional Problems},
  author={Ata, Baris and Kasikaralar, Ebru},
  journal={arXiv preprint arXiv:2605.09799},
  year={2026}
}

@article{ataxu2025, 
    title={Dynamic Control of Stochastic Matching Systems in Heavy-Traffic: 
    An Effective Computational Method for High-Dimensional Problems},   
    author={Ata, Barış and Xu, Yaosheng},
    year={2025},
    note={Working paper}
}

@article{atazhou2024,
    title={Analysis and Improvement of Eviction Enforcement},   
    author={Ata, Barış and Zhou, Yuwei},
    year={2024},
    note={Working paper}
}

@article{atalisi2024,
    author = {Ata, Barış and Li, Chun and Si, Nian},
    title = {Dynamic Control of a Make-to-Order Manufacturing System under Throughput Time Constraints: An effective computational method in the high-dimensional case},
    note = {Working paper},
    year = {2024} 
}

@article{bar1995explicit,
  title={Explicit solution of inventory problems with delivery lags},
  author={Bar-Ilan, Avner and Sulem, Agn{\`e}s},
  journal={Mathematics of Operations Research},
  volume={20},
  number={3},
  pages={709--720},
  year={1995},
  publisher={INFORMS}
}

@book{bertsekas2012dynamic,
  title={Dynamic Programming and Optimal Control},
  author={Bertsekas, Dimitri P},
  publisher={Athena Scientific},
  volume={II},
  edition={4th},
  year={2012}
}

@article{nahmias1994optimizing,
  title={Optimizing inventory levels in a two-echelon retailer system with partial lost sales},
  author={Nahmias, Steven and Smith, Stephen A},
  journal={Management Science},
  volume={40},
  number={5},
  pages={582--596},
  year={1994},
  publisher={INFORMS}
}

@article{ehrhardt1979power,
  title={The power approximation for computing {$(s, S)$} inventory policies},
  author={Ehrhardt, Richard},
  journal={Management Science},
  volume={25},
  number={8},
  pages={777--786},
  year={1979},
  publisher={INFORMS}
}

@article{jackson1985joint,
  title={The joint replenishment problem with a powers-of-two restriction},
  author={Jackson, Peter and Maxwell, William and Muckstadt, John},
  journal={IIE Transactions},
  volume={17},
  number={1},
  pages={25--32},
  year={1985},
  publisher={Taylor \& Francis}
}

@article{federgruen1992joint,
  title={The joint replenishment problem with general joint cost structures},
  author={Federgruen, Awi and Zheng, Yu-Sheng},
  journal={Operations Research},
  volume={40},
  number={2},
  pages={384--403},
  year={1992},
  publisher={INFORMS}
}

@article{harris1990many,
  title={How many parts to make at once},
  author={Harris, Ford W},
  journal={Operations Research},
  volume={38},
  number={6},
  pages={947--950},
  year={1990},
  publisher={INFORMS}
}

@article{roundy198598,
  title={98\%-effective integer-ratio lot-sizing for one-warehouse multi-retailer systems},
  author={Roundy, Robin},
  journal={Management Science},
  volume={31},
  number={11},
  pages={1416--1430},
  year={1985},
  publisher={INFORMS}
}

@article{roundy198698,
  title={A 98\%-effective lot-sizing rule for a multi-product, multi-stage production/inventory system},
  author={Roundy, Robin},
  journal={Mathematics of Operations Research},
  volume={11},
  number={4},
  pages={699--727},
  year={1986},
  publisher={INFORMS}
}

@article{goyal1979note,
  title={Note ON “{A} SIMPLE METHOD OF DETERMINING ORDER QUANTITIES IN JOINT REPLENISHMENTS UNDER DETERMINISTIC DEMAND”},
  author={Goyal, SK and Belton, AS},
  journal={Management Science},
  volume={25},
  number={6},
  pages={604},
  year={1979},
  publisher={Institute for Operations Research and the Management Sciences}
}

@article{viswanathan1996new,
  title={A new optimal algorithm for the joint replenishment problem},
  author={Viswanathan, S},
  journal={Journal of the Operational Research Society},
  volume={47},
  number={7},
  pages={936--944},
  year={1996},
  publisher={Taylor \& Francis}
}

@article{kaspi1991economic,
  title={On the economic ordering quantity for jointly replenished items},
  author={Kaspi, Moshe and Rosenblatt, Meir J},
  journal={The International Journal of Production Research},
  volume={29},
  number={1},
  pages={107--114},
  year={1991},
  publisher={Taylor \& Francis}
}

@article{balintfy1964basic,
  title={On a basic class of multi-item inventory problems},
  author={Balintfy, Joseph L},
  journal={Management Science},
  volume={10},
  number={2},
  pages={287--297},
  year={1964},
  publisher={INFORMS}
}

@article{azimzadeh2016weakly,
  title={Weakly chained matrices, policy iteration, and impulse control},
  author={Azimzadeh, Parsiad and Forsyth, Peter A},
  journal={SIAM Journal on Numerical Analysis},
  volume={54},
  number={3},
  pages={1341--1364},
  year={2016},
  publisher={SIAM}
}

@article{adelman2012computing,
  title={Computing near-optimal policies in generalized joint replenishment},
  author={Adelman, Daniel and Klabjan, Diego},
  journal={INFORMS Journal on Computing},
  volume={24},
  number={1},
  pages={148--164},
  year={2012},
  publisher={INFORMS}
}

@article{adelman2005duality,
  title={Duality and existence of optimal policies in generalized joint replenishment},
  author={Adelman, Daniel and Klabjan, Diego},
  journal={Mathematics of Operations Research},
  volume={30},
  number={1},
  pages={28--50},
  year={2005},
  publisher={INFORMS}
}

@book{touzi2012optimal,
  title={Optimal Stochastic Control, Stochastic Target Problems, and Backward SDE},
  author={Touzi, Nizar},
  volume={29},
  year={2012},
  publisher={Springer Science \& Business Media}
}

@article{aksoy1988multi,
  title={Multi-item inventory models with co-ordinated replenishments: A survey},
  author={Aksoy, Yasemin and Erenguc, Selcuk},
  journal={International Journal of Operations \& Production Management},
  volume={8},
  number={1},
  pages={63--73},
  year={1988},
  publisher={MCB UP Ltd}
}

@article{ozkaya2006stochastic,
  title={The stochastic joint replenishment problem: a new policy, analysis, and insights},
  author={{\"O}zkaya, Banu Y{\"u}ksel and G{\"u}rler, {\"U}lk{\"u} and Berk, Emre},
  journal={Naval Research Logistics (NRL)},
  volume={53},
  number={6},
  pages={525--546},
  year={2006},
  publisher={Wiley Online Library}
}

@article{zheng1991finding,
  title={Finding optimal {$(s, S)$} policies is about as simple as evaluating a single policy},
  author={Zheng, Yu-Sheng and Federgruen, Awi},
  journal={Operations Research},
  volume={39},
  number={4},
  pages={654--665},
  year={1991},
  publisher={INFORMS}
}

@article{nielsen2005analytical,
  title={An analytical study of the {$Q (s, S)$} policy applied to the joint replenishment problem},
  author={Nielsen, Christina and Larsen, Christian},
  journal={European Journal of Operational Research},
  volume={163},
  number={3},
  pages={721--732},
  year={2005},
  publisher={Elsevier}
}

@article{altarovici2017optimal,
  title={Optimal consumption and investment with fixed and proportional transaction costs},
  author={Altarovici, Albert and Reppen, Max and Soner, H Mete},
  journal={SIAM Journal on Control and Optimization},
  volume={55},
  number={3},
  pages={1673--1710},
  year={2017},
  publisher={SIAM}
}

@article{liu2004optimal,
  title={Optimal consumption and investment with transaction costs and multiple risky assets},
  author={Liu, Hong},
  journal={The Journal of Finance},
  volume={59},
  number={1},
  pages={289--338},
  year={2004},
  publisher={Wiley Online Library}
}

@article{beckmann1961inventory,
  title={An inventory model for arbitrary interval and quantity distributions of demand},
  author={Beckmann, Martin},
  journal={Management Science},
  volume={8},
  number={1},
  pages={35--57},
  year={1961},
  publisher={INFORMS}
}

@article{han2017deep,
  title={Deep learning-based numerical methods for high-dimensional parabolic partial differential equations and backward stochastic differential equations},
  author={E, Weinan and Han, Jiequn and Jentzen, Arnulf},
  journal={Communications in Mathematics and Statistics},
  volume={5},
  number={4},
  pages={349--380},
  year={2017},
  publisher={Springer}
}

@article{han2020convergence,
  title={Convergence of the deep {BSDE} method for coupled {FBSDEs}},
  author={Han, Jiequn and Long, Jihao},
  journal={Probability, Uncertainty and Quantitative Risk},
  volume={5},
  number={1},
  pages={5},
  year={2020},
  publisher={Springer}
}

@article{buehler2019deep,
  title={Deep hedging},
  author={Buehler, Hans and Gonon, Lukas and Teichmann, Josef and Wood, Ben},
  journal={Quantitative Finance},
  volume={19},
  number={8},
  pages={1271--1291},
  year={2019},
  publisher={Taylor \& Francis}
}

@article{gallego2005k,
  title={K-convexity in {$\mathbb{R}^n$}},
  author={Gallego, Guillermo and Sethi, Suresh},
  journal={Journal of Optimization Theory and Applications},
  volume={127},
  number={1},
  pages={71--88},
  year={2005}
}

@article{gallego2001integrating,
  title={Integrating replenishment decisions with advance demand information},
  author={Gallego, Guillermo and \"{O}zer, \"{O}zalp},
  journal={Management Science},
  volume={47},
  number={10},
  pages={1344--1360},
  year={2001},
  publisher={INFORMS}
}

@article{veinott1965computing,
  title={Computing optimal {$(s, S)$} inventory policies},
  author={Veinott Jr, Arthur F and Wagner, Harvey M},
  journal={Management Science},
  volume={11},
  number={5},
  pages={525--552},
  year={1965},
  publisher={INFORMS}
}

@inproceedings{scarf1960,
    author = {Scarf, H},
    title = {The optimality of {$(S,s)$} policies in dynamic inventory problems},
    booktitle = {Mathematical Methods in Social Sciences},
    year = {1960},
    pages = {196--202},
    editors = {Arrow, K J and Karlin, S and Suppes P}
}

@article{atkins1988periodic,
  title={Periodic versus {“can-order”} policies for coordinated multi-item inventory systems},
  author={Atkins, Derek R and Iyogun, Paul O},
  journal={Management Science},
  volume={34},
  number={6},
  pages={791--796},
  year={1988},
  publisher={INFORMS}
}

@article{creemers2022joint,
  title={The joint replenishment problem: Optimal policy and exact evaluation method},
  author={Creemers, Stefan and Boute, Robert},
  journal={European Journal of Operational Research},
  volume={302},
  number={3},
  pages={1175--1188},
  year={2022},
  publisher={Elsevier}
}

@article{federgruen1984coordinated,
  title={Coordinated replenishments in a multi-item inventory system with compound {Poisson} demands},
  author={Federgruen, Awi and Groenevelt, H and Tijms, Hendrik Cornelis},
  journal={Management Science},
  volume={30},
  number={3},
  pages={344--357},
  year={1984},
  publisher={INFORMS}
}

@article{ignall1969optimal,
  title={Optimal continuous review policies for two product inventory systems with joint setup costs},
  author={Ignall, Edward},
  journal={Management Science},
  volume={15},
  number={5},
  pages={278--283},
  year={1969},
  publisher={INFORMS}
}

@article{goyal1989joint,
  title={Joint replenishment inventory control: Deterministic and stochastic models},
  author={Goyal, Suresh K and Satir, Ahmet T},
  journal={European Journal of Operational Research},
  volume={38},
  number={1},
  pages={2--13},
  year={1989},
  publisher={Elsevier}
}

@article{silver1965some,
  title={Some characteristics of a special joint-order inventory model},
  author={Silver, Edward A},
  journal={Operations Research},
  volume={13},
  number={2},
  pages={319--322},
  year={1965},
  publisher={INFORMS}
}

@article{silver1981establishing,
  title={Establishing reorder points in the {$(S, c, s)$} coordinated control system under compound {Poisson} demand},
  author={Silver, Edward A},
  journal={The International Journal of Production Research},
  volume={19},
  number={6},
  pages={743--750},
  year={1981},
  publisher={Taylor \& Francis}
}

@article{johnson1967optimality,
  title={Optimality and computation of {$(\sigma, S)$} policies in the multi-item infinite horizon inventory problem},
  author={Johnson, Ellis L},
  journal={Management Science},
  volume={13},
  number={7},
  pages={475--491},
  year={1967},
  publisher={INFORMS}
}

@article{kalin1980optimality,
  title={On the optimality of {$(\sigma, S)$} policies},
  author={Kalin, Dieter},
  journal={Mathematics of Operations Research},
  volume={5},
  number={2},
  pages={293--307},
  year={1980},
  publisher={INFORMS}
}

@article{porteus1971optimality,
  title={On the optimality of generalized {$(s, S)$} policies},
  author={Porteus, Evan L},
  journal={Management Science},
  volume={17},
  number={7},
  pages={411--426},
  year={1971},
  publisher={INFORMS}
}

@article{schal1976optimality,
  title={On the Optimality of {$(s,S)$}--Policies in Dynamic Inventory Models with Finite Horizon},
  author={Sch{\"a}l, Manfred},
  journal={SIAM Journal on Applied Mathematics},
  volume={30},
  number={3},
  pages={528--537},
  year={1976},
  publisher={SIAM}
}

@book{bensoussan1984impulse,
  title={Impulse Control and Quasi-Variational Inequalities},
  author={Bensoussan, Alain and Lions, Jacques Louis},
  publisher={Gauthier-Villars},
  year={1984}
}

@article{ata2024singular,
  title={Singular control of (reflected) {Brownian} motion: A computational method suitable for queueing applications},
  author={Ata, Baris and Harrison, J Michael and Si, Nian},
  journal={Queueing Systems},
  pages={1--37},
  year={2024},
  publisher={Springer}
}

@article{ata2023dynamic,
  title={Dynamic Scheduling of a Multiclass Queue in the {Halfin-Whitt} Regime: A Computational Approach for High-Dimensional Problems},
  author={Ata, Baris and Kasikaralar, Ebru},
  journal={Available at SSRN 4649551},
  year={2023}
}

@article{beck2020overview,
  title={An overview on deep learning-based approximation methods for partial differential equations},
  author={Beck, Christian and Hutzenthaler, Martin and Jentzen, Arnulf and Kuckuck, Benno},
  journal={arXiv preprint arXiv:2012.12348},
  year={2022}
}

@article{chessari2023numerical,
  title={Numerical methods for backward stochastic differential equations: A survey},
  author={Chessari, Jared and Kawai, Reiichiro and Shinozaki, Yuji and Yamada, Toshihiro},
  journal={Probability Surveys},
  volume={20},
  pages={486--567},
  year={2023},
  publisher={The Institute of Mathematical Statistics and the Bernoulli Society}
}

@article{ata2023drift,
  title={Drift Control of High-Dimensional Reflected {Brownian} Motion: A Computational Method Based on Neural Networks},
  author={Ata, Baris and Harrison, J Michael and Si, Nian},
  journal={Stochastic Systems},
  year={2024},
  publisher={INFORMS}
}

@article{azimzadeh2018convergence,
  title={Convergence of implicit schemes for {Hamilton--Jacobi--Bellman} quasi-variational inequalities},
  author={Azimzadeh, Parsiad and Bayraktar, Erhan and Labahn, George},
  journal={SIAM Journal on Control and Optimization},
  volume={56},
  number={6},
  pages={3994--4016},
  year={2018},
  publisher={SIAM}
}

@article{bezanson2017julia,
  title={Julia: A fresh approach to numerical computing},
  author={Bezanson, Jeff and Edelman, Alan and Karpinski, Stefan and Shah, Viral B},
  journal={SIAM Review},
  volume={59},
  number={1},
  pages={65--98},
  year={2017},
  publisher={SIAM}
}

@article{biagini2023neural,
  title={Neural network approximation for superhedging prices},
  author={Biagini, Francesca and Gonon, Lukas and Reitsam, Thomas},
  journal={Mathematical Finance},
  volume={33},
  number={1},
  pages={146--184},
  year={2023},
  publisher={Wiley Online Library}
}

@article{bensoussan2005optimality,
  title={Optimality of an {$(s,S)$} policy with compound {Poisson} and diffusion demands: A quasi-variational inequalities approach},
  author={Bensoussan, Alain and Liu, RH and Sethi, Suresh P},
  journal={SIAM Journal on Control and Optimization},
  volume={44},
  number={5},
  pages={1650--1676},
  year={2005},
  publisher={SIAM}
}

@article{bouchard2009stochastic,
  title={A stochastic target formulation for optimal switching problems in finite horizon},
  author={Bouchard, Bruno},
  journal={Stochastics: An International Journal of Probability and Stochastics Processes},
  volume={81},
  number={2},
  pages={171--197},
  year={2009},
  publisher={Taylor \& Francis}
}

@article{helmes2017continuous,
  title={Continuous inventory models of diffusion type: Long-term average costs criterion},
  author={Helmes, Kurt L and Stockbridge, Richard H and Zhu, Chao},
  journal={The Annals of Applied Probability},
  volume={27},
  number={3},
  pages={1831--1885},
  year={2017},
  publisher={Institute of Mathematical Statistics}
}

@article{innes:2018,
  author    = {Mike Innes},
  title     = {Flux: Elegant Machine Learning with {Julia}},
  journal   = {Journal of Open Source Software},
  year      = {2018},
  doi       = {10.21105/joss.00602},
}

@article{kharroubi2010backward,
  title={Backward {SDEs} with constrained jumps and quasi-variational inequalities},
  author={Kharroubi, Idris and Ma, Jin and Pham, Huy{\^e}n and Zhang, Jianfeng},
  year={2010},
  journal={The Annals of Probability},
  volume={38},
  issue={2},
  pages={797--840}
}

@article{li2022optimal,
  title={Optimal Ordering Policy for Two Product Inventory Models with Fixed Ordering Costs},
  author={Li, Yu and Sethi, Suresh},
  journal={Available at SSRN 4199040},
  year={2022}
}

@article{han2018solving,
  title={Solving high-dimensional partial differential equations using deep learning},
  author={Han, Jiequn and Jentzen, Arnulf and E, Weinan},
  journal={Proceedings of the National Academy of Sciences},
  volume={115},
  number={34},
  pages={8505--8510},
  year={2018},
  publisher={National Acad Sciences}
}

@book{oksendal2019stochastic,
  title={Stochastic Control of Jump Diffusions},
  author={{\O}ksendal, Bernt and Sulem, Agnes},
  booktitle={Applied Stochastic Control of Jump Diffusions},
  year={2019},
  publisher={Springer}
}

@article{perera2023survey2,
  title={A survey of stochastic inventory models with fixed costs: Optimality of {$(s, S)$} and  {$(s, S)$}-type policies—{Discrete-time} case},
  author={Perera, Sandun C and Sethi, Suresh P},
  journal={Production and Operations Management},
  volume={32},
  number={1},
  pages={131--153},
  year={2023},
  publisher={Wiley Online Library}
}

@article{perera2023survey,
  title={A survey of stochastic inventory models with fixed costs: Optimality of {$(s, S)$} and {$(s, S)$}-type policies—{Continuous-time} case},
  author={Perera, Sandun C and Sethi, Suresh P},
  journal={Production and Operations Management},
  volume={32},
  number={1},
  pages={154--169},
  year={2023},
  publisher={Wiley Online Library}
}

@article{pham2021neural,
  title={Neural networks-based backward scheme for fully nonlinear {PDEs}},
  author={Pham, Huyen and Warin, Xavier and Germain, Maximilien},
  journal={SN Partial Differential Equations and Applications},
  volume={2},
  number={1},
  pages={16},
  year={2021},
  publisher={Springer}
}

@article{germain2022approximation,
  title={Approximation error analysis of some deep backward schemes for nonlinear {PDEs}},
  author={Germain, Maximilien and Pham, Huy{\^e}n and Warin, Xavier},
  journal={SIAM Journal on Scientific Computing},
  volume={44},
  number={1},
  pages={A28--A56},
  year={2022},
  publisher={SIAM}
}

@article{sulem1986solvable,
  title={A solvable one-dimensional model of a diffusion inventory system},
  author={Sulem, Agn{\`e}s},
  journal={Mathematics of Operations Research},
  volume={11},
  number={1},
  pages={125--133},
  year={1986},
  publisher={INFORMS}
}

@article{sulem1986explicit,
  title={Explicit solution of a two-dimensional deterministic inventory problem},
  author={Sulem, Agn{\`e}s},
  journal={Mathematics of Operations Research},
  volume={11},
  number={1},
  pages={134--146},
  year={1986},
  publisher={INFORMS}
}

@article{khouja2008review,
  title={A review of the joint replenishment problem literature: 1989--2005},
  author={Khouja, Moutaz and Goyal, Suresh},
  journal={European Journal of Operational Research},
  volume={186},
  number={1},
  pages={1--16},
  year={2008},
  publisher={Elsevier}
}

@article{pantumsinchai1992comparison,
  title={A comparison of three joint ordering inventory policies},
  author={Pantumsinchai, Pricha},
  journal={Decision Sciences},
  volume={23},
  number={1},
  pages={111--127},
  year={1992},
  publisher={Wiley Online Library}
}

@article{viswanathan1997note,
  title={Note. {Periodic} review {$(s, S)$} policies for joint replenishment inventory systems},
  author={Viswanathan, S},
  journal={Management Science},
  volume={43},
  number={10},
  pages={1447--1454},
  year={1997},
  publisher={INFORMS}
}

@article{weinan2021algorithms,
  title={Algorithms for solving high dimensional {PDEs}: From nonlinear {Monte Carlo} to machine learning},
  author={E, Weinan and Han, Jiequn and Jentzen, Arnulf},
  journal={Nonlinearity},
  volume={35},
  number={1},
  pages={278},
  year={2021},
  publisher={IOP Publishing}
}

@book{yong1999stochastic,
  title={Stochastic Controls: Hamiltonian Systems and HJB Equations},
  author={Yong, Jiongmin and Zhou, Xun Yu},
  volume={43},
  year={1999},
  publisher={Springer Science \& Business Media}
}

@article{constantinides1976stochastic,
  title={Stochastic cash management with fixed and proportional transaction costs},
  author={Constantinides, George M},
  journal={Management Science},
  volume={22},
  number={12},
  pages={1320--1331},
  year={1976},
  publisher={INFORMS}
}

@article{harrison1983impulse,
  title={Impulse control of Brownian motion},
  author={Harrison, J Michael and Sellke, Thomas M and Taylor, Allison J},
  journal={Mathematics of Operations Research},
  volume={8},
  number={3},
  pages={454--466},
  year={1983},
  publisher={INFORMS}
}

@book{stokey2008economics,
  title={The Economics of Inaction: Stochastic Control models with fixed costs},
  author={Stokey, Nancy L},
  year={2008},
  publisher={Princeton University Press}
}
\bibliographystyle{apalike}
\end{singlespace}


\newpage
\begin{appendix}


\section{Formal Derivation of Stochastic Identity \eqref{eq:keyidentityprop}}\label{app:A}
{In this section, we provide a heuristic derivation of the stochastic identity \eqref{eq:keyidentityprop}. We proceed in three steps. 
First, formally applying a generalized It\^o's lemma for jump diffusions, 
we arrive at Eq. \eqref{eq:align_ito}, which expresses the value function $V(x)$ as the sum of 
the right-hand side of the inequality defining \eqref{eq:keyidentityprop} 
and two slack quantities $A$ and $B$ (to be defined below). 
Second, given \eqref{eq:align_ito}, Lemma~\ref{lemma:derivation} provides 
an equivalent reformulation of \eqref{eq:keyidentityprop} in terms of two conditions 
on the slack quantities $A$ and $B$. 
Third, we argue that both conditions are natural consequences of the HJB equation \eqref{eq:qvi}, 
leading to the desired identity \eqref{eq:keyidentityprop}.}

Recall the value function $V$ defined in Eq.~\eqref{eq:impulsevaluefunction}. $V$ is a formal solution to the HJB equation \eqref{eq:qvi}.
Let $T>0$, and let $U(\cdot;\alpha,\lambda,\Phi)$ be a reference policy defined by Eq.~\eqref{eq:uncontrolled_control} in Section~\ref{sec:SDE}. Also recall the corresponding reference process $\tilde{X}(\cdot)$ given by
\eqref{eq:uncontrolled_reference}:
\[
\tilde{X}(t)=x-\mu t-\sigma B(t)+\sum_{j:T_j\le t} Y_j.
\]
Formally applying a generalized Itô's lemma for jump diffusions to the function $(x,t)\mapsto e^{-rt}V(x)$ (see, e.g., \cite{applebaum2009levy}), we have
\begin{equation}\label{eq:generalized_ito}
\begin{aligned}
e^{-rT}V(\tilde{X}(T))-V(x)
&=\int_0^T e^{-rt}\bigl(-rV-\mathcal{L}V\bigr)(\tilde{X}(t-))dt
-\int_0^T e^{-rt}\nabla V(\tilde{X}(t-))^\top\sigma\,dB(t) \\
&\quad+\sum_{j:T_j\le T}e^{-rT_j}\Big[V(\tilde{X}(T_j))-V(\tilde{X}(T_j-))\Big].
\end{aligned}
\end{equation}
Then
\begin{equation}\label{eq:align_ito}
\begin{aligned}
V(x)
&=e^{-rT}V(\tilde{X}(T))
-\int_0^T e^{-rt}\bigl(-rV-\mathcal{L}V\bigr)(\tilde{X}(t-))dt
+\int_0^T e^{-rt}\nabla V(\tilde{X}(t-))^\top\sigma\,dB(t) \\
&\quad-\sum_{j:T_j\le T}e^{-rT_j}\Big[V(\tilde{X}(T_j))-V(\tilde{X}(T_j-))\Big] \\
&=e^{-rT}V(\tilde{X}(T))
+\int_0^T e^{-rt}\bigl[(\mathcal{L}+r)V-f\bigr](\tilde{X}(t-))dt
+\int_0^T e^{-rt}f(\tilde{X}(t-))dt \\
&\quad+\int_0^T e^{-rt}\nabla V(\tilde{X}(t-))^\top\sigma\,dB(t) \\
&\quad-\sum_{j:T_j\le T}e^{-rT_j}\Big[V\big(\tilde{X}(T_j-)+Y_j\big)+c(Y_j)-V(\tilde{X}(T_j-))\Big]
+\sum_{j:T_j\le T}e^{-rT_j}c(Y_j) \\
&=e^{-rT}V(\tilde{X}(T))
+\int_0^T e^{-rt}f(\tilde{X}(t-))dt
+\int_0^T e^{-rt}\nabla V(\tilde{X}(t-))^\top\sigma\,dB(t)
+\sum_{j:T_j\le T}e^{-rT_j}c(Y_j) \\
&\quad+\int_0^T e^{-rt}\bigl[(\mathcal{L}+r)V-f\bigr](\tilde{X}(t-))dt\\
&\quad+\sum_{j:T_j\le T}e^{-rT_j}\Big[V(\tilde{X}(T_j-))-\big(V(\tilde{X}(T_j-)+Y_j)+c(Y_j)\big)\Big].
\end{aligned}
\end{equation}
Consider {the last two terms on the right-hand side} of \eqref{eq:align_ito}. For notational convenience, write
\begin{equation}
A(\tilde{X}):=\int_{0}^{T}e^{-rt}[(\mathcal{L}+r)V-f](\tilde{X}(t-))dt,
\label{eq:A}
\end{equation}
and
\begin{equation}
B(\tilde{X}):=\sum_{j:T_j \le T}e^{-rT_{j}}[V(\tilde{X}(T_{j}-))-(V(\tilde{X}(T_{j}-)+Y_{j}) + c(Y_{j}))],
\label{eq:B}
\end{equation}
where we include the argument $\tilde{X}$ to emphasize the dependence of $A$ and $B$ on the reference process $\tilde{X}$.
The following lemma provides a tight characterization of identity \eqref{eq:keyidentityprop} in terms of $A$ and $B$.
\begin{lemma}\label{lemma:derivation}
Let $x \in \mathbb{R}^d$. Suppose \eqref{eq:align_ito} holds a.s. Then, the stochastic identity \eqref{eq:keyidentityprop} holds if and only if the following is true:
\begin{enumerate}
    \item[\textnormal{(i)}] $A(\tilde{X})+B(\tilde{X})\le0$ a.s.; and
    \item[\textnormal{(ii)}] for any $w>0$, $\P(A(\tilde{X})+B(\tilde{X})>-w)>0$.
\end{enumerate}
\end{lemma}

We defer the proof of Lemma~\ref{lemma:derivation} to the end of the section. 
{By Lemma~\ref{lemma:derivation}, 
if conditions (i) and (ii) hold, then \eqref{eq:keyidentityprop} follows. Thus, 
it remains to justify conditions (i) and (ii). 
To do so, we establish (i) directly from the HJB equation \eqref{eq:qvi},
and then provide a plausibility argument for (ii).} 

By \eqref{eq:qvi}, the following pointwise inequalities hold for all $x \in \mathbb{R}^d$:
$$
(\mathcal{L}+r)V(x) - f(x) \le 0, 
\quad\text{and}\quad 
V(x) - \inf_{y \in \mathbb{R}_{+}^{d}}\bigl(V(x+y) + c(y)\bigr) \le 0.
$$
Applying the first inequality pointwise to $\tilde{X}(t-)$ and integrating against 
the nonnegative measure $e^{-rt}\,dt$ yields $A(\tilde{X}) \le 0$ a.s. 
For the second inequality, since each order $Y_j \in \mathbb{R}_{+}^{d}$, we have
$$
V(\tilde{X}(T_j-)) - \bigl(V(\tilde{X}(T_j-)+Y_j) + c(Y_j)\bigr) 
\le V(\tilde{X}(T_j-)) - \inf_{y \in \mathbb{R}_{+}^{d}}\bigl(V(\tilde{X}(T_j-)+y) + c(y)\bigr) \le 0, \quad \text{a.s.}
$$
Summing with nonnegative factors $e^{-rT_j}$ yields $B(\tilde{X}) \le 0$ a.s. 
Therefore $A(\tilde{X}) + B(\tilde{X}) \le 0$ a.s., establishing (i).

Next, we provide a plausibility argument for (ii). 
Suppose that an optimal control exists, and let $X^*$ denote the corresponding 
state process. One expects, under standard verification-theorem reasoning, 
that each optimal intervention $Y_j^*$ achieves the infimum in the intervention term of \eqref{eq:qvi}, so that 
$$
V(X^*(T_j-)) - \bigl(V(X^*(T_j-) + Y_j^*) + c(Y_j^*)\bigr) = 0, \quad \text{a.s.}
$$
for each $j$, yielding $B(X^*) = 0$ a.s. One also expects that $X^*(t-)$ 
lies in the no-action region $\{x : (\mathcal{L}+r)V(x) - f(x) = 0\}$ 
for all $t > 0$, a.s., yielding $A(X^*) = 0$ a.s.

Granted $A(X^{*}) = B(X^{*}) = 0$ a.s., 
one can argue $\P(A(\tilde{X}) + B(\tilde{X}) > -w) > 0$ for every $w > 0$ 
via a change-of-measure argument that tilts the reference process toward behaving like an optimal one. 
We do not pursue this here; arguments of this flavor appear in \cite{bouchard2009stochastic} and \cite{kharroubi2010backward}.

\begin{remark} \rm
Loosely speaking, Eq.~\eqref{eq:align_ito} takes the form of a BSDE with constrained jumps treated in \cite{kharroubi2010backward}. Indeed, writing 
$$
Y_t = e^{-rt}V(\tilde{X}(t)) \quad\text{and}\quad Z_t = e^{-rt} \sigma^\top\nabla V(\tilde{X}(t-)),
$$
we may rearrange Eq.~\eqref{eq:align_ito} to obtain
\begin{equation}
Y_0 = Y_T + \int_0^T e^{-rt}f(\tilde{X}(t-))\,dt + \sum_{j:T_j \le T} e^{-rT_j}c(Y_j) + \int_0^T Z_t^\top dB(t) + (A + B),
\label{eq:bsde}
\end{equation}
where $A$ and $B$ are defined in Eqs.~\eqref{eq:A} 
and \eqref{eq:B}, respectively. 
Eq. \eqref{eq:bsde} parallels the constrained BSDE~(1.9) in \cite{kharroubi2010backward}.
The sign of our stochastic integral differs from theirs, owing to our convention that Brownian motion drives cumulative demand. 
Moreover, since we minimize cost rather than maximize reward, when $A$ and $B$ are viewed as stochastic processes indexed by the horizon $T$, $-(A+B)$ plays the role of the nondecreasing process $\bar{K}$ in the framework of \cite{kharroubi2010backward}.
Indeed, as can be seen from the proof of Lemma~\ref{lemma:derivation}, the HJB equation implies that the integrand defining $A$ and the summands defining $B$ are nonpositive, ensuring that $-(A+B)$ is nondecreasing in the time horizon.
Condition~(ii) of Lemma~\ref{lemma:derivation} mirrors the minimality requirement for the constrained BSDE in \cite{kharroubi2010backward}.
The stochastic target formulation (cf.\ Eq.~(1.10) in \citealp{kharroubi2010backward}) provides an equivalent characterization of $V(x)$ that motivates our computational approach.
\end{remark}

\begin{proof}[Proof of Lemma~\ref{lemma:derivation}]
Fix $x$ and let
\begin{align*}
S := \bigg\{ v\in \mathbb{R} \, : \, & v-\int_{0}^{T}e^{-rt}\nabla V(\tilde{X}(t))^{\top}\sigma dB(t) \\
& \le e^{-rT}V(\tilde{X}(T))+\int_{0}^{T}e^{-rt}f(\tilde{X}(t))dt+\sum_{j:T_{j}\le T}c(Y_{j})e^{-rT_{j}},\ \textnormal{a.s.} \bigg\}
\end{align*}
Then, identity \eqref{eq:keyidentityprop} is equivalent to $V(x)=\sup S$. 
Using Eq.~\eqref{eq:align_ito}, the set $S$ admits an equivalent representation. 
Rearranging the inequality that defines membership in $S$ and substituting from  \eqref{eq:align_ito}, we obtain
$$
S = \{v \in \mathbb{R} : v \leq V(x) - A(\tilde{X}) - B(\tilde{X}) \text{ a.s.}\}.
$$
We claim that $\sup S = V(x)$ if and only if conditions (i) and (ii) in the lemma statement hold.
Indeed, for any random variable $W$, we have $\sup\{v \in \mathbb{R} : v \le W \text{ a.s.}\} = \operatorname{ess\,inf}\{W\}$, and hence
$$
\sup S = \operatorname{ess\,inf}\{V(x) - A - B\} = V(x) - \operatorname{ess\,sup}\{A + B\}.
$$
Thus, it follows that $\sup S = V(x)$ if and only if $\operatorname{ess\,sup}\{A+B\} = 0$. 
The latter equality is equivalent to the two inequalities: 
$\operatorname{ess\,sup}\{A+B\} \le 0$ and $\operatorname{ess\,sup}\{A+B\} \ge 0$. 
By definition of the essential supremum, the first inequality is equivalent to 
$A + B \le 0$ a.s., which is condition (i) of the lemma. 
The second inequality says that $A + B$ is not essentially bounded above by 
any negative number; equivalently, for every $w > 0$, the event $\{A + B > -w\}$ 
has positive probability, which is condition (ii) of the lemma.

\end{proof}

\section{Implementation Details}\label{app:B}
\subsection{Materials and Method Modifications}\label{app:B_material}
We implement two deep neural networks, $H(\cdot;\theta)$ and $G(\cdot;\vartheta)$, both consisting of {three to five} (hidden) layers. The neural networks are implemented using the \texttt{Flux.jl} package \citep{innes:2018} in the \texttt{Julia 1.11} computing language \citep{bezanson2017julia}. Currently, we conduct the numerical experiments
on a high-performance computing node equipped with an
NVIDIA H100 GPU with 80 GB memory and an AMD EPYC 9334 32-core processor, allowing up to 64 logical threads, with
694 GB of usable RAM.

{\bf Neural network architecture.} {As mentioned, our neural networks are fully connected with three to five hidden layers, using \texttt{elu} activation functions at the nodes.}
We initialize the weights with random numbers drawn from a uniform distribution using Kaiming (He) initialization \citep{he2015delving}.

{\bf Data generation.} The time horizon $T$ is tuned for each test problem, 
and we choose {the number of time intervals} $N$ large enough so that the time step satisfies $\Delta t\le 0.005$, 
which our empirical tests indicate is sufficient for accurate results. 
The tuning of the time horizon $T$ is governed by the following considerations.
In our experiments, we observe that a smaller $T$ improves training stability.
However, if $T$ is too small, an insufficient number of ordering events are observed per sample path,
degrading training quality.
A secondary consideration is that, for a fixed number of time intervals $N$,
decreasing $T$ reduces the time discretization step $\Delta t$, leading to more faithful simulation of the reference process.
In practice, we balance these considerations by choosing $T$
large enough to capture a meaningful number of ordering events
while keeping training stable.

The batch size refers to the number of sample paths generated. 
The number of iterations equals the number of epochs, as for each iteration, the entire data set, {i.e., output from Subroutine \ref{algo:euler} over all sample paths in the batch,} 
is generated and used during training. 
See the tables in Section~\ref{app:B_configs} below for more details.

{\bf Reference policy parameters.} We choose $\nu$ and $\alpha$ via extensive hyperparameter tuning. We set $\lambda$ to the average ordering rate of the {benchmark policy whose $S$ vector is used to set $\E[Z]=\psi S$.} 

{\bf NN optimizer.} The neural network parameters are updated using the \texttt{Adam} optimizer. {The joint penalty-rate and learning-rate schedules are detailed below.}

{\bf Cost scaling.} It is easy to show, under the assumed cost structure, that the impulse-control problem is invariant under any positive scaling of the cost parameters. For $\kappa>0$, consider scaled costs ${\kappa c_0,\kappa c,\kappa h,\kappa p}$. Then, by linearity, the inventory-state cost is $\kappa f(x)$ and the ordering cost is $\kappa c(y)$. For any admissible policy $u$,
$$
{J_\kappa(x,u)} = \mathbb{E}_{x}\left[\int_{0}^{\infty}e^{-rt}\kappa f(X^{u}(t))dt+\sum_{j}e^{-r\tau_{j}}\kappa c(y^{(j)})\right] = \kappa {J(x,u)}.
$$
Hence
$$
{V_\kappa(x) = \inf_{u\in\mathcal{U}} J_\kappa(x,u) = \inf_{u\in\mathcal{U}} \kappa J(x,u) = \kappa \inf_{u\in\mathcal{U}} J(x,u) = \kappa V(x)}.
$$
Thus, scaling all costs by $\kappa$ multiplies the value function by $\kappa$, and allows us to use cost rescaling during neural network training for the sake of numerical stability, without loss of generality. 
{For our test problems, we choose a suitable scaling factor $\kappa$; see below for details.}

\paragraph{Numerical optimizer.} 
{To compute the order-up-to vector $z^*$ used in our policy, we adopt one of two approaches. The first one solves the minimization problem in Eq.~(22) directly. The second approach uses the first-order condition in Eq.~(25). To solve the minimization problem in Eq.~(22), we use the implementation of L-BFGS \citep{nocedal2006numerical} in \texttt{LBFGS.jl}. The stationary point of Eq.~(25) is found with the same routine, applied to the least-squares objective $\tfrac12\|G_\vartheta(z)+c\|_2^2$.

For both methods, we restrict the search for $z^*$ to a compact box of the form $[\underline{z},\overline{z}] := \{z\in\mathbb{R}_+^d:\underline{z}\le z\le\overline{z}\}$. The box constraints $[\underline{z},\overline{z}]$ are chosen to address two competing considerations: On the one hand, the search should be confined to regions of the state space that were sufficiently explored during training, where the neural network approximations are expected to be reliable. On the other hand, the box should be large enough to contain a near-optimal order-up-to vector. Accordingly, $\underline{z}$ and $\overline{z}$ are each chosen to be a constant multiple of {$\psi S$}, where {$S$ is the benchmark order-up-to vector and $\psi\ge1$ is the aforementioned tuning parameter}. The initial point for the optimization routine is also chosen to be a constant multiple of {$\psi S$}; details can be found in Appendix~B.
As explained in Section 5, 
finding the order-up-to vector $z^*$ is computationally cheap, as it only needs to be performed once at the start of a performance simulation. 

}

{\bf Joint penalty-rate and learning-rate schedules.}
We provide further details on the joint penalty rate and learning rate schedules, 
which are described at a higher level in Section~\ref{sec:algorithm}. 
As explained therein, neural network training requires jointly increasing the penalty rate $\beta$ 
and decreasing the learning rate $\eta$ over the course of the loss minimization.
Both the penalty rate and learning rate are updated according to a piecewise constant schedule, 
where they are each held fixed for a prescribed number of iterations 
and then updated to their next values, possibly asynchronously. 

Figure~\ref{fig:joint_schedule} illustrates this procedure on a two-dimensional grid, 
where each point corresponds to a $(\beta, \eta)$ configuration, 
and each segment corresponds to an update in either $\beta$ or $\eta$ (or both)\footnote{For clarity within this appendix, $\beta_1, \ldots, \beta_M$ and $\eta_1, \ldots, \eta_K$ denote the \emph{distinct} values taken by the piecewise-constant penalty rate and learning rate schedules, so $M$ and $K$ here count distinct levels rather than the total number of iterations and the batch size as in Algorithm~\ref{algo}.}.
The training begins at the initial configuration $(\beta_1, \eta_1)$ 
(low penalty, high learning rate) and must reach the final configuration 
$(\beta_M, \eta_K)$ (high penalty, low learning rate) 
via a sequence of moves that increase $\beta$ and/or decrease $\eta$. 
Each move is one of three types, serving different purposes: 
A \emph{horizontal} move increases $\beta$ only to more aggressively penalize constraint violations while keeping the learning rate unchanged; 
a \emph{vertical} move decreases $\eta$ only to stabilize 
training under a fixed penalty; 
and a \emph{diagonal} move updates both simultaneously, 
more aggressively penalizing constraint violations while stabilizing training.
A joint schedule is thus specified by both the sequence of moves taken through 
the grid and the iteration indices at which transitions occur.

Two extreme paths trace the boundary of the grid:
\begin{itemize}
\item[(a)] Increase $\beta$ to its final value while holding $\eta$ fixed (all horizontal moves), 
then decrease $\eta$ (all vertical moves). This may result in unstable training 
because high penalty rates combined with a high learning rate may produce excessively 
large gradient updates, causing the loss to oscillate or diverge. 
\item[(b)] Decrease $\eta$ to its final value while holding $\beta$ fixed (all vertical moves), 
then increase $\beta$ (all horizontal moves). This may result in inefficient training, 
because, with a low penalty rate, violations of the almost-sure inequality are 
insufficiently penalized during early training, and by the time $\beta$ is increased, 
the learning rate is too small to correct the accumulated errors.
\end{itemize}
A more robust approach is to interleave increases in $\beta$ with decreases in $\eta$ 
via a combination of horizontal, vertical, and diagonal moves. 
Because we aim to quickly bring the penalty rate to its final value while maintaining training stability, 
we typically start with horizontal and diagonal moves, reserving vertical moves toward the end of 
training when the penalty rate is close to or has reached its final value.

For each group of problems (distinguished by a combination of problem dimension and demand variability), 
the specific path and transition iterations are determined empirically in two phases. 
In the first phase, we focus on a single problem instance 
and move to the next $(\beta, \eta)$ configuration each time the loss stabilizes, 
using horizontal and diagonal moves to raise the penalty rate toward its final value. 
Once the penalty rate has reached its final value, we continue training and 
decrease the learning rate (via vertical moves) as needed to stabilize the loss. 
If the loss fails to stabilize at any stage, the schedule is adjusted and the process is repeated.
We then apply the resulting stable schedule to other problem instances in the same group; 
if training remains stable and produces competitive policies, the schedule is retained.
The resulting schedules are reported in Table~\ref{tab:hyperparams}.

\begin{figure}[H]
\centering
\begin{tikzpicture}[>=stealth, scale=1.0]
  \foreach \x in {1,2,3,4,5,6} {
    \foreach \y in {1,2,3,4,5} {
      \fill[gray!40] (\x,\y) circle (1.5pt);
    }
  }

  \draw[->, thick] (0.3,0.3) -- (7.2,0.3) 
    node[right, font=\small] {$\beta$ (increasing)};
  \draw[->, thick] (0.3,0.3) -- (0.3,5.8) 
    node[above, font=\small] {$\eta$ (increasing)};

  \foreach \x/\lab in {1/$\beta_1$, 2/$\beta_2$, 3/$\beta_3$, 
                        4/$\beta_4$, 5/$\beta_5$, 6/$\beta_M$} {
    \node[below, font=\footnotesize] at (\x,0.3) {\lab};
  }
  \foreach \y/\lab in {1/$\eta_K$, 2/$\eta_4$, 3/$\eta_3$, 
                        4/$\eta_2$, 5/$\eta_1$} {
    \node[left, font=\footnotesize] at (0.3,\y) {\lab};
  }

  \fill[black] (1,5) circle (3pt);
  \node[above right, font=\footnotesize] at (.3,5.1) {Start $(\beta_1, \eta_1)$};
  \fill[black] (6,1) circle (3pt);
  \node[right, font=\footnotesize] at (6.2,1) {End $(\beta_M, \eta_K)$};

  \draw[very thick, red!70, dashed] 
    (1,5) -- (6,5);
  \draw[->, very thick, red!70, dashed] 
    (6,5) -- (6,1.15);
  \node[red!70, above, font=\footnotesize] at (4.5,5.2) {(a) unstable schedule};

  \draw[very thick, orange!70, dashed] 
    (1,5) -- (1,1);
  \draw[->, very thick, orange!70, dashed] 
    (1,1) -- (5.85,1);
  \node[orange!70, above, font=\footnotesize] at (3.5,1.15) {(b) unstable schedule};

  \draw[->, very thick, blue!80] 
    (1,5) -- (2,5) -- (2,4) -- (3,4) -- (4,3) -- (5,3) -- (5,2) -- (6,2) -- (6,1);
  \node[blue!80, above right, font=\footnotesize] at (2.8,3.3) {Stable schedule};
\end{tikzpicture}
\caption{Schematic of the joint penalty-rate and learning-rate schedule. 
Each grid point represents a $(\beta, \eta)$ configuration. 
Training progresses from $(\beta_1, \eta_1)$ to $(\beta_M, \eta_K)$. 
Horizontal segments correspond to penalty-only updates, vertical segments 
to learning-rate-only updates, and diagonal segments to simultaneous updates. 
Extreme paths (a) and~(b) traverse the grid boundary and may lead to unstable 
or inefficient training; a balanced path through the interior 
yields a more stable schedule.}
\label{fig:joint_schedule}
\end{figure}

\subsection{Hyperparameter Configurations}\label{app:B_configs}
{This section presents the hyperparameter values used during neural network training. Specifically, Table~\ref{tab:hyperparams} reports the values common across all test problems, while Table~\ref{tab:hyperparams-2d} reports the remaining hyperparameter values for the 2-dimensional test problems. Similarly, Tables \ref{tab:hyperparams-12d-low-cv}, \ref{tab:hyperparams-12d-med-cv}, and \ref{tab:hyperparams-12d-high-cv} report the remaining hyperparameter values for the low-, medium-, and high-variability cases, respectively, of the 12-dimensional test problems. Finally, Table~\ref{tab:hyperparams-50d} reports the remaining hyperparameter values for the 50-dimensional test problems.
The last columns of Tables~\ref{tab:hyperparams-2d}--\ref{tab:hyperparams-50d} report the performance simulation times. All simulations were run on a single thread. Multithreading could substantially reduce these times. However, because the number of available CPUs varied over time in our shared computing environment, we avoided multithreading to ensure consistent reporting.
}

\begin{table}[H]
    \centering
    \caption{Common hyperparameter configurations for neural network training outlined per group of test problems (the final penalty rate, $\beta_{\text{final}}$, is tuned for each test problem separately; {see Tables \ref{tab:hyperparams-12d-low-cv} - \ref{tab:hyperparams-50d} for details}).}
    \label{tab:hyperparams}
    \resizebox{\textwidth}{!}{
    \begin{tabular}{@{}lcccc@{}}
        \toprule
        {Hyperparameter} & 2-Dimensional & {12-D (Low \& Medium CV)} & {12-D (High CV)} & {50-Dimensional} \\
        \midrule
        Number of hidden layers & 5 & 3 & 3 & 4 \\
        Neurons per layer & 100 & 1,000 & 1,000 & 1,000 \\
        Activation function & \texttt{elu} & \texttt{elu} & \texttt{elu} & \texttt{elu} \\
        \midrule
        Batch size ($K$)      & 2,500 & 5,000 & 5,000 & 5,000 \\
        Time intervals ($N$)  & 100 & 100 & 100 & 100 \\
        Number of iterations     & 25,000 & 40,000 & 40,000 & 40,000 \\
        Scaling $\kappa$ & 0.1 & 0.1 & 0.1 & 0.01 \\
        Sampling distribution $\Phi$ & Lognormal & Lognormal & Lognormal & Lognormal \\
        \midrule
        Learning rate (iteration range) & 
        \makecell[l]{$10^{-3}$ (1--10,000) \\ $10^{-4}$ (10,000--15,000) \\ $10^{-5}$ (15,000--20,000) \\ $10^{-6}$ (20,000--25,000) \\ $ $} &
        \makecell[l]{$10^{-3}$ (1--5,000) \\ $5 \cdot 10^{-4}$ (5,000--10,000) \\ $10^{-4}$ (10,000--15,000) \\ $10^{-5}$ (15,000--20,000) \\ $10^{-6}$ (20,000--40,000)} &
        \makecell[l]{$10^{-3}$ (1--5,000) \\ $2 \cdot 10^{-4}$ (5,000--10,000) \\ $10^{-4}$ (10,000--20,000) \\ $10^{-5}$ (20,000--30,000) \\ $10^{-6}$ (30,000--40,000)} &
        \makecell[l]{$10^{-3}$ (1--5,000) \\ $2 \cdot 10^{-4}$ (5,000--10,000) \\ $10^{-4}$ (10,000--20,000) \\ $10^{-5}$ (20,000--30,000) \\ $10^{-6}$ (30,000--40,000)} \\
        \midrule
        Penalty rate (iteration range) &
        \makecell[l]{$10^0$ (1--2,500) \\ $10^1$ (2,500--5,000) \\ $10^2$ (5,000--7,500) \\ $10^3$ (7,500--10,000) \\ $10^4$ (10,000--15,000) \\ $10^5$ (15,000--20,000) \\ $\beta_{\text{final}}$ (20,000--40,000)} &
        \makecell[l]{$10^0$ (1--2,500) \\ $10^1$ (2,500--5,000) \\ $10^2$ (5,000--7,500) \\ $10^3$ (7,500--10,000) \\ $10^4$ (10,000--15,000) \\ $10^5$ (15,000--20,000) \\ $\beta_{\text{final}}$ (20,000--40,000)} &
        \makecell[l]{$10^0$ (1--2,500) \\ $10^1$ (2,500--5,000) \\ $10^2$ (5,000--7,500) \\ $10^3$ (7,500--10,000) \\ $10^4$ (10,000--20,000) \\ $\beta_{\text{final}}$ (20,000--40,000) \\ $ $} &
        \makecell[l]{$10^0$ (1--2,500) \\ $10^1$ (2,500--5,000) \\ $10^2$ (5,000--7,500) \\ $10^3$ (7,500--10,000) \\ $10^4$ (10,000--20,000) \\ $\beta_{\text{final}}$ (20,000--40,000) \\ $ $} \\
        \bottomrule
    \end{tabular}
    }
\end{table}

\begin{table}[H]
    \centering
    \caption{Problem-specific hyperparameters for 2-dimensional test problems ({$\mathbb{E}[Z]$ is set to $\psi$ times the $S$ vector of one of the benchmark policies (with $\psi = 1$ unless noted), which also serves as the optimization starting point;} bounds are scaled by $\mathbb{E}[Z]$ componentwise; training and simulation times reported in minutes).}
    \label{tab:hyperparams-2d}
    \resizebox{0.95\textwidth}{!}{%
    \begin{tabular}{@{}lccccccccccc@{}}
        \toprule
        Test Instance & $T$ & $\lambda$ & $\nu$ & $\alpha$ & $\beta_{\text{final}}$ & $\mathbb{E}[Z]$ & $\epsilon$ & Opt. Method & Opt. Bounds & Train & Sim \\
        \midrule
        Base Case & 0.20 & 1.21 & 0.40 & 0.00 & $10^{7}$ & Can-order & $-2.50$ & Eq.~\eqref{eq:argminimpulse} & [0.00, 1.50] & 10.81 & 6.81 \\
        Low CV    & 0.10 & 1.14 & 0.05 & 0.00 & $10^{6}$ & $(Q,S)$ & $-1.25$ & Eq.~\eqref{eq:argminimpulse} & [0.00, 1.50] & 10.25 & 6.82 \\
        High CV   & 0.05 & 1.18 & 1.00 & 0.00 & $10^{6}$ & MDP & $-1.00$ & Eq.~\eqref{eq:stationary_conditionZ} & [0.90, 1.10] & 10.86 & 4.51 \\
        $c_0=20$  & 0.05 & 1.79 & 0.20 & 0.00 & $10^{7}$ & Can-order & $-2.50$ & Eq.~\eqref{eq:argminimpulse} & [0.00, 1.50] & 10.20 & 6.79 \\
        $c_0=100$ & 0.10 & 0.83 & 0.40 & 0.00 & $10^{5}$ & Can-order & $-2.50$ & Eq.~\eqref{eq:argminimpulse} & [0.00, 1.50] & 11.07 & 6.92 \\
        $p=10$    & 0.20 & 1.10 & 0.10 & 0.00 & $10^{7}$ & Can-order & $-2.50$ & Eq.~\eqref{eq:argminimpulse} & [0.00, 1.50] & 10.18 & 6.70 \\
        $p=100$   & 0.05 & 1.13 & 0.40 & 0.00 & $10^{7}$ & Can-order & $-1.25$ & Eq.~\eqref{eq:argminimpulse} & [0.00, 1.25] & 10.15 & 6.68 \\
        \bottomrule
    \end{tabular}
    }
\end{table}

\begin{table}[H]
    \centering
    \caption{Problem-specific hyperparameters for 12-dimensional test problems with Poisson demand ({$\mathbb{E}[Z]$ is set to $\psi$ times the $S$ vector of one of the benchmark policies (with $\psi = 1$ unless noted), which also serves as the optimization starting point;} bounds are scaled by $\mathbb{E}[Z]$ componentwise; training and simulation times reported in minutes).}
    \label{tab:hyperparams-12d-low-cv}
    \resizebox{0.95\textwidth}{!}{%
    \begin{tabular}{@{}ccccccccccccc@{}}
        \toprule
        $c_0$ & $p$ & $T$ & $\lambda$ & $\nu$ & $\alpha$ & $\beta_{\text{final}}$ & $\mathbb{E}[Z]$ & $\epsilon$ & Opt. Method & Opt. Bounds & Train & Sim \\
        \midrule
        20  & 10  & 0.10 & 3.93 & 0.80 & 0.40 & $10^{7}$ & $(Q,S)$ & $-7.50$  & Eq.~\eqref{eq:stationary_conditionZ} & [0.00, 2.00] & 148.24 & 385.56 \\
            & 50  & 0.10 & 4.75 & 0.80 & 0.00 & $10^{6}$ & $(Q,S)$ & $-7.50$  & Eq.~\eqref{eq:stationary_conditionZ} & [0.00, 2.00] & 148.21 & 386.72 \\
            & 100 & 0.10 & 4.75 & 0.80 & 0.20 & $10^{6}$ & $(Q,S)$ & $-7.50$  & Eq.~\eqref{eq:stationary_conditionZ} & [0.00, 2.00] & 148.40 & 386.21 \\
        100 & 10  & 0.10 & 1.73 & 0.40 & 0.00 & $10^{6}$ & $(Q,S)$ & $-20.00$ & Eq.~\eqref{eq:stationary_conditionZ} & [0.00, 2.00] & 148.56 & 392.44 \\
            & 50  & 0.10 & 2.05 & 0.20 & 0.20 & $10^{6}$ & $(Q,S)$ & $-20.00$ & Eq.~\eqref{eq:stationary_conditionZ} & [0.00, 2.00] & 148.60 & 388.06 \\
            & 100 & 0.10 & 2.08 & 0.40 & 0.00 & $10^{6}$ & $(Q,S)$ & $-20.00$ & Eq.~\eqref{eq:stationary_conditionZ} & [0.00, 2.00] & 148.41 & 391.87 \\
        200 & 10  & 0.10 & 1.21 & 0.80 & 0.20 & $10^{6}$ & $(Q,S)$ & $-25.00$ & Eq.~\eqref{eq:stationary_conditionZ} & [0.00, 2.00] & 148.50 & 387.72 \\
            & 50  & 0.10 & 1.35 & 1.00 & 0.00 & $10^{6}$ & $(Q,S)$ & $-25.00$ & Eq.~\eqref{eq:stationary_conditionZ} & [0.00, 2.00] & 146.98 & 373.96 \\
            & 100 & 0.10 & 1.40 & 0.10 & 0.00 & $10^{6}$ & $(Q,S)$ & $-40.00$ & Eq.~\eqref{eq:stationary_conditionZ} & [0.00, 2.00] & 148.07 & 369.68 \\
        \bottomrule
    \end{tabular}
    }
\end{table}

\begin{table}[H]
    \centering
    \caption{Problem-specific hyperparameters for 12-dimensional test problems with negative binomial demand ($\operatorname{CV}=0.5$) ({$\mathbb{E}[Z]$ is set to $\psi$ times the $S$ vector of one of the benchmark policies (with $\psi = 1$ unless noted), which also serves as the optimization starting point;} bounds are scaled by $\mathbb{E}[Z]$ componentwise; training and simulation times reported in minutes).}
    \label{tab:hyperparams-12d-med-cv}
    \resizebox{0.95\textwidth}{!}{%
    \begin{tabular}{@{}ccccccccccccc@{}}
        \toprule
        $c_0$ & $p$ & $T$ & $\lambda$ & $\nu$ & $\alpha$ & $\beta_{\text{final}}$ & $\mathbb{E}[Z]$ & $\epsilon$ & Opt. Method & Opt. Bounds & Train & Sim \\
        \midrule
        20  & 10  & 0.10 & 5.74 & 0.10 & 0.00 & $10^{6}$ & $(Q,S)$   & $-15.00$ & Eq.~\eqref{eq:stationary_conditionZ} & [0.00, 2.00] & 148.43 & 370.94 \\
            & 50  & 0.10 & 6.87 & 0.20 & 0.40 & $10^{6}$ & Can-order & $-12.50$ & Eq.~\eqref{eq:stationary_conditionZ} & [0.00, 2.00] & 148.49 & 363.19 \\
            & 100 & 0.10 & 6.94 & 0.10 & 0.20 & $10^{6}$ & Can-order & $-12.50$ & Eq.~\eqref{eq:stationary_conditionZ} & [0.00, 2.00] & 148.15 & 366.50 \\
        100 & 10  & 0.10 & 2.04 & 0.80 & 0.00 & $10^{6}$ & $(Q,S)$   & $-25.00$ & Eq.~\eqref{eq:stationary_conditionZ} & [0.00, 2.00] & 148.12 & 362.82 \\
            & 50  & 0.10 & 2.88 & 0.40 & 0.40 & $10^{6}$ & Can-order & $-30.00$ & Eq.~\eqref{eq:stationary_conditionZ} & [0.00, 2.00] & 148.34 & 388.84 \\
            & 100 & 0.10 & 2.65 & 0.80 & 0.40 & $10^{6}$ & Can-order & $-30.00$ & Eq.~\eqref{eq:stationary_conditionZ} & [0.00, 2.00] & 148.45 & 389.29 \\
        200 & 10  & 0.10 & 1.37 & 0.40 & 0.20 & $10^{6}$ & $(Q,S)$   & $-40.00$ & Eq.~\eqref{eq:stationary_conditionZ} & [0.00, 2.00] & 148.28 & 384.80 \\
            & 50  & 0.10 & 1.66 & 1.00 & 0.40 & $10^{6}$ & $(Q,S)$   & $-30.00$ & Eq.~\eqref{eq:stationary_conditionZ} & [0.00, 2.00] & 147.96 & 389.94 \\
            & 100 & 0.10 & 1.95 & 0.80 & 0.40 & $10^{6}$ & Can-order & $-30.00$ & Eq.~\eqref{eq:stationary_conditionZ} & [0.00, 2.00] & 147.67 & 384.54 \\
        \bottomrule
    \end{tabular}
    }
\end{table}

\begin{table}[H]
    \centering
    \caption{Problem-specific hyperparameters for 12-dimensional test problems with negative binomial demand ($\operatorname{CV}=1.0$) ({$\mathbb{E}[Z]$ is set to $\psi$ times the $S$ vector of one of the benchmark policies (with $\psi = 1$ unless noted), which also serves as the optimization starting point;} bounds are scaled by $\mathbb{E}[Z]$ componentwise; training and simulation times reported in minutes).}
    \label{tab:hyperparams-12d-high-cv}
    \resizebox{0.95\textwidth}{!}{%
    \begin{tabular}{@{}ccccccccccccc@{}}
        \toprule
        $c_0$ & $p$ & $T$ & $\lambda$ & $\nu$ & $\alpha$ & $\beta_{\text{final}}$ & $\mathbb{E}[Z]$ & $\epsilon$ & Opt. Method & Opt. Bounds & Train & Sim \\
        \midrule
        20  & 10  & 0.05 & 7.54 & 0.80 & 0.00 & $10^{6}$ & {$(Q,S)$, with $\psi=2$} & $-7.50$ & Eq.~\eqref{eq:argminimpulse} & [0.75, 1.25] & 148.61 & 382.04 \\
            & 50  & 0.05 & 8.99 & 0.80 & 0.20 & $10^{6}$ & Can-order                & $-20.00$  & Eq.~\eqref{eq:stationary_conditionZ} & [0.50, 1.50] & 148.44 & 384.54 \\
            & 100 & 0.10 & 6.96 & 0.10 & 0.40 & $10^{6}$ & Can-order                & $-25.00$ & Eq.~\eqref{eq:stationary_conditionZ} & [0.50, 1.50] & 147.72 & 387.45 \\
        100 & 10  & 0.05 & 2.97 & 0.80 & 0.40 & $10^{7}$ & $(Q,S)$                  & $-100.00$ & Eq.~\eqref{eq:argminimpulse}         & [0.75, 1.25] & 148.56 & 361.59 \\
            & 50  & 0.10 & 3.54 & 1.00 & 0.00 & $10^{5}$ & Can-order                & $-2.50$   & Eq.~\eqref{eq:stationary_conditionZ} & [0.50, 1.50] & 148.64 & 362.54 \\
            & 100 & 0.20 & 3.60 & 1.00 & 0.00 & $10^{6}$ & Can-order                & $-2.50$   & Eq.~\eqref{eq:stationary_conditionZ} & [0.50, 1.50] & 149.11 & 363.30 \\
        200 & 10  & 0.05 & 1.99 & 0.80 & 0.40 & $10^{6}$ & $(Q,S)$                  & $-100.00$ & Eq.~\eqref{eq:argminimpulse}         & [0.75, 1.25] & 148.03 & 91.80  \\
            & 50  & 0.10 & 2.79 & 0.10 & 0.20 & $10^{6}$ & Can-order                & $-75.00$  & Eq.~\eqref{eq:stationary_conditionZ} & [0.50, 1.50] & 147.95 & 89.21  \\
            & 100 & 0.20 & 2.54 & 0.80 & 0.00 & $10^{6}$ & Can-order                & $-30.00$  & Eq.~\eqref{eq:stationary_conditionZ} & [0.50, 1.50] & 147.16 & 89.52  \\
        \bottomrule
    \end{tabular}
    }
\end{table}

\begin{table}[H]
    \centering
    \caption{Problem-specific hyperparameters for 50-dimensional test problems ({$\mathbb{E}[Z]$ is set to $\psi$ times the $S$ vector of one of the benchmark policies (with $\psi = 1$ unless noted), which also serves as the optimization starting point;} bounds are scaled by $\mathbb{E}[Z]$ componentwise; training and simulation times reported in minutes).}
    \label{tab:hyperparams-50d}
    \resizebox{0.95\textwidth}{!}{%
    \begin{tabular}{@{}llccccccccccc@{}}
        \toprule
        CV & $c_0$ & $T$ & $\lambda$ & $\nu$ & $\alpha$ & $\beta_{\text{final}}$ & $\mathbb{E}[Z]$ & $\epsilon$ & Opt. Method & Opt. Bounds & Train & Sim \\
        \midrule
        \multirow{3}{*}{Poisson} &  50 & 0.20 & 5.87  & 0.40 & 0.00 & $10^{6}$ & $(Q,S)$ & $-1.25$  & Eq.~\eqref{eq:argminimpulse}         & [0.50, 1.50] & 225.95 & 488.04 \\
                                 & 150 & 0.05 & 3.21  & 0.40 & 0.00 & $10^{5}$ & $(Q,S)$ & $-1.25$  & Eq.~\eqref{eq:argminimpulse}         & [0.50, 1.50] & 226.66 & 474.86 \\
                                 & 250 & 0.20 & 2.56  & 0.40 & 0.00 & $10^{6}$ & $(Q,S)$ & $-5.00$  & Eq.~\eqref{eq:argminimpulse}         & [0.50, 1.50] & 227.04 & 583.81 \\
        \midrule
        \multirow{3}{*}{0.5}     &  50 & 0.20 & 8.91  & 0.40 & 0.00 & $10^{5}$ & $(Q,S)$ & $-2.50$  & Eq.~\eqref{eq:argminimpulse}         & [0.50, 1.50] & 226.51 & 629.24 \\
                                 & 150 & 0.10 & 4.36  & 0.10 & 0.20 & $10^{6}$ & $(Q,S)$ & $-10.00$ & Eq.~\eqref{eq:stationary_conditionZ} & [0.50, 1.50] & 227.30 & 635.07 \\
                                 & 250 & 0.20 & 3.35  & 1.00 & 0.20 & $10^{5}$ & $(Q,S)$ & $-2.50$  & Eq.~\eqref{eq:stationary_conditionZ} & [0.50, 1.50] & 226.91 & 620.31 \\
        \midrule
        \multirow{3}{*}{1.0}     &  50 & 0.10 & 16.52 & 0.20 & 0.00 & $10^{6}$ & $(Q,S)$ & $-10.00$ & Eq.~\eqref{eq:stationary_conditionZ} & [0.50, 1.50] & 226.38 & 630.49 \\
                                 & 150 & 0.05 & 8.45  & 0.20 & 0.00 & $10^{6}$ & $(Q,S)$ & $-15.00$ & Eq.~\eqref{eq:stationary_conditionZ} & [0.50, 1.50] & 227.33 & 634.61 \\
                                 & 250 & 0.05 & 5.88  & 0.20 & 0.00 & $10^{5}$ & $(Q,S)$ & $-15.00$ & Eq.~\eqref{eq:stationary_conditionZ} & [0.50, 1.50] & 227.19 & 627.81 \\
        \bottomrule
    \end{tabular}
    }
\end{table}


\section{Implementation of Benchmarks}\label{app:C}
\subsection{MDP Benchmarks}\label{app:C1}
To obtain the optimal solution for the two-dimensional test problems, we minimize \eqref{eq:MDPobjective1} to obtain the value function ${\tilde V(x)\coloneqq \inf_u \tilde J(x,u)}$, which is known to be characterized by the Bellman equation
\begin{equation}\label{eq:bellman}
    {\tilde V(x) = \min_{y\in\mathbb{N}_0^d} \left\{\sum_{x'\in\Z^d} p(x, x', y) \left(f(x')+c(y) + \gamma \tilde V(x')\right)\right\}},
\end{equation}
where {$\mathbb{N}_0^d$ and $\Z^d$ denote the sets of $d$-dimensional natural numbers and integers, respectively.} Also, $f(x')$ is the per-period inventory cost when the end-of-period inventory is $x'$, and $p(x,x',y)$ is the probability of moving from inventory state $x$ to $x'$ given order vector $y$, determined by the underlying demand distribution.

It is computationally feasible to solve the two-dimensional problems using standard MDP techniques, yielding the benchmarks to which we compare our proposed policy derived from the continuous-time impulse control formulation. To be specific, we solve \eqref{eq:bellman} using policy iteration (see, e.g., \citealp[Section~2.3]{bertsekas2012dynamic}).

To achieve such computational feasibility, it is necessary to truncate the state and action spaces. The following propositions formally justify that we can perform these truncations.

\begin{proposition}\label{prop:mdpequiv}
The minimization in \eqref{eq:bellman} over order quantities $y \in \mathbb{N}_0^d$ is equivalent to a minimization over order-up-to levels $z \in \mathbb{Z}^d$ with $z \ge x$ componentwise.
\end{proposition}

\begin{proof}
Let us denote the state after ordering by $z = x+y$. Since $y \ge 0$, it must hold that $z \ge x$. We can then perform a change of variables in the Bellman equation from $y$ to $z$. The action space becomes choosing an order-up-to level $z \in \{z' \in \mathbb{Z}^d : z' \ge x \}$. The Bellman equation is then rewritten as
\begin{equation}\label{eq:bellman_z}
    {\tilde V(x) = \min_{z \ge x} \left\{ c(z-x) + \mathbb{E}\left[ f(z-\xi) + \gamma  \tilde V(z-\xi) \right] \right\}}.
\end{equation}
It is easily verified that there exists a one-to-one mapping between any feasible order quantity $y$ and an order-up-to level $z$ for each inventory state $x$. Therefore, the minimization problems in \eqref{eq:bellman} and \eqref{eq:bellman_z} are equivalent. Thus, restricting the search to order-up-to policies can be done without loss of optimality.
\end{proof}

Proposition \ref{prop:mdpequiv} yields an alternative Bellman equation that we can solve using policy iteration. 
Recall that \cite{johnson1967optimality} characterizes the optimal policy as the $(\sigma,S)$ policy described earlier. 
This implies that if the initial state $x \le S$ componentwise, then one orders up to the vector $S$ thereafter, since demand is nonnegative. 
Under mild assumptions on the cost parameters (e.g., backlogging costs substantially exceed holding costs), the order-up-to vector is known to be positive. 
In our numerical experiments, we set the initial state to $x=0$, so a single optimal vector $S$ exists. 
The following proposition is essentially an immediate corollary of the results in \cite{johnson1967optimality}, and greatly simplifies the search for the optimal solution.

\begin{proposition}
Suppose that demand is nonnegative, the optimal order-up-to vector $S^*$ is nonnegative, and the initial inventory state $x$ satisfies $x_i\leq S_i$ for each item $i$. Then, there exists a bounded set $\mathcal{Z} \subset \mathbb{N}_0^d$ such that an optimal order-up-to policy orders up to $S^* \in \mathcal{Z}$. Moreover, the state space can be truncated without loss of optimality.
\end{proposition}

\begin{proof}
Under the stated assumptions, it follows from the results in \cite{johnson1967optimality} that under the optimal policy, inventory is raised to a single order-up-to vector $ S^*$ (independent of $x$). It is easily verified that the unbounded increase in costs for large $z$ ensures that the optimal order-up-to level, $S^*$, which minimizes the total cost, must be finite. Note further from our assumptions, that we can lower bound $z$ by zero. Therefore, we can restrict the search for $S^*$ to a sufficiently large compact set $\mathcal{Z} = \{z \in \mathbb{Z}^d: \underline{z}_i \le z_i \le \bar{z}_i, \ i=1,\dots,d\}$ without loss of optimality.

Restricting the action space allows for the truncation of the state space. Under any order-up-to policy with $z\in\mathcal{Z}$, the inventory level immediately after replenishment is $z \in \mathcal{Z}$, so that $x_i \le \bar{z}_i$. Since demand is nonnegative, the inventory level for item $i$ never exceeds $\bar{z}_i$. This means that states where $x_i > \bar{z}_i$ are not reachable under such a policy. It is therefore sufficient to solve the MDP on a truncated state space $\mathcal{X}:=\{x\in\mathbb{Z}^d: x_i \leq \bar{x}_i,\ i=1,\ldots,d\}$, provided the upper bound $\bar{x}$ is chosen to be at least as large as the upper bound $\bar{z}$ of the optimal action space.
\end{proof}

In our two-dimensional test problems, we truncate the action space by considering only order-up-to levels $z \in [0,100]^2$. These bounds were chosen via preliminary computations with the other benchmark policies so that we could identify a region where the optimal policy is likely to operate. The state space is similarly restricted by setting the lower and upper bounds to -250 and 100, respectively.

\subsection{Benchmark Policies}\label{app:C2}

{In this section, we describe in detail the benchmark policies 
that we use for performance comparisons. We evaluate the performances of the benchmark policies under the expected total discounted cost criterion, with the annual interest rate being set 
to $r = 0.05$. In all our experiments, the initial inventory levels are set to $0$ for all items.}

All benchmark policies described here are 
each defined by a set of parameters.
To determine these parameters, 
we are often required to consider inventory problems that are related but different from the original SJRP. For example, 
to set the parameters $s$ and $S$ for independent $(s,S)$ policies, we consider several single-item inventory problems. 
For these related inventory problems, we use the same annual interest rate of $r = 0.05$, and the expected total discounted cost objective. 
We first discuss the independent $(s,S)$ policy, and then review {the $(R,S)$, $(Q,S)$, the {hybrid} $(R,Q,S)$ and can-order policies}. 

\textbf{Independent $(s, S)$ policies.} An independent $(s, S)$ policy computes a reorder point $s_i$ and an order-up-to level $S_i$ for item $i$, $i = 1, \dots, d$. Whenever the inventory level of item $i$ drops to or below $s_i$, the policy places an order to bring the $i$th inventory level to $S_i$. In other words, the policy performs independent replenishment for each item, without considering the opportunities for cost savings through coordinated replenishment. 

We consider a parametrized family of independent $(s,S)$ policies, parametrized by $\alpha \in (0, 1]$. 
For a given $\alpha \in (0, 1]$, the corresponding independent 
$(s,S)$ policy determines the reorder points $s_i$ and 
the order-up-to level $S_i$ as follows. For $i=1,2, \ldots, d$, $s_i$ and $S_i$ correspond to the optimal reorder point 
and order-up-to level, respectively, of a one-item inventory 
model involving only item $i$, with demand process $D_i$ 
defined in Eq.~\eqref{eq:browniandemand}, unit holding cost $h_i$, 
unit backlogging cost $p_i$, unit variable cost $c_i$, 
and fixed ordering cost $\alpha c_0$. {In other words,
the independent $(s,S)$ policies differ only in the fixed cost used to determine the $s_i$ and $S_i$ levels.}

When $\alpha = 1$, we call the corresponding policy the vanilla independent $(s,S)$ policy. The vanilla policy is a common and natural benchmark used in the literature (see, e.g., \citealp{atkins1988periodic,viswanathan1997note}). 
We consider a larger family of policies, 
as $\alpha$ ranges in $(0, 1]$, 
to {potentially} improve performance over the vanilla policy. 
{Policies with different $\alpha$ achieve different tradeoffs 
between the inventory related costs (i.e., holding and backlogging costs) and the fixed ordering costs. Indeed, an independent $(s,S)$ policy with a smaller $\alpha$ places more frequent orders, thereby reducing inventory related costs and increasing the ordering costs.}
Computationally, the parameters $s_i$ and $S_i$ are determined using the approach in \cite{zheng1991finding}.

{We simulate the performances of independent $(s,S)$ policies 
with $\alpha$ taking values in the set $\{0.05,0.1,\ldots,1\}$, 
and report the one with the lowest expected total discounted cost. 
For each performance evaluation,}
we generate 10,000 sample paths, each run for 10,000 periods. Given an annual interest rate of $r=0.05$, the discount factor at termination is less than $\exp(-r \cdot 10000/52) \approx 7\times10^{-5}$, indicating a negligible error resulting from truncation of the time horizon.

\textbf{Periodic review $(R, S)$ policy.} A periodic review $(R, S)$ policy is specified by $d + 1$ parameters: The review period $R$, and for item $i$, $i = 1, \dots, d$, the base-stock level $S_i$. 
{Starting at the beginning of the first week, an order is placed  every $R$ weeks}, which brings the inventory level of item $i$ to $S_i$, $i = 1, \dots, d$. {Let us note that even though a negative basestock level is theoretically possible under certain parameter regimes---for example, when the backlogging cost is much smaller than the holding cost---we do not consider this case and restrict our attention to nonnegative basestock levels. 
Indeed, in all our numerical experiments, backlogging costs are substantially larger than holding costs, as is often the case in practice, so we expect that $(R,S)$ policies with $S \in \R_+^d$ to 
perform competitively. }

{We now explain and describe a computationally efficient
procedure to determine the best $(R,S)$ policy. 
We begin by deriving explicit expressions for the expected total discounted fixed cost of ordering and sum of holding, backlogging, 
and variable ordering costs. 
\begin{proposition}\label{prop:RS}
Under a given $(R,S)$ policy, with $S \geq 0$ component-wise:
\begin{itemize}
\item[\textnormal{(a)}] If $S=0$, the expected total discounted fixed cost of ordering is  
\begin{equation}\label{eq:RS_ordering0}
c_0\pr\left(\sum_{i=1}^d D_i^{(R)}>0\right) \left(\gamma^R + \gamma^{2R} + \cdots \right) = \frac{c_0\gamma^R}{1-\gamma^R} \pr\left(\sum_{i=1}^d D_i^{(R)}>0\right),
\end{equation}
where $c_0$ is fixed ordering cost, 
$\gamma$ is the weekly discount factor, 
and for $i = 1, 2, \ldots, d$ and $r = 1, 2, \ldots, R$, 
$D_i^{(r)}$ has the same distribution as the cumulative demand for item $i$ over $r$ weeks.

If there exists $i$ such that $S_i > 0$, 
the expected total discounted fixed cost of ordering is  
\begin{equation}\label{eq:RS_ordering1}
c_0 + \frac{c_0\gamma^R}{1-\gamma^R} \pr\left(\sum_{i=1}^d D_i^{(R)}>0\right),
\end{equation}
\item[\textnormal{(b)}] For item $i$, the expected total discounted holding,  backlogging and variable ordering cost is given by
\begin{equation}\label{eq:RS_inv}
{c_i S_i + \frac{1}{1-\gamma^R}\sum_{r=0}^{R-1} \gamma^r \E\left[f_i\left(S_i-D_i^{(r+1)}\right)\right] + \frac{\gamma^R}{1-\gamma^R} \cdot c_i \E\left[D_i^{(R)}\right]},
\end{equation}
where $\gamma$ and $D_i(r)$ are as before, $c_i$ is the unit
variable cost for item $i$, and $f_i(\cdot)$ is the inventory 
cost function defined in \textnormal{Eq.~\eqref{eq:inv_cost}}.
\end{itemize}
\end{proposition}
\begin{proof} We first prove part (a). Consider a $(R,S)$ policy with $S = 0$ 
component-wise. Then, no order is placed at time $0$. 
From time $0$ to the beginning of week $R$, item-$i$ inventory 
is reduced by an amount distributed exactly as $D_i^{(R)}$, 
$i=1, 2, \ldots, d$. If the total inventory consumed during this time is positive, which happens with probability 
$\pr\left(\sum_{i=1}^d D_i^{(R)}>0\right)$, an order is placed 
at the beginning of week $R$, incurring a discounted fixed cost of $c_0\gamma^R$. Thus, the expected discounted fixed cost of ordering incurred is 
$\pr\left(\sum_{i=1}^d D_i^{(R)}>0\right)c_0\gamma^R$. 
Similarly, the expected discounted fixed cost of ordering incurred 
at time $2R$ is $\pr\left(\sum_{i=1}^d D_i^{(R)}>0\right)c_0\gamma^{2R}$, and so on and so forth. 
Summing up the expected discounted fixed cost of ordering 
at times $R, 2R, \ldots$ gives the expression \eqref{eq:RS_ordering0}.

If there exists $i$ such that $S_i > 0$, then an order is placed 
at time $0$, incurring an additional cost of $c_0$. 
The subsequent expected discounted fixed costs of ordering incurred 
at times $R, 2R, \ldots$ are exactly the same as in the case with 
$S=0$. Thus, the expected total discounted fixed cost of ordering 
in this case is the sum of $c_0$ and expression \eqref{eq:RS_ordering0}, giving the expression \eqref{eq:RS_ordering1}.

Next, we prove part (b). For item $i$, the discounted variable cost
incurred at time $0$ is $c_i S_i$, and that incurred at time $kR$ 
has distribution $\gamma^{kR} c_i D_i^{(R)}$, $k = 1, 2, \ldots$
Thus, the expected total discounted variable cost of ordering is 
\[
c_i S_i + \sum_{k=1}^{\infty} \gamma^{kR} c_i \E\left[D_i^{(R)}\right] 
= c_i S_i + \frac{\gamma^R}{1-\gamma^R} \cdot c_i \E\left[D_i^{(R)}\right].
\]
The discounted holding and backlogging cost incurred at time $kR+r$, 
where $r = 0, 1, \ldots, R-1$, $k = 0, 1, 2, \ldots$, 
has distribution $\gamma^{kR+r}f_i\left(S_i - D_i^{(r+1)}\right)$. 
Summing over all times and taking expectations, we have that 
the expected total discounted holding and backlogging cost is 
\[
(1+\gamma^R+\cdots) \sum_{r=0}^{R-1} \gamma^r \E\left[f_i\left(S_i-D_i^{(r+1)}\right)\right] = \frac{1}{1-\gamma^R}\sum_{r=0}^{R-1} \gamma^r \E\left[f_i\left(S_i-D_i^{(r+1)}\right)\right].
\]
Summing up the preceding two expected total discounted costs gives 
the expression \eqref{eq:RS_inv}.
\end{proof}
Note that expressions \eqref{eq:RS_ordering0} and \eqref{eq:RS_ordering1} are independent of 
the levels $S_i$, and the expression \eqref{eq:RS_inv} involves 
only parameters and demand distributions associated with item $i$, 
not other items. These properties suggest the following method to 
find the optimal parameters $R^*$ and $S_i^*$. 
First, under a fixed review period $R$, for each $i$, 
we compute the optimal 
item-$i$ base-stock $S_i^*(R)$ by optimizing the expression 
\eqref{eq:RS_inv} over $S_i$, {which becomes}
\begin{equation}\label{eq:optimalbasestock}
S_i^*(R) = \arg\min\limits_{y\geq0} (1-\gamma^R)c_i y + \sum_{r=0}^{R-1} \gamma^r \E[f(y-D_i^{(r+1)})].
\end{equation}
Next, to find the optimal review period $R^*$, we simulate the performances of all $(R, S^*(R))$ policies, where $R \in \{1, 2, \dots, R_{\text{max}}\}$, and $S^*(R)$ is the vector of optimal base-stock levels corresponding to review period $R$. We set $R_{\text{max}}=100$, roughly $2$ years, since we do not expect the best policy to not order for that long (given our current choices for the fixed cost $c_0$).}

{\bf $(Q,S)$ policy.} Similar to the $(R,S)$ policies, a $(Q, S)$ policy is also specified by $d+1$ parameters: The replenishment coordination parameter $Q$ and a base-stock level $S_i$ for each item $i=1,\ldots,d$. Every time the {cumulative aggregate} demand since the previous replenishment reaches or exceeds $Q$, an order is placed to bring the inventory level of item $i$ back to its base-stock level $S_i$, $i=1,\ldots,d$.

{The method for finding the optimal policy parameters 
for $(Q,S)$ policies is similar to the one for $(R,S)$ policies. 
We first fix $Q$ and compute the optimal base-stocks $S^*(Q)$, 
and then simulate performances of $(Q, S^*(Q))$ policies to determine the best $Q^*$ (and associated $S^*(Q^*)$). 
A major difference is that a $(Q,S)$ policy involves a random review period, which we denote by $R_Q$, whereas the review period 
is deterministic under a $(R,S)$ policy. 
More specifically, as previously described, 
$R_Q$ is the number of weeks until the cumulative aggregate demand 
since the last order reaches or exceeds $Q$. Thus $R_Q$ 
has the same distribution as 
\begin{equation}\label{eq:R_Q}
\min\left\{R\in \Z_+: \sum_{i=1}^d D_i^{(R)} \geq Q\right\}, 
\end{equation}
where we recall that $D_i^{(r)}$ is the cumulative demand 
of item $i$ over $r$ weeks.}

We next derive an explicit expression for the expected total discounted cost under an arbitrary $(Q,S)$ policy, 
{which allows us to efficiently compute the optimal base-stock levels $S_i^*(Q)$ for any given $Q$}. It follows from recognizing that the system regenerates at each order epoch, and by applying a suitable renewal-theoretic argument.
\begin{proposition}\label{prop:QS}
Under a given $(Q,S)$ policy and assuming i.i.d.\ demand, with $S \geq 0$ component-wise and an order placed at time $t=0$ to bring the inventory level back to $S$:
\begin{itemize}
    \item[\textnormal{(a)}] If there exists $i$ such that $S_i>0$, then the expected total discounted fixed cost of ordering is given by
    \begin{equation}\label{eq:QS_fixed_cost}
        \frac{c_0}{1 - \E[\gamma^{R_Q}]}.
    \end{equation}
    \item[\textnormal{(b)}] For item $i$, the total expected discounted holding, backlogging, and variable ordering cost is given by
    \begin{equation}\label{eq:QS_item_cost}
    c_i S_i  + \frac{\E\left[\sum_{t=0}^{R_Q-1} \gamma^t f_i(S_i - D_i^{(t+1)})\right] + c_i\E[\gamma^{R_Q}D_i^{(R_Q)}]}{1 - \E[\gamma^{R_Q}]}.
    \end{equation}
\end{itemize}
\end{proposition}
\begin{proof}
The proof relies on a renewal-reward expression for discounted costs. Let $\tilde J(S)$ be the total expected discounted cost starting from the regenerative inventory state $S$, which occurs immediately after an order is placed. By definition, this total cost is given by
$$
\tilde J(S) = \mathbb{E}\Bigl[ \sum_{t=0}^{\infty} \gamma^t C_t \Bigr],
$$
where $C_t$ denotes all costs that are incurred in period $t$.
We next decompose this sum into the costs incurred during the first cycle and all subsequent costs. That is,
$$
\tilde J(S) = \mathbb{E}\Bigl[ \sum_{t=0}^{R_Q-1} \gamma^t C_t +\sum_{t=R_Q}^{\infty} \gamma^t C_t \Bigr] 
= \mathbb{E}\Bigl[ \sum_{t=0}^{R_Q-1} \gamma^t C_t \Bigr] +\mathbb{E} \Bigl[\sum_{t=R_Q}^{\infty} \gamma^t C_t \Bigr].
$$
The first term in the expectation is the total discounted cost within the first cycle, which we denote as $\mathbb{E}[C_{\text{cycle}}]$. For the second term, we can factor out $\gamma^{R_Q}$ so that
$$
\sum_{t=R_Q}^{\infty} \gamma^t C_t = \gamma^{R_Q} \sum_{k=0}^{\infty} \gamma^k C_{R_Q+k}.
$$
It is key to note that at time $R_Q$, the system probabilistically restarts from state $S$ so that from that point forward the total discounted cost is again $\tilde J(S)$, regardless of the history that led to the first replenishment. From the law of total expectation and the fact that the replenishment cycles are of an i.i.d.\ nature, we thus obtain
$$
\begin{aligned}
\tilde J(S) &= \mathbb{E}[C_{\text{cycle}}] + \mathbb{E}\left[ \gamma^{R_Q} \left(\sum_{k=0}^{\infty} \gamma^k C_{R_Q+k}\right) \right] \\
& = \mathbb{E}[C_{\text{cycle}}] + \mathbb{E}\left[\mathbb{E}\left[ \gamma^{R_Q} \left(\sum_{k=0}^{\infty} \gamma^k C_{R_Q+k}\right) \Big| R_Q \right]\right] \\
&= \mathbb{E}[C_{\text{cycle}}] + \mathbb{E}[\gamma^{R_Q}]  \tilde J(S).
\end{aligned}
$$
Solving for $\tilde J(S)$ then yields
$$
\tilde J(S) = \frac{\mathbb{E}[C_{\text{cycle}}]}{1 - \mathbb{E}[\gamma^{R_Q}]}.
$$
This principle allows us to determine the total expected discounted cost by considering a single replenishment cycle.

(a) Note that we incur the fixed cost $c_0$ at the beginning of each cycle (i.e., at times $T_0=0, T_1=R_{Q,1}, T_2=R_{Q,1}+R_{Q,2},$ etc.). We can write the expected present value of this stream of fixed costs as
$$
\E\Bigl[\sum_{k=0}^{\infty} \gamma^{T_k} c_0\Bigr] = c_0 \sum_{k=0}^{\infty} \E[\gamma^{T_k}] = c_0 \sum_{k=0}^{\infty} (\E[\gamma^{R_Q}])^k = \frac{c_0}{1 - \E[\gamma^{R_Q}]},
$$
where exchanging the sum and expectation is justified by nonnegativity of the summands, and the second equality uses that the cycle lengths $R_{Q,k}$ are i.i.d.

(b) For item $i$, the cost consists of the initial order cost and the subsequent costs. The initial variable cost to raise inventory from 0 to $S_i$ is $c_i S_i$. After this, the system is in the regeneration state $S$. The expected discounted cost over a single cycle for item $i$, discounted to the start of the cycle, consists of holding/backlogging costs during the cycle and the variable ordering cost at the end of the cycle, and is thus given by
$$
\E\Bigl[\sum_{t=0}^{R_Q-1} \gamma^t f_i(S_i - D_i^{(t+1)}) + \gamma^{R_Q} c_i D_i^{(R_Q)}\Bigr].
$$
The total expected discounted cost for item $i$ from $t=0+$ onward is obtained by dividing the single-cycle cost by $1 - \E[\gamma^{R_Q}]$. Adding the initial variable cost yields \eqref{eq:QS_item_cost}.
\end{proof}


{By Proposition \ref{prop:QS}, 
given $Q$, the optimal base-stock level $S_i^*(Q)$ can be 
computed by solving 
\begin{equation}\label{eq:optimalbasestock_QS}
S_i^*(Q) = \arg\min\limits_{y\geq0} \E\left[(1-\gamma^{R_Q})c_i y 
+ \sum_{r=0}^{R_Q-1} \gamma^r f_i(y-D_i^{(r+1)})\right],
\end{equation}
for $i=1, 2, \ldots, d$, where the expectation is over the joint distribution of 
$(R_Q, D_i^{(1)}, D_i^{(2)}, \ldots)$. 
For the demand distributions we consider, the random review period $R_Q$ may not have a simple closed-form distribution function. However, we can jointly simulate the review period and the associated demand paths $(R_Q, D_1^{(1)}, \ldots, D_d^{(R_Q)})$ using Eq.~\eqref{eq:R_Q}, and 
solve the optimization problem \eqref{eq:optimalbasestock_QS} via sample average approximation.
We then compute the performances of all $(Q, S^*(Q))$ policies for $Q$ in a judiciously chosen range. Let $\bar{D}$ denote the total expected weekly demand, i.e., 
\begin{equation}
    \bar{D} := \sum_{i=1}^d \E[D_i].
    \label{eq:bar-d}
\end{equation}
Using the optimal review period $R^*$ from the $(R,S)$ policy found earlier, we vary $Q$ over
\begin{equation}
    Q \in \left\{\max\{1, \lfloor (R^*-5)\bar{D} \rfloor\}, \ldots, \lfloor(R^*+5)\bar{D} \rfloor -1, \lfloor(R^*+5)\bar{D} \rfloor \right\}.
\end{equation}
For each $Q$ in this range, we compute the expected total discounted cost using the expressions in Proposition~\ref{prop:QS}, where the expectations are approximated via sample average approximation with 100,000 samples, and select the $Q^*$ that achieves the minimum cost.
}

{
{\bf $(R,Q,S)$ policy.}
The $(R,Q,S)$ policy, proposed by \cite{ozkaya2006stochastic}, 
builds on the $(R, S)$ and $(Q, S)$ policies. 
It is specified by $d+2$ parameters: A maximum review period $R$, a cumulative aggregate demand threshold $Q$, and a base-stock level $S_i$ for each item $i = 1, \ldots, d$. Under this policy, a replenishment order is placed whenever one of two conditions is met, whichever occurs first: (i) The cumulative aggregate demand since the previous replenishment reaches or exceeds $Q$, or (ii) $R$ periods have elapsed since the previous replenishment. At each ordering epoch, the inventory level of each item $i$ is raised to its order-up-to level $S_i$, $i = 1, \ldots, d$. 

The method for finding the optimal policy parameters for $(R,Q,S)$ policies also builds on those for $(R,S)$ and $(Q,S)$ policies. We first fix $(R,Q)$ and compute the optimal base-stocks $S^*(R,Q)$, and then search over $(R,Q)$ pairs to determine the best $(R^*, Q^*)$ and associated $S^*(R^*, Q^*)$. Since orders are placed based on whichever condition is met first, the $(R,Q,S)$ policy involves a random review period that is bounded above by $R$. More specifically, define
\begin{equation}
    \tilde{R} := \min\{R_Q, R\},
\end{equation}
where $R_Q$ is the random review period under the $(Q,S)$ policy, as defined in Eq.~\eqref{eq:R_Q}. Thus, $\tilde{R}$ has the same distribution as
\begin{equation}
    \min\Big\{R, \min\Big\{r \in \mathbb{Z}_+ : \sum_{i=1}^d D_i^{(r)} \geq Q\Big\}\Big\}.
\end{equation}

We next derive an explicit expression for the expected total discounted cost under an arbitrary $(R,Q,S)$ policy. The structure follows from recognizing that the system regenerates at each order epoch, with the cycle length now given by $\tilde{R}$ rather than the deterministic $R$ (as in the $(R,S)$ policy) or the uncapped $R_Q$ (as in the $(Q,S)$ policy).

\begin{proposition}
\label{prop:RQS}
Under a given $(R,Q,S)$ policy and assuming i.i.d.\ demand, with $S \geq 0$ component-wise and an order placed at time $t = 0$ to bring the inventory level back to $S$:
\begin{enumerate}
    \item[(a)] If there exists $i$ such that $S_i > 0$, then the expected total discounted fixed cost of ordering is given by
    \begin{equation}
        \frac{c_0}{1 - E[\gamma^{\tilde{R}}]}.
    \end{equation}
    \item[(b)] For item $i$, the total expected discounted holding, backlogging, and variable ordering cost is given by
    \begin{equation}
        c_i S_i + \frac{\E\left[\sum_{t=0}^{\tilde{R}-1} \gamma^t f_i(S_i - D_i^{(t+1)})\right] + c_i \E[\gamma^{\tilde{R}} D_i^{(\tilde{R})}]}{1 - \E[\gamma^{\tilde{R}}]}.
    \end{equation}
\end{enumerate}
\end{proposition}

\begin{proof}
The proof follows the same renewal-reward argument as for Proposition~\ref{prop:QS}. The system regenerates at each ordering epoch when inventory positions return to $S$. The only modification is that the cycle length is now $\tilde{R} = \min(R_Q, R)$ rather than $R_Q$. Part (a) follows immediately. For part (b), the expected discounted cost over a single cycle for item $i$, discounted to the start of the cycle, consists of holding and backlogging costs during the cycle and the variable ordering cost at the end of the cycle, and is given by
\begin{equation}
    \E\left[\sum_{t=0}^{\tilde{R}-1} \gamma^t f_i(S_i - D_i^{(t+1)}) + \gamma^{\tilde{R}} c_i D_i^{(\tilde{R})}\right].
\end{equation}
The total expected discounted cost for item $i$ after the initial order is obtained by dividing the single-cycle cost by $1 - \E[\gamma^{\tilde{R}}]$. Adding the initial variable cost yields the result.
\end{proof}

Similar to the computation of the base-stock levels $S_i^*(Q)$ 
for the $(Q,S)$ policy under a given $Q$, the optimal 
base-stock level $S_i^*(R,Q)$ for the $(R,Q,S)$ policy can be 
computed by solving 
\begin{equation}\label{eq:optimalbasestock_RQS}
    S_i^*(R,Q) = \arg\min_{y \geq 0} \E\left[(1 - \gamma^{\tilde{R}}) c_i y + \sum_{t=0}^{\tilde{R}-1} \gamma^t f_i(y - D_i^{(t+1)})\right],
\end{equation}
for $i = 1, 2, \ldots, d$, where the expectation is over the joint distribution of $(\tilde{R}, D_i^{(1)}, D_i^{(2)}, \ldots)$. 
To solve the optimization problem \eqref{eq:optimalbasestock_RQS}, 
we jointly simulate the review period and the associated demand paths $(\tilde{R}, D_1^{(1)}, \ldots, D_1^{(\tilde{R})}, \ldots, D_d^{(\tilde{R})})$ and use sample average approximation. 


Having obtained $S_i^*(R,Q)$, $i=1, 2, \ldots, d$, 
for given $(R, Q)$, we then search over $(R,Q)$ pairs to find the optimal policy. Let $Q^*$ denote the optimal parameters from the $(Q,S)$ policy. 
We vary $R$ over
\begin{equation}
    R \in \left\{1,2, \ldots, 100\right\},
\end{equation}
and $Q$ over
\begin{equation}
    Q \in \Big\{\max\{1, \lfloor Q^* - 5\bar{D} \rfloor\}, \ldots, \lceil Q^* + 15\bar{D} \rceil\Big\},
\end{equation}
where $\bar D$ is the total expected weekly demand defined in Eq.~\eqref{eq:bar-d}, and 
with an adaptive step size chosen so that the $Q$ grid contains approximately 200 points. As before, for each $(R,Q)$ pair in this grid, we compute the analytical expected total discounted cost of the $(R,Q,S^*(R,Q)))$ policy using the expressions in Proposition~\ref{prop:RQS}, where the expectations are approximated via sample average approximation with 100,000 samples, and select the pair $(R^*, Q^*)$ that achieves the minimum cost.
}
{\bf Can-order policy.} 
{A can-order policy is specified by $3d$ parameters, 
reorder points $s_i$, can-order points $o_i$, and order-up-to 
levels $S_i$, $i=1, \ldots, d$, with $s_i \leq o_i \leq S_i$ 
for each $i$. Whenever the inventory level of some item, 
say item $i$, drops to or below $s_i$, an order is placed. 
For item $i$, the order raises its inventory level to $S_i$. 
For item $j$ with $j\neq i$, if its inventory level is 
at or below $o_j$, the same order replenishes item $j$
to raise the inventory level to $S_j$.}

Finding the best can-order policy involves optimizing over $3d$ parameters {\em a priori}, which is computationally infeasible for large $d$.
Therefore, we proceed with a heuristic approach, 
which involves optimization over two parameters, $\alpha$ and $\omega$, {both of which take values in $[0,1]$. The parameter $\alpha$ is used to determine 
the reorder points $s_i$ and the order-up-to levels $S_i$, and 
$\omega$ is then used to determine the can-order levels 
$o_i$. More specifically, for a given $\alpha$, we first determine the values of $s_i$ and $S_i$ according to an independent $(s,S)$
policy with parameter $\alpha$, as previously described.} 
Then, {given $\omega$}, we set the can-order levels $o_i$ according to 
$$
o_i = \omega s_i + (1-\omega) S_i, \quad i=1,\ldots,d.
$$
We simulate the performances of can-order policies for all pairs 
$(\alpha, \omega)$ in the set $\{0,0.05,\ldots,1\}\times\{0,0.1,\ldots,1.0\}$, and report the best performing policy. 
Each performance evaluation is obtained using 10,000 sample paths, 
each run for 10,000 periods. 

\subsection{Benchmark Results}

\newpage

\begin{table}[H]
    \centering
    \caption{\footnotesize Benchmark policy costs for $d=12$ with Poisson demand and $h=2.0$ (standard errors in absolute terms).}\label{table:12dlow_benchmarks}
    \resizebox{0.95\textwidth}{!}
    {%
    \footnotesize
    \begin{tabular}{@{}llrrrrr@{}}
    \toprule
    \multicolumn{2}{c}{} & \multicolumn{5}{c}{Cost of Benchmark Policies} \\ 
    \cmidrule(lr){3-7}
    $c_0$ & $p$ & $(R,S)$ & $(Q,S)$ & $(R,Q,S)$ & Can-Order & Ind.\ $(s,S)$ \\
    \midrule
    \multirow{3}{*}{20} & 10  & $6286.39 \pm 0.42$ & $6231.29 \pm 0.42$ & $6231.36 \pm 0.42$ & $6890.77 \pm 0.52$ & $12225.46 \pm 0.97$ \\
                        & 50  & $7396.61 \pm 0.63$ & $7290.59 \pm 0.57$ & $7290.98 \pm 0.58$ & $7715.51 \pm 0.54$ & $12951.65 \pm 1.02$ \\
                        & 100 & $7814.15 \pm 0.77$ & $7685.13 \pm 0.66$ & $7685.67 \pm 0.68$ & $7798.01 \pm 0.53$ & $13140.32 \pm 1.03$ \\
    \midrule
    \multirow{3}{*}{100} & 10  & $10226.74 \pm 0.66$ & $10157.61 \pm 0.65$ & $10157.63 \pm 0.65$ & $11744.19 \pm 1.00$ & $25319.96 \pm 2.31$ \\
                         & 50  & $12003.00 \pm 0.98$ & $11862.07 \pm 0.98$ & $11858.91 \pm 0.98$ & $12645.90 \pm 1.12$ & $26926.53 \pm 2.41$ \\
                         & 100 & $12641.75 \pm 1.28$ & $12466.51 \pm 1.23$ & $12466.56 \pm 1.23$ & $13045.77 \pm 1.06$ & $27018.78 \pm 2.43$ \\
    \midrule
    \multirow{3}{*}{200} & 10  & $13151.39 \pm 0.84$ & $13076.50 \pm 0.82$ & $13076.52 \pm 0.82$ & $14964.50 \pm 1.38$ & $35301.18 \pm 3.58$ \\
                         & 50  & $15344.44 \pm 1.31$ & $15188.84 \pm 1.27$ & $15188.54 \pm 1.28$ & $16450.56 \pm 1.54$ & $37542.41 \pm 3.92$ \\
                         & 100 & $16122.94 \pm 1.76$ & $15927.10 \pm 1.58$ & $15927.11 \pm 1.58$ & $16955.14 \pm 1.47$ & $37524.42 \pm 3.74$ \\
    \bottomrule
    \end{tabular}
    }
\end{table}

\begin{table}[H]
    \centering
    \caption{\footnotesize Optimal policy parameters for $d=12$ with Poisson demand and $h=2.0$.}\label{table:12dlow_params}
    \footnotesize
    \begin{tabular}{@{}llcccccccc@{}}
    \toprule
    \multicolumn{2}{c}{} & $(R,S)$ & $(Q,S)$ & \multicolumn{2}{c}{$(R,Q,S)$} & Ind.\ $(s,S)$ & \multicolumn{2}{c}{Can-Order} \\ 
    \cmidrule(lr){3-3} \cmidrule(lr){4-4} \cmidrule(lr){5-6} \cmidrule(lr){7-7} \cmidrule(lr){8-9}
    $c_0$ & $p$ & $R^*$ & $Q^*$ & $R^*$ & $Q^*$ & $\alpha^*$ & $\alpha^*$ & $\omega^*$ \\
    \midrule
    \multirow{3}{*}{20} & 10  & 13 & 84  & 17 & 84  & 0.80 & 0.15 & 0.0 \\
                        & 50  & 11 & 69  & 14 & 70  & 0.75 & 0.15 & 0.0 \\
                        & 100 & 11 & 69  & 14 & 70  & 0.75 & 0.20 & 0.2 \\
    \midrule
    \multirow{3}{*}{100} & 10  & 30 & 195 & 37 & 195 & 0.90 & 0.20 & 0.3 \\
                         & 50  & 25 & 164 & 32 & 166 & 0.85 & 0.10 & 0.1 \\
                         & 100 & 25 & 162 & 31 & 162 & 0.85 & 0.15 & 0.2 \\
    \midrule
    \multirow{3}{*}{200} & 10  & 43 & 281 & 52 & 281 & 0.95 & 0.15 & 0.2 \\
                         & 50  & 38 & 251 & 44 & 248 & 0.85 & 0.10 & 0.1 \\
                         & 100 & 37 & 242 & 47 & 242 & 0.85 & 0.15 & 0.2 \\
    \bottomrule
    \end{tabular}
\end{table}

\begin{table}[H]
    \centering
    \caption{\footnotesize Benchmark policy costs for $d=12$ with negative binomial demand (CV $= 0.5$) and $h=2.0$ (standard errors in absolute terms).}\label{table:12dmed_benchmarks}
    \resizebox{0.95\textwidth}{!}
    {%
    \footnotesize
    \begin{tabular}{@{}llrrrrr@{}}
    \toprule
    \multicolumn{2}{c}{} & \multicolumn{5}{c}{Cost of Benchmark Policies} \\ 
    \cmidrule(lr){3-7}
    $c_0$ & $p$ & $(R,S)$ & $(Q,S)$ & $(R,Q,S)$ & Can-Order & Ind.\ $(s,S)$ \\
    \midrule
    \multirow{3}{*}{20} & 10  & $8763.03 \pm 1.61$ & $8183.51 \pm 1.46$ & $8183.52 \pm 1.46$ & $9420.05 \pm 1.74$ & $12246.64 \pm 1.86$ \\
                        & 50  & $12793.22 \pm 3.08$ & $11313.54 \pm 2.28$ & $11313.56 \pm 2.28$ & $10718.65 \pm 2.06$ & $13764.24 \pm 2.16$ \\
                        & 100 & $14673.49 \pm 4.17$ & $12727.43 \pm 2.72$ & $12727.46 \pm 2.72$ & $11867.47 \pm 2.31$ & $14810.35 \pm 2.52$ \\
    \midrule
    \multirow{3}{*}{100} & 10  & $13710.86 \pm 2.75$ & $13166.34 \pm 2.50$ & $13166.32 \pm 2.50$ & $15658.73 \pm 3.27$ & $25409.71 \pm 3.89$ \\
                         & 50  & $19032.25 \pm 5.26$ & $17725.43 \pm 4.29$ & $17725.42 \pm 4.29$ & $17511.04 \pm 3.63$ & $27469.34 \pm 4.17$ \\
                         & 100 & $21343.34 \pm 7.04$ & $19614.70 \pm 5.38$ & $19630.47 \pm 5.25$ & $18041.22 \pm 3.66$ & $28089.78 \pm 4.34$ \\
    \midrule
    \multirow{3}{*}{200} & 10  & $17064.28 \pm 3.44$ & $16528.51 \pm 3.24$ & $16528.60 \pm 3.24$ & $19950.26 \pm 4.60$ & $35468.63 \pm 5.47$ \\
                         & 50  & $23174.09 \pm 6.70$ & $21918.53 \pm 5.71$ & $21918.62 \pm 5.71$ & $22202.04 \pm 4.77$ & $38213.58 \pm 5.98$ \\
                         & 100 & $25757.59 \pm 9.19$ & $24113.84 \pm 7.35$ & $24113.88 \pm 7.35$ & $22498.78 \pm 5.06$ & $38704.55 \pm 6.08$ \\
    \bottomrule
    \end{tabular}
    }
\end{table}

\begin{table}[H]
    \centering
    \caption{\footnotesize Optimal policy parameters for $d=12$ with negative binomial demand (CV $= 0.5$) and $h=2.0$.}\label{table:12dmed_params}
    \footnotesize
    \begin{tabular}{@{}llcccccccc@{}}
    \toprule
    \multicolumn{2}{c}{} & $(R,S)$ & $(Q,S)$ & \multicolumn{2}{c}{$(R,Q,S)$} & Ind.\ $(s,S)$ & \multicolumn{2}{c}{Can-Order} \\ 
    \cmidrule(lr){3-3} \cmidrule(lr){4-4} \cmidrule(lr){5-6} \cmidrule(lr){7-7} \cmidrule(lr){8-9}
    $c_0$ & $p$ & $R^*$ & $Q^*$ & $R^*$ & $Q^*$ & $\alpha^*$ & $\alpha^*$ & $\omega^*$ \\
    \midrule
    \multirow{3}{*}{20} & 10  &  9 &  53 & 25 &  53 & 0.75 & 0.45 & 0.3 \\
                        & 50  &  7 &  42 & 21 &  42 & 0.80 & 0.30 & 0.0 \\
                        & 100 &  7 &  36 & 19 &  36 & 0.70 & 0.25 & 0.0 \\
    \midrule
    \multirow{3}{*}{100} & 10  & 26 & 162 & 49 & 162 & 0.90 & 0.30 & 0.2 \\
                         & 50  & 20 & 127 & 39 & 127 & 0.85 & 0.25 & 0.0 \\
                         & 100 & 19 & 115 & 30 & 110 & 0.85 & 0.25 & 0.1 \\
    \midrule
    \multirow{3}{*}{200} & 10  & 38 & 244 & 63 & 244 & 0.90 & 0.25 & 0.2 \\
                         & 50  & 30 & 200 & 55 & 200 & 0.90 & 0.25 & 0.1 \\
                         & 100 & 29 & 180 & 52 & 180 & 0.90 & 0.20 & 0.0 \\
    \bottomrule
    \end{tabular}
\end{table}

\begin{table}[H]
    \centering
    \caption{\footnotesize Benchmark policy costs for $d=12$ with negative binomial demand (CV $= 1.0$) and $h=2.0$ (standard errors in absolute terms).}\label{table:12dhigh_benchmarks}
    \resizebox{0.95\textwidth}{!}
    {%
    \footnotesize
    \begin{tabular}{@{}llrrrrr@{}}
    \toprule
    \multicolumn{2}{c}{} & \multicolumn{5}{c}{Cost of Benchmark Policies} \\ 
    \cmidrule(lr){3-7}
    $c_0$ & $p$ & $(R,S)$ & $(Q,S)$ & $(R,Q,S)$ & Can-Order & Ind.\ $(s,S)$ \\
    \midrule
    \multirow{3}{*}{20} & 10  & $10364.38 \pm 3.40$ & $8422.10 \pm 2.67$ & $8422.12 \pm 2.67$ & $10077.73 \pm 3.09$ & $11107.61 \pm 3.08$ \\
                        & 50  & $19031.98 \pm 6.66$ & $13737.11 \pm 4.75$ & $13737.37 \pm 4.75$ & $13544.73 \pm 4.71$ & $14507.90 \pm 4.73$ \\
                        & 100 & $23948.00 \pm 8.45$ & $17217.75 \pm 6.63$ & $17217.82 \pm 6.63$ & $16383.79 \pm 6.39$ & $17711.01 \pm 6.70$ \\
    \midrule
    \multirow{3}{*}{100} & 10  & $18187.28 \pm 6.11$ & $15624.37 \pm 4.85$ & $15624.40 \pm 4.85$ & $19305.83 \pm 6.50$ & $24399.29 \pm 6.51$ \\
                         & 50  & $31166.61 \pm 13.67$ & $23454.24 \pm 7.86$ & $23454.32 \pm 7.86$ & $22733.11 \pm 7.25$ & $27651.66 \pm 7.69$ \\
                         & 100 & $37805.49 \pm 19.20$ & $27111.44 \pm 9.36$ & $27111.86 \pm 9.36$ & $24340.99 \pm 8.57$ & $29660.22 \pm 8.84$ \\
    \midrule
    \multirow{3}{*}{200} & 10  & $23162.50 \pm 8.36$ & $20554.59 \pm 6.71$ & $20554.58 \pm 6.71$ & $25345.36 \pm 9.32$ & $34607.79 \pm 9.34$ \\
                         & 50  & $37956.77 \pm 18.20$ & $30360.59 \pm 11.00$ & $30360.59 \pm 11.00$ & $29112.46 \pm 10.45$ & $38394.40 \pm 10.57$ \\
                         & 100 & $45189.33 \pm 24.83$ & $34621.62 \pm 13.20$ & $34621.78 \pm 13.20$ & $30493.69 \pm 10.94$ & $40065.95 \pm 11.40$ \\
    \bottomrule
    \end{tabular}
    }
\end{table}

\begin{table}[H]
    \centering
    \caption{\footnotesize Optimal policy parameters for $d=12$ with negative binomial demand (CV $= 1.0$) and $h=2.0$.}\label{table:12dhigh_params}
    \footnotesize
    \begin{tabular}{@{}llcccccccc@{}}
    \toprule
    \multicolumn{2}{c}{} & $(R,S)$ & $(Q,S)$ & \multicolumn{2}{c}{$(R,Q,S)$} & Ind.\ $(s,S)$ & \multicolumn{2}{c}{Can-Order} \\ 
    \cmidrule(lr){3-3} \cmidrule(lr){4-4} \cmidrule(lr){5-6} \cmidrule(lr){7-7} \cmidrule(lr){8-9}
    $c_0$ & $p$ & $R^*$ & $Q^*$ & $R^*$ & $Q^*$ & $\alpha^*$ & $\alpha^*$ & $\omega^*$ \\
    \midrule
    \multirow{3}{*}{20} & 10  &  6 &  28 & 33 &  28 & 0.75 & 0.75 & 0.4 \\
                        & 50  &  3 &  20 & 26 &  20 & 0.80 & 0.55 & 0.0 \\
                        & 100 &  2 &  21 & 29 &  21 & 0.65 & 0.50 & 0.1 \\
    \midrule
    \multirow{3}{*}{100} & 10  & 16 &  97 & 62 &  97 & 0.90 & 0.55 & 0.2 \\
                         & 50  & 12 &  62 & 48 &  62 & 0.85 & 0.60 & 0.2 \\
                         & 100 & 11 &  56 & 42 &  56 & 0.85 & 0.45 & 0.0 \\
    \midrule
    \multirow{3}{*}{200} & 10  & 27 & 154 & 85 & 154 & 0.90 & 0.45 & 0.2 \\
                         & 50  & 20 & 105 & 71 & 105 & 0.85 & 0.40 & 0.0 \\
                         & 100 & 18 &  97 & 62 &  97 & 0.85 & 0.40 & 0.0 \\
    \bottomrule
    \end{tabular}
\end{table}

\begin{table}[H]
    \centering
    \caption{\footnotesize Benchmark policy costs for $d=50$ (standard errors in absolute terms).}\label{table:50d_benchmarks}
    \resizebox{0.95\textwidth}{!}
    {%
    \footnotesize
    \begin{tabular}{@{}llrrrrr@{}}
    \toprule
    \multicolumn{2}{c}{} & \multicolumn{5}{c}{Cost of Benchmark Policies} \\ 
    \cmidrule(lr){3-7}
    CV & $c_0$ & $(R,S)$ & $(Q,S)$ & $(R,Q,S)$ & Can-Order & Ind.\ $(s,S)$ \\
    \midrule
    \multirow{3}{*}{Poisson} 
        &  50 & $22065.98 \pm 1.07$ & $21977.50 \pm 1.00$ & $21975.75 \pm 0.99$ & $27351.58 \pm 1.85$ & $60972.63 \pm 3.92$ \\
        & 150 & $30771.78 \pm 1.46$ & $30681.17 \pm 1.49$ & $30683.31 \pm 1.44$ & $35449.80 \pm 2.79$ & $111031.35 \pm 8.93$ \\
        & 250 & $36567.00 \pm 1.73$ & $36449.72 \pm 1.83$ & $36449.11 \pm 1.84$ & $43682.24 \pm 3.94$ & $144903.04 \pm 13.29$ \\
    \midrule
    \multirow{3}{*}{0.5} 
        &  50 & $38893.13 \pm 4.33$ & $37686.67 \pm 3.93$ & $37697.07 \pm 4.01$ & $41624.19 \pm 4.64$ & $63620.70 \pm 5.01$ \\
        & 150 & $51055.26 \pm 6.36$ & $50038.41 \pm 5.96$ & $50031.13 \pm 5.88$ & $58295.25 \pm 7.02$ & $114526.42 \pm 10.37$ \\
        & 250 & $58619.23 \pm 7.58$ & $57671.54 \pm 7.29$ & $57671.69 \pm 7.24$ & $67810.36 \pm 9.95$ & $150188.01 \pm 14.55$ \\
    \midrule
    \multirow{3}{*}{1.0} 
        &  50 & $60063.26 \pm 10.35$ & $53205.38 \pm 8.68$ & $53205.39 \pm 8.68$ & $59257.45 \pm 9.19$ & $65854.80 \pm 8.50$ \\
        & 150 & $84009.31 \pm 14.81$ & $76252.62 \pm 13.14$ & $76257.13 \pm 13.16$ & $89009.34 \pm 15.04$ & $115911.52 \pm 15.17$ \\
        & 250 & $98012.81 \pm 18.87$ & $90460.97 \pm 16.49$ & $90470.66 \pm 16.49$ & $106554.54 \pm 20.40$ & $152478.11 \pm 20.29$ \\
    \bottomrule
    \end{tabular}
    }
\end{table}

\begin{table}[H]
    \centering
    \caption{\footnotesize Optimal policy parameters for $d=50$.}\label{table:50d_params}
    \footnotesize
    \begin{tabular}{@{}llcccccccc@{}}
    \toprule
    \multicolumn{2}{c}{} & $(R,S)$ & $(Q,S)$ & \multicolumn{2}{c}{$(R,Q,S)$} & Ind.\ $(s,S)$ & \multicolumn{2}{c}{Can-Order} \\ 
    \cmidrule(lr){3-3} \cmidrule(lr){4-4} \cmidrule(lr){5-6} \cmidrule(lr){7-7} \cmidrule(lr){8-9}
    CV & $c_0$ & $R^*$ & $Q^*$ & $R^*$ & $Q^*$ & $\alpha^*$ & $\alpha^*$ & $\omega^*$ \\
    \midrule
    \multirow{3}{*}{Poisson} 
        &  50 &  9 & 221 & 10 & 220 & 0.35 & 0.05 & 0.2 \\
        & 150 & 16 & 416 & 18 & 411 & 0.60 & 0.05 & 0.1 \\
        & 250 & 20 & 524 & 25 & 527 & 0.60 & 0.05 & 0.1 \\
    \midrule
    \multirow{3}{*}{0.5} 
        &  50 &  6 & 137 & 11 & 142 & 0.35 & 0.15 & 0.0 \\
        & 150 & 12 & 298 & 20 & 293 & 0.60 & 0.15 & 0.1 \\
        & 250 & 16 & 394 & 21 & 393 & 0.60 & 0.10 & 0.0 \\
    \midrule
    \multirow{3}{*}{1.0} 
        &  50 &  3 &  53 & 13 &  53 & 0.40 & 0.30 & 0.0 \\
        & 150 &  6 & 132 & 20 & 133 & 0.60 & 0.40 & 0.2 \\
        & 250 &  9 & 203 & 21 & 204 & 0.70 & 0.30 & 0.0 \\
    \bottomrule
    \end{tabular}
\end{table}

\section{Validation of Methodology in 1-Dimensional Setting with Diffusion Demand}\label{app:D}
{To validate our computational method, we consider the single-item inventory control problem studied in \cite{sulem1986solvable}. In this setting, replenishment is instantaneous, inventory levels are continuous and continuously reviewed, and cumulative demand follows the Brownian motion}
$$
D(t) = \mu t + \sigma B(t).
$$
We treat the infinite-horizon, discounted cost problem in one dimension. 
The evolution of the inventory level is described by the one-dimensional stochastic process $X^{u}(\cdot)$, with 
\begin{equation}
X^{u}(t) = x - D(t) + \sum_{\{j : \tau_j \leq t\}} y_j, 
\end{equation}
where $x$ is the initial inventory level, and $y_j$ is the amount ordered at time $\tau_j$. 
The goal is to minimize the expected total discounted cost, where costs are continuously discounted over time.
Putting this all together, the total expected discounted cost is given by
\begin{equation}\label{eq:discountedinventorycont}
{J(x, u) = \E_x\Big[\int_0^\infty e^{-r t} f(X^u(t)) dt + \sum_j e^{-r \tau_j} c(y_j)\Big]},
\end{equation}
where the cost functions $f(\cdot)$ and $c(\cdot)$ are as defined in Section~\ref{sec:mathmodel}.

The problem parameters for the one-dimensional test problem are listed in Table~\ref{tab:lowtestbed}. 
The hyperparameter configuration for this example was determined empirically. 
The terminal time is $T=0.8$, and the number of time intervals is $N=200$. 
The batch size is $K=1250$. For the neural-network architecture, the hidden layers have width 250, 
and we use four hidden layers with ELU activation functions in this one-dimensional example. 
During training, we use learning rates $10^{-3}$, $10^{-4}$, and $10^{-5}$ consecutively, each for 5000 iterations, 
to achieve high accuracy. The penalties are $10^2$, $10^4$, and $2\cdot10^5$, also for 5000 iterations each, 
and the cost scaling parameter is set to $\kappa=1.0$, {i.e., no scaling}. 
The primitives of the reference policy are $\lambda=0.4$ {and $\Phi$ is lognormal with mean 2.0 and variance 1.0.} 
Figure~\ref{fig:value-gradient} shows that the neural-network approximations determined via our numerical procedure 
are close to the exact solutions for the optimal value function and its derivative in the one-dimensional setting. 
{The exact solutions are computed using the expressions in \cite{sulem1986solvable}.}

\begin{table}[H]
\centering
\caption{Problem parameters for one-dimensional example with diffusion demand}
\label{tab:lowtestbed}
\begin{tabular}{ccccccc}
\toprule
$r$ & $p$ & $h$ & $c$ & $c_0$ & $\mu$ & $\sigma$ \\
\midrule
0.05 & 2 & 0.5 & 1 & 1.5 & 1 & 0.2  \\
\bottomrule
\end{tabular}
\end{table}

\begin{figure}[H]
  \centering
  \begin{subfigure}[t]{0.49\linewidth}
    \centering
    \includegraphics[width=\linewidth]{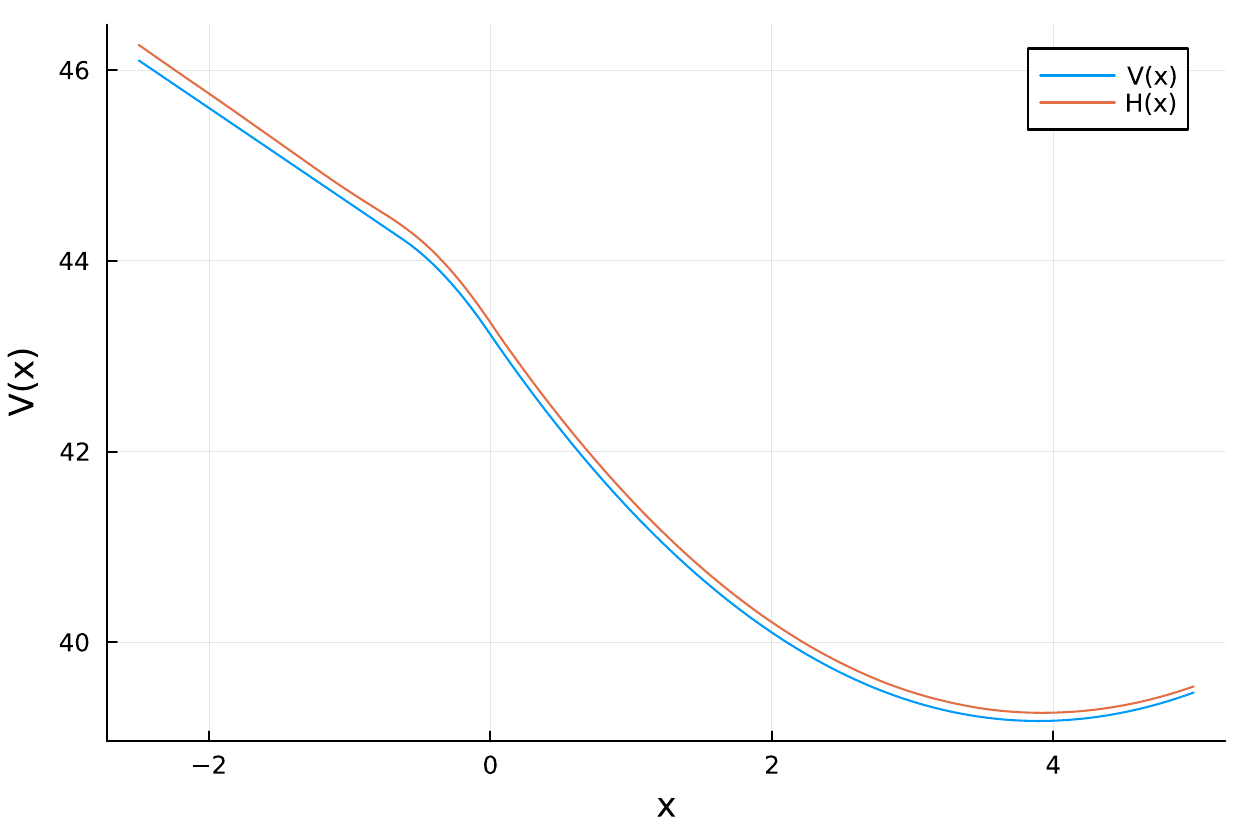}%
    \caption{$V(x)$ and $H_\theta(x)$}
    \label{fig:value}
  \end{subfigure}\hfill
  \begin{subfigure}[t]{0.49\linewidth}
    \centering
    \includegraphics[width=\linewidth]{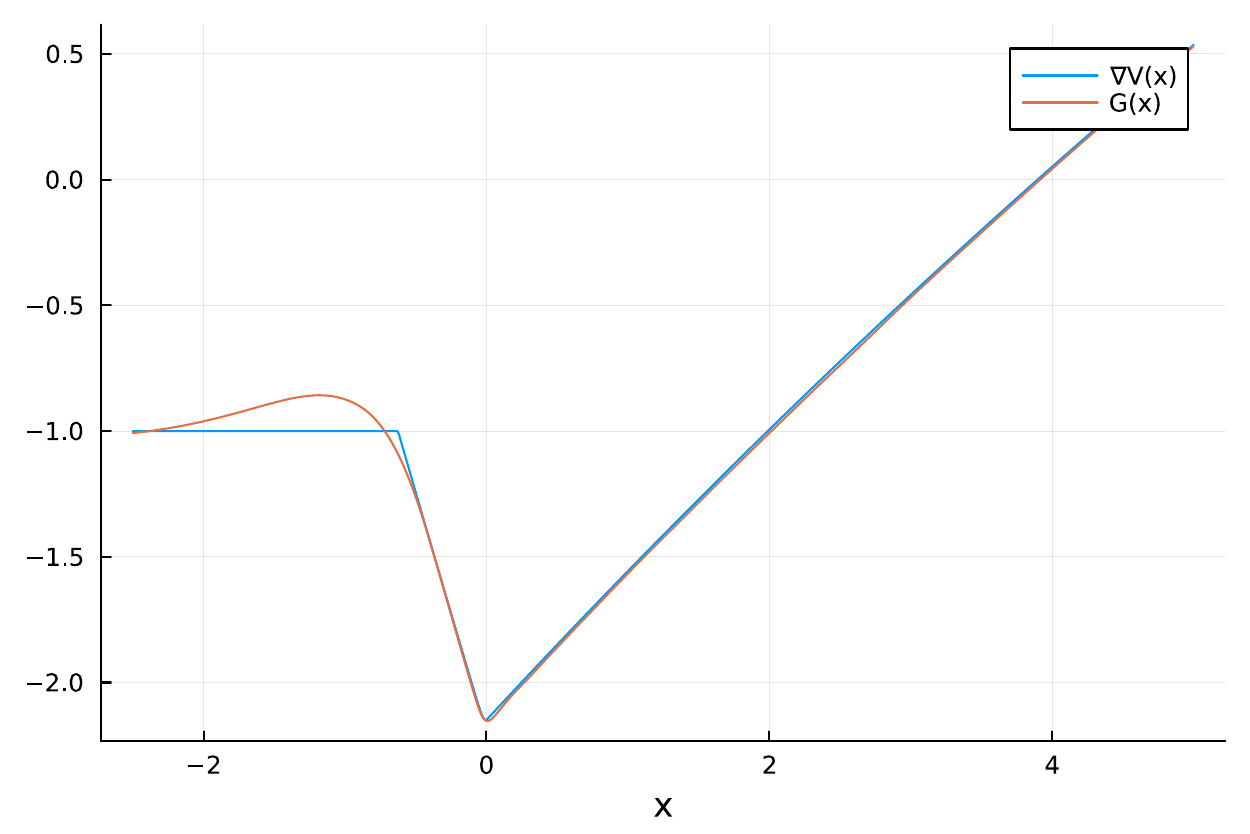}%
    \caption{$\nabla V(x)$ and $G_\vartheta(x)$}
    \label{fig:grad}
  \end{subfigure}
  \caption{Value function and gradient computed using the expressions in \cite{sulem1986solvable}, and the neural network approximations}
  \label{fig:value-gradient}
\end{figure}

\end{appendix}

\end{document}